\documentclass[11pt,a4paper]{article}

\usepackage{a4wide}
\usepackage{amsmath,amsthm,amssymb}
\usepackage{enumerate}
\usepackage{xparse}
\usepackage{filecontents} 

\usepackage{tikz}
\usetikzlibrary{calc}
\usetikzlibrary{decorations.markings}

\newcommand{\E}{\mathbb{E}}
\renewcommand{\Pr}{\mathbb{P}}
\NewDocumentCommand{\MM}{gg}{
	\IfNoValueTF{#2}
		{\IfNoValueTF{#1}
			{M}
			{M(#1)}}
		{M^{#1}(#2)}
}

\newcommand{\ceil}[1]{\lceil #1 \rceil}
\newcommand{\floor}[1]{\lfloor #1 \rfloor}
\newcommand{\dr}[1]{\textnormal{Dr}(#1)}

\newcommand{\reals}{\mathbb{R}}

\newcommand{\nlevs}{h}

\newcommand{\dbar}{\overline{d}}

\newcommand{\calE}{\mathcal{E}}
\newcommand{\calG}{\mathcal{G}}

\newcommand{\calS}{\mathcal{S}}

\newcommand{\eps}{\varepsilon}
\renewcommand{\epsilon}{\eps}

\newcommand{\phif}{\phi_{\!f}}
\newcommand{\psif}{\psi_{\!f}}

\newcommand{\pfix}[1]{f_{#1}}

\newtheorem{theorem}{Theorem}
\newtheorem{lemma}[theorem]{Lemma}
\newtheorem{corollary}[theorem]{Corollary}
\theoremstyle{definition}
\newtheorem{definition}[theorem]{Definition}
\newtheorem{observation}[theorem]{Observation}
\newtheorem{remark}[theorem]{Remark}

\newcommand{\tagme}{\stepcounter{equation}\tag{\theequation}}

\newtheorem*{rep@theorem}{\rep@title}
\newcommand{\newreptheorem}[2]{%
	\newenvironment{rep#1}[1]{%
		\def\rep@title{#2 \ref{##1} (restated)}%
		\begin{rep@theorem}}%
		{\end{rep@theorem}}}
\newreptheorem{theorem}{Theorem}
\newreptheorem{lemma}{Lemma}

\title {Phase Transitions of the Moran Process and Algorithmic~Consequences\thanks{The research leading to these results has received funding from the European Research Council under the European Union's Seventh Framework Programme (FP7/2007--2013) ERC grant agreement no.\ 334828. The paper reflects only the authors' views and not the views of the ERC or the European Commission. The European Union is not liable for any use that may be made of the information contained therein.}}
\author{Leslie Ann Goldberg, John Lapinskas and David Richerby\thanks{University of Oxford, UK. \texttt{\{leslie.goldberg, john.lapinskas, david.richerby\}@cs.ox.ac.uk}.}}
\date{}

\begin{document}
\maketitle{}

\begin{abstract}
The Moran process is a random process that models the spread of
genetic mutations through graphs.  On connected graphs, the
process eventually reaches ``fixation'', where all vertices are
mutants, or ``extinction'', where none are.

Our main result is an almost-tight upper bound on expected absorption
time.  For all $\epsilon>0$, we show that the expected absorption time
on an $n$-vertex graph is $o(n^{3+\epsilon})$.  Specifically, it is at
most $n^3e^{O((\log\log{}n)^3)}$, and there is a family of graphs
where it is $\Omega(n^3)$.  In proving this, we establish a phase
transition in the probability of fixation, depending on mutants'
fitness~$r$.  We show that no similar phase transition occurs for
digraphs, where it is already known that the expected absorption time
can be exponential.  Finally, we give an improved FPRAS for
approximating the probability of fixation.  On degree-bounded graphs
where some basic properties are given, its running time is independent
of the number of vertices.

\end{abstract}

\section{Introduction}

The Moran process~\cite{Mor1958}, as generalised by Lieberman, Hauert
and Nowak~\cite{LHN2005}, is a random process that models the
spread of genetic mutations through spatially structured populations.
Similar models have been applied to the study of epidemics, voting preferences, monopolies, ideas in social networks, and neural development~\cite{Ber2001:Monopolies,KKT2015:Influence,Lig1999:IntSys,Turney:Neuro}.
The individuals in the population are represented as the vertices of a
graph, which may be directed or undirected.  Initially, a single
vertex is chosen uniformly at random to possess some mutation.
Individuals that have the mutation are known as ``mutants'' and have
fitness given by some positive real number~$r$; individuals that do
not possess the mutation are ``non-mutants'' and have fitness~$1$.
Generally the size of the graph is assumed to be large relative to~$r$, which is seen as fixed.

At each step of the process, a vertex~$x$ is chosen with probability
proportional to its fitness.  A neighbour~$y$ of~$x$ (out-neighbour,
if the graph is directed) is then chosen uniformly at random and the
mutant/non-mutant status of~$x$ is copied to~$y$. We call the state in which every vertex is a mutant ``fixation'', and the state in which no vertex is a mutant ``extinction''; note that each of these states is absorbing.
The graph is generally assumed to
be connected (if it is undirected) or strongly connected (if it is directed)
which ensures that with probability~$1$, one of these states is reached. The principal
quantities of interest are the fixation probability, which is the probability
of reaching the all-mutant state; and the expected absorption time,
which is the expected number of steps the process runs before
terminating at fixation or extinction.

\subsection{Absorption time}

Our main result is an upper bound on the expected absorption time
of the Moran process with any positive fitness $r \ne 1$, which is almost tight on the family of all undirected graphs.
It is known that there are classes of digraphs on which the
expected absorption time is exponential in the number of
vertices~\cite{MoranAbsorb}, so we focus on undirected graphs.
D\'iaz \emph{et al.}~\cite{DGMRSS2014} showed that the expected absorption time on an $n$-vertex undirected graph
is $O(n^4)$, but noted that they were unaware of any class of graphs 
with an absorption time of $\omega(n^3)$.  We almost
close this gap.

\begin{theorem}\label{thm:abs-time-intro}
	For all positive $r \ne 1$ and all $\epsilon>0$, the expected absorption time
	of the Moran process with fitness~$r$ on a connected undirected $n$-vertex
	graph is $o(n^{3+\epsilon})$.
\end{theorem}

More specifically, we show that, for all positive $r \ne 1$, there is a real
number~$C_r$ such that the expected absorption time is at most $n^3
e^{C_r(\log \log n)^3}$. (See Theorem~\ref{thm:tabs-detail} in
Section~\ref{sec:abs-time}.)
Theorem~\ref{thm:abs-time-intro} is almost tight: we show that a
family of graphs known as ``double stars'', consisting of two disjoint
equal-sized stars with an edge between their centres, have expected
absorption time $\Omega(n^3)$ (see Theorem~\ref{thm:dubstar} in Section~\ref{sec:dubstar}).

To prove Theorem~\ref{thm:abs-time-intro}, first note that the $O(n^4)$ bound of~\cite{DGMRSS2014} is derived using a lower bound on the expected change, at any given time step, in the potential function $\phi(S) = \sum_{v\in S}1/d(v)$, where $S$~is the set of mutants and $d(v)$ is the degree of the vertex~$v$. A key notion in the proof of Theorem~\ref{thm:abs-time-intro} is that of a barrier, a set of vertices from which this expected change is small (substantially less than $1/n^2$). For example, each star in the double star is a barrier. If~$G$ contains no barriers, then it is easy to prove Theorem~\ref{thm:abs-time-intro} using the methods of~\cite{DGMRSS2014}. Dealing with barriers is substantially more difficult, and takes up the majority of the paper. It turns out that if the Moran process reaches a barrier~$S$, it is overwhelmingly likely that it subsequently fixates, and that, if it encounters another barrier~$S'$ before doing so, then $S' \supset S$. We use this fact, together with a structural result limiting the total ``time cost'' of any chain of barriers, to prove Theorem~\ref{thm:abs-time-intro}.

We also apply the machinery developed over the course of proving Theorem~\ref{thm:abs-time-intro} to prove the other main result of the paper, so we discuss it further in the next section.

\subsection{Phase transitions}\label{sec:intro-pt}

In a connected regular $n$-vertex digraph, such as the clique (which corresponds to the original Moran process of~\cite{Mor1958}), 
the number of mutants in a fitness-$r$ Moran process evolves as a simple random walk on $\{0, \dots, n\}$ with forward probability $r/(1+r)$~\cite{LHN2005}. This leads to a phase transition between three qualitatively different regimes.

\begin{itemize}
	\item If $r>1$, the fixation probability is $(1-\tfrac1r)/
	(1-\tfrac1{r^n})$.  This is bounded below by the constant
	$1-\tfrac1r$.
	\item If $r=1$, the fixation probability is~$1/n$, which converges to zero
	as $n$~increases.
	\item If $r<1$, the fixation probability is again $(1-\tfrac1r)/
	(1-\tfrac1{r^n})$, but this now converges to zero exponentially
	quickly as~$n$ increases.
\end{itemize}

This raises the question of whether all families of graphs exhibit a significant difference in behaviour between the three regimes described above. We show that this is the case for undirected graphs, writing~$\pfix{G,r}$ for the fixation probability of a Moran process on $G$ with fitness $r$.

\begin{theorem}\label{thm:new-phase-trans}
	For all $C,r>0$, for all sufficiently large connected undirected graphs $G$, we have
	\begin{alignat*}{2}
	f_{G,r} &< 1/n^{C}&\textnormal{ if }r<1,\\
	f_{G,r} &= 1/n&\textnormal{ if }r=1,\\
	f_{G,r} &> (\log n)^C/n\quad&\textnormal{ if }r>1.
	\end{alignat*}
\end{theorem}

In fact, we prove substantially stronger bounds: for a suitable choice of $c_r>0$, we show that if~$r<1$, then $f_{G,r} \le e^{-\exp(c_r(\log n)^{1/3})}$, and if $r>1$, then $f_{G,r} \ge e^{c_r(\log n)^{1/3}}/n$.  (See Theorems~\ref{thm:disadvantageous} and~\ref{thm:advantageous} in Section~\ref{sec:layers}.) Note that the $r=1$ case of Theorem~\ref{thm:new-phase-trans} is not original to this paper; see Section~\ref{sec:rel-work} for a discussion of related work.

In proving Theorem~\ref{thm:new-phase-trans}, the main difficulty lies in the case where $\phi(V)$ is small.
For example, consider the upper bound when $r<1$. In this case, the main obstacle is that, even though $\phi$ is backward-biased in expectation and starts close to $0$, it could still increase by a significant proportion of $\phi(V)$ in a single step if a low-degree vertex becomes a mutant. This sharply limits the effectiveness of martingale-based bounds. Fortunately, it is easy to see that since $\phi(V)$ is small, low-degree vertices must be rare. We exploit this to define a variant $\phi'$ of $\phi$ which remains backward-biased, but assigns low-degree vertices a substantially reduced weight. When $r>1$, $\phi'$ is instead forward-biased; $\phi'$ is also crucial to the proofs of the lower bound of Theorem~\ref{thm:new-phase-trans}, and of Theorem~\ref{thm:abs-time-intro}, where it is applied to barriers.

In order to ensure that $\phi'$ is backward-biased for $r<1$, it is important to minimise the number of edges whose endpoints receive substantially different weights. We therefore carefully construct a partition $(S_0, \dots, S_h)$ of $V$ into layers, where $h$ is as large as possible, such that there is ``little communication'' between non-adjacent layers and all of the low-degree vertices lie in $S_0$. We then define $\phi'$ by assigning weight $\lambda^i/\lambda^h$ to every vertex in $S_i$  for a suitable choice of $\lambda>1$, so that low-degree vertices receive low weight and ``most communication'' happens between layers with similar weights. The construction of $(S_0, \dots, S_h)$ is by far the most technical part of the paper; a detailed sketch is given in Section~\ref{sec:layers-sketch}, but one way of looking at it is that we start with a partition based on vertex degree, then greedily fix individual forward-biased sets in a very carefully chosen order.

Given Theorem~\ref{thm:new-phase-trans}, it is natural to ask whether an analogous result holds for directed graphs. Perhaps surprisingly, we prove that it does not, giving a family with fixation probability $\Theta(1/n)$ for all $r>0$.

\begin{theorem}\label{thm:dir-nophase}
	There is a class~$\calG$ of strongly connected digraphs and a function $C\colon\mathbb{R}_{>0}\to\mathbb{R}_{>0}$ such that for all $r>0$, for all sufficiently large graphs $G \in \calG$,
    \[
        \frac{1}{C(r) |V(G)|}\leq \pfix{G,r}\leq \frac{C(r)}{|V(G)|}\,.
    \]
\end{theorem}

(See Theorems~\ref{thm:dir-fixate} and~\ref{thm:dir-extinct} in Section~\ref{sec:dir-supp} for the value of $C(r)$.) It remains open whether there exists such a family with~$C(r) \rightarrow 1$ as $r\rightarrow 1$. Note it is important that the graphs in $\calG$ are strongly connected, as if this condition were omitted then Theorem~\ref{thm:dir-nophase} would become trivial; for example, an $n$-vertex directed path attains fixation probability $1/n$ for all $r>0$~\cite{LHN2005}. 

In the study of fixation probability, there has been substantial attention devoted to extremal questions. When $r>1$, a graph is said to be an \textit{amplifier} if its fixation probability is larger than that of the clique, and a \textit{suppressor} if it is smaller. It is natural to ask: how strong can an $n$-vertex amplifier or suppressor be? For amplifiers this problem has essentially been solved for both directed and undirected graphs~\cite{Megastars,Incubators}; see Section~\ref{sec:rel-work} for details. For suppressors, much less is known; to our knowledge, the strongest family of (both directed and undirected) suppressors in the literature is due to Giakkoupis~\cite{Gia2016}, and has fixation probability $O(n^{-1/4}\log n)$. Theorem~\ref{thm:dir-nophase}, together with the fact that $f_{G,r} \ge 1/n$ for all $n$-vertex digraphs~$G$ (see Section~\ref{sec:rel-work}), shows that the strongest directed suppressors achieve fixation probability $\Theta(1/n)$.

We note that in the undirected setting, there is a considerable gap between the lower bound on
fixation probability of Theorem~\ref{thm:new-phase-trans} for $r>1$ and the undirected suppressors of~\cite{Gia2016}.  We narrow this gap by exhibiting a family of undirected graphs with fixation probability $O(n^{-1/2})$; see Section~\ref{sec:undir-supp}.

\begin{theorem}\label{thm:undir-suppress}
    There is an infinite family~$\mathcal{H}$ of connected undirected graphs such that for
    all $r>1$, for all
    sufficiently large $H\in\mathcal{H}$,
    \[
        \pfix{H,r}\leq \frac{10r^2}{|V(H)|^{1/2}}\,.
    \]
\end{theorem}

\subsection{Approximation algorithms}

Finally, we consider the computational problem of approximating the fixation probability of an undirected input graph.  We first set out some standard terminology. For all $\eps>0$ and all $x,y\in\mathbb{R}$, we say that $x$ is an \textit{$\eps$-approximation} to $y$ if $(1-\eps)y < x < (1+\eps)y$. Writing $\Sigma^*$ for the set of finite binary strings, consider a computational problem $f\colon\Sigma^* \rightarrow \mathbb{R}_{\ge 0}$. A \emph{randomised approximation scheme (RAS)} for $f$ is a randomised algorithm that takes as input an instance $x \in \Sigma^*$ and a rational error tolerance $\eps \in (0,1)$, and outputs a value which, with probability at least $2/3$, is an $\eps$-approximation to $f(x)$. A \emph{fully polynomial randomised approximation scheme (FPRAS)} is a RAS with running time polynomial in $|x|$ and $\eps^{-1}$. Given a fixed $r>1$, we will consider the problem of approximating $f_{G,r}$ given an undirected input graph $G$ and an error tolerance $\eps$.

Since the expected absorption time of a Moran process 
on an undirected $n$-vertex graph is polynomial in~$n$, and $f_{G,r} \ge 1/n$ for all connected undirected graphs $G$ and all
$r>1$, it is clear that there is an FPRAS for the problem of computing $f_{G,r}$ based on a Monte Carlo approach~\cite[Theorem~13]{DGMRSS2014}.
The FPRAS presented in~\cite{DGMRSS2014} was not optimised, and simulates
$O(n^8\eps^{-4})$ steps of Moran processes. This was later improved by Chatterjee, Ibsen-Jensen and
Nowak~\cite[Theorem~11]{CIN2017}, partly by cleaning up the analysis and
partly by only simulating the steps of the process in which the state actually changes. 
The latter is formalised as an \emph{active Moran process} (see Definition~\ref{def:actm}).
When $G$ has maximum degree~$\Delta$, their algorithm samples $O(n^2\Delta\eps^{-2}\log(\Delta\eps^{-1}))$ steps of active Moran processes, and they present an algorithm for sampling such a step in $O(\Delta)$ time, so overall their running time is $O(n^2\Delta^2\eps^{-2}\log(\Delta\eps^{-1})$. We present another substantial improvement, derived by applying an improved lower bound on
$f_{G,r}$ (see~\cite[Lemma~6]{Gia2016} and Corollary~\ref{cor:Dd}) and by terminating early under
conditions which we show are sufficient to ensure that fixation is overwhelmingly likely. Early termination
was first considered by Barbosa, Donangelo and Souza~\cite{Barbosa},
justified empirically; we make this idea rigorous.

\newcommand{\statenewfpras}{
	Let $r>1$. Then there is a RAS for $f_{G,r}$ that takes as input an undirected graph~$G$ with maximum degree~$\Delta$ and average degree~$\dbar$, and a rational error tolerance $\eps \in (0,1)$, and simulates $O(\Delta\dbar\eps^{-2}\log(\dbar\eps^{-1}))$ steps of active Moran processes.
}
\begin{theorem}\label{thm:new-fpras}
	\statenewfpras
\end{theorem}

Using the sampling algorithm of~\cite{CIN2017}, which takes $O(\Delta)$ time per step simulated, with the Monte Carlo algorithm of Theorem~\ref{thm:new-fpras}, which simulates $T = O(\Delta\dbar\eps^{-2}\log(\dbar\eps^{-1}))$ steps, would actually yield an FPRAS with running time $O(nT)$ rather than $O(\Delta T)$ due to the preprocessing time required for each iteration of the Moran process simulated. In Section~\ref{sec:FPRAS} we present an alternative sampling algorithm, which yields an FPRAS with running time $O(n\dbar + \Delta T)$ under the standard word RAM model when $G$ is given in adjacency-list format. In fact, the $O(n\dbar)$ term in the running time can be removed if, in addition to the adjacency lists, we have access to some basic properties of $G$: how many vertices and edges it has, the degrees of its vertices, its maximum degree, and knowledge of whether or not it is connected. The algorithm is therefore highly efficient for graphs of low maximum degree.

\subsection{Related work}\label{sec:rel-work}

Prior to Theorem~\ref{thm:new-phase-trans}, it was known that for all $n$-vertex graphs $G$, $f_{G,r}$ is increasing in $r$~\cite{MoranAbsorb}
and that $f_{G,1} = 1/n$ \cite[Lemma~1]{DGMRSS2014}. Thus $f_{G,r} \ge 1/n$ for all $r>1$, and $f_{G,r} \le 1/n$ for all $r<1$. To our knowledge, these were the best previously known bounds that apply to all graphs, and the only known bound for $r<1$. When $r>1$, given a fixed initial mutant~$v$, Mertzios and Spirakis \cite[Theorem~4]{MS2013} showed that fixation occurs with probability at least $(r-1)/(r+d(v)/\delta(G))$ in a graph with minimum degree~$\delta(G)$, and a result of Giakkoupis \cite[Lemma~6]{Gia2016} implies that fixation occurs with probability at least $1-r^{-\delta(G)/d(v)}$. These results both imply better bounds than $f_{G,r} \ge 1/n$ in the case where $G$ is sparse or has high minimum degree.

We have already discussed the previously best-known family of suppressors from~\cite{Gia2016}; note that using \cite[Lemma~6]{Gia2016}, one can show that they have fixation probability $\Omega(n^{-1/4})$ and so are improved on by Theorem~\ref{thm:undir-suppress}. Theorem~\ref{thm:dir-nophase} provides the first known family of digraphs with fixation probability $\Omega(1/\mbox{poly}(n))$ when $r<1$.

In Section~\ref{sec:intro-pt}, we alluded to the fact that the problem of finding the strongest possible amplifier had essentially been solved. Any directed graph has extinction probability $\Omega(n^{-1/2})$~\cite{Incubators}, which is tight up to a polylogarithmic factor~\cite{Megastars}, and any undirected graph has extinction probability $\Omega(n^{-1/3})$, which is tight up to a constant factor~\cite{Incubators}. In fact, the results of~\cite{Incubators} generalise to sparse graphs; any $m$-edge undirected graph has extinction probability $\Omega(\max\{n^{-1/3}, n/m\})$, which is also tight up to a constant factor. (See also Giakkoupis~\cite{Gia2016} for a lower bound of $\Omega(n^{-1/3}(\log n)^{-4/3})$ in the dense undirected setting, which is proved to be tight to within a factor of $(\log n)^{7/3}$.) 

Several generalisations of the Moran process are also studied~\cite{LHN2005,SRJ-survey}. Some versions of the process allow edge weights, which is equivalent to allowing multiple edges in the graph, and others determine the fitness of a vertex partly or fully by game payoffs with its neighbours. More recently, a variant has been proposed~\cite{Spirakis:TwoGraphs} in which the mutants and non-mutants interact along different graphs. In this paper, we consider the original process of Lieberman \emph{et al.}~\cite{LHN2005}, with a single, simple, unweighted graph and mutants with fixed fitness~$r$.

The Moran process fits into a very well-studied family of graph processes sometimes known as epidemic models or interacting particle systems (see e.g.~\cite{Grimmett,DS,Durrett:NACS,Lig1999:IntSys,Shah}). It is particularly similar to the contact process and the voter model. In the contact process, non-mutants do not reproduce, and instead mutants spontaneously become non-mutants after a random time. In the voter model, individuals are chosen for replacement rather than reproduction. However, note that the behaviour of all of these processes is very sensitive to their specific definitions, and results generally do not transfer between them. 

\section{Preliminaries and notation}\label{sec:prelim}

In our notation, multiplication has higher precedence
than division, so $a/bc$ denotes $\tfrac{a}{bc}$. 
For all positive integers $n$, we write $[n] = \{1, \dots, n\}$.
Whenever we write a logarithm without specifying the base, we mean that
the base is~$e$.

Given a graph $G=(V,E)$, a Moran process on $G$ with fitness $r>0$ is a discrete-time Markov chain with state space $2^V$. It evolves as described in the introduction, taking the state at time~$t$ to be the set $M(t)$ of mutants at that time.  Even though we are ultimately interested in the case where the initial state contains a uniformly random mutant, it helps to consider arbitrary initial states. If a vertex~$u$ is chosen for reproduction at time~$t$, and its state is copied to~$v$, we say $u$ spawns onto $v$ at time~$t$. For all $X\subseteq V$, we write $W(X) = |V| + (r-1)|X|$ for the total population fitness when $X$~is the current state.
We write $\pfix{G,r}$ for the fixation probability of the Moran process
on an undirected graph~$G$, with mutant fitness~$r$, when a single
initial mutant is chosen u.a.r. 

Graphs are undirected unless otherwise stated.
For a graph $G=(V,E)$ and sets $A,B\subseteq V$, we write $G[A]$ for the subgraph induced by~$A$, and $E(A,B)$ for
the set of ordered pairs $(v,w)$ in $A\times B$ such that $vw\in E$.
For $v\in V$, $N(v) = \{w\mid vw\in E\}$. $\Delta(A) = \Delta_G(A) =
\max\,\{d(a)\mid a\in A\}$ and $\delta(A) = \delta_G(A) =
\min\,\{d(a)\mid a\in A\}$. Given some $v\in V$ and
$A\subseteq V$, we write $d_A(v) = |E(\{v\},A)|$.
We write $\Delta(G)=\Delta(V(G))$ and similarly for $\delta(G)$. 
Finally, we write $\dbar(G) = \frac{1}{|V|}\sum_{v \in V}d(v)$.

Given sets $S_0, \dots, S_q$ and some $i\in\mathbb{Z}$,
we write $S_{\leq i} = S_0\cup\dots \cup S_i$ and
$S_{\geq i} = S_i \cup \dots \cup S_q$.  We adopt the convention that
$S_i=\emptyset$ for any $i<0$ or $i>q$.  When we write a partition as
$S_0, \dots, S_\nlevs$, $\nlevs$~is not arbitrary but is specifically
the value defined in Definition~\ref{def:consts}; we will only do so
after Definition~\ref{def:consts}.  $\subset$~denotes the proper subset relation.

\section{Drift and potential}

We will rely on the following potential function, first defined in~\cite{DGMRSS2014}, throughout the paper.
\begin{definition}\label{def:phi}
	Given a connected graph $G$ with at least two vertices, 
	define $\phi\colon 2^{V(G)}\rightarrow\reals$ by $\phi(S) =
	\sum_{v \in S}1/d(v)$.
	We write $\phi(v)$ for~$\phi(\{v\})$.
\end{definition}

Note that in Definition~\ref{def:phi}, since $G$ is connected, we have $d(v) \ge 1$ for all $v \in S$.
We
often use the fact that $\delta(S)\phi(S)\leq |S|\leq
\Delta(S)\phi(S)$, which is immediate from the definition.

\begin{definition}
    Let $G$ be a connected graph with at least two vertices.  For disjoint $A,B\subseteq V(G)$,
    we define the \emph{drift}
    \begin{equation*}
        \dr{A,B} \ = \!\!\sum_{(x,y) \in E(A,B)}\frac{1}{d(x)\,d(y)}\,.
    \end{equation*}
\end{definition}

We will often use the trivial fact that drift is symmetric, additive and monotone. Thus for all $A,B,C \subseteq V(G)$ such that $(A \cup B) \cap C = \emptyset$:
\begin{itemize}
	\item $\dr{A,C} = \dr{C,A}$;
	\item if $A$ and $B$ are disjoint, then $\dr{A \cup B,C} = \dr{A,C} + \dr{B,C}$;
	\item if $A\subseteq B$, then $\dr{A, C} \leq \dr{B, C}$.
\end{itemize}

The definition of drift is motivated by the following simple lemma.

\begin{lemma}\label{lem:drift-works}
	Let $r>1$, let $G = (V,E)$ be a connected graph with at least two vertices, and $\MM$ be a Moran process on $G$ with fitness~$r$. For all $S \subseteq V$ and all $t \ge 0$, we have
	\[
		\E[\phi(\MM{t+1}) - \phi(\MM{t}) \mid \MM{t} = S] \ge \frac{r-1}{rn}\dr{S,V \setminus S}.
	\]
\end{lemma}
\begin{proof}
	We have
	\begin{align*}
	 \E[\phi(\MM(t+1)) - \phi(\MM(t))\mid \MM(t)=S]
	     &= \frac{1}{W(S)}
	        \!\!\!\!\sum_{\substack{(x,y)\in\\E(S,V\setminus S)}}\!\!\!\! \bigg(
	            \frac{r}{d(x)}\,\frac{1}{d(y)} - \frac{1}{d(y)}\,\frac{1}{d(x)}
	        \bigg)\\
	     &= \frac{r-1}{W(S)}\, \dr{S, V\setminus S}\,.
	\end{align*}
	The result follows since $W(S) \le rn$.
\end{proof}

We collect some other important properties of drift here for easy reference. 

\begin{lemma}\label{lem:drift-props}
    Let $G=(V,E)$ be a connected graph with at least two vertices and
    let $\emptyset\subset A\subseteq V$.
	\begin{enumerate}[(i)]
	\item For any $x \in A$, $d_{V \setminus A}(x) \le \Delta(G)\,\dr{A,V \setminus A}\,d(x)$.\label{d:ext-deg}
	\item If $\dr{A,V\setminus A} \le 1/2\Delta(G)$ and $A\subset V$, then $|A|
        \ge 1/(2\Delta(G)\,\dr{A,V\setminus A})$.\label{d:size}
    \item If $\dr{A,V\setminus A} \le 1/2\Delta(G)$, then $\phi(A) \ge 1/2$.\label{d:phi}
	\end{enumerate}
\end{lemma}
\begin{proof}
    For all $x \in A$, 
	\[
		\dr{A,V\setminus A} \ge \dr{\{x\},V\setminus A}
            = \sum_{\substack{y: (x,y)\in\\E(A,V\setminus A)}}\frac{1}{d(x)\,d(y)}
            \ge \frac{d_{V\setminus A}(x)}{\Delta(G)\,d(x)}\,,
	\]
	and (\ref{d:ext-deg}) is immediate.
	
    For (\ref{d:size}), suppose that $\dr{A,V\setminus A} \le 1/2\Delta(G)$ and $A \subset V$. By connectedness, there is some $x\in A$ with a neighbour in
    $V\setminus A$ (so $\dr{A,V \setminus A} \ne 0$). By (\ref{d:ext-deg}),
    \[
    	|A| \ge d_A(x) = d(x) - d_{V \setminus A}(x) \ge d(x)\big(1 - \Delta(G)\dr{A, V \setminus A}\big).
    \]
    By hypothesis, this implies $|A| \ge d(x)/2$. Applying (\ref{d:ext-deg}) again yields
    \[
    	|A| \ge \frac{d_{V \setminus A}(x)}{2\Delta(G)\dr{A,V \setminus A}} \ge \frac{1}{2\Delta(G)\dr{A,V \setminus A}}.
    \]

    For (\ref{d:phi}), there are two cases.  If $A=V$, then
    $\phi(A)\geq |V|/(|V|-1)>1$.  Otherwise,
    let $y \in A$ be a vertex of maximum degree (and hence
    minimum potential). By (\ref{d:ext-deg}), $d_{V \setminus A}(y)\leq d(y)/2$, so
    $|A|\geq d_A(y)\geq d(y)/2$, and each neighbour of~$y$ in~$A$ has
    degree at most $d(y)$. Therefore, $\phi(A) \ge 1/2$.
\end{proof}

\begin{lemma}\label{lem:more-drift-props}
	Let $G=(V,E)$ be a connected graph with at least two vertices, and suppose $S_1, S_2 \subseteq V$ are non-empty. Then:
	\begin{enumerate}[(i)]
		\item $\dr{S_1 \setminus S_2,V\setminus(S_1\setminus S_2)} \le \dr{S_1, V\setminus S_1} + \dr{S_2, V\setminus S_2}$;\label{d:sub}
		\item if, for all non-empty $S_1' \subset S_1$ and non-empty $S_2' \subset S_2$, $\dr{S_1', V\setminus S'_1} > \dr{S_1, V\setminus S_1}$ and $\dr{S_2',V\setminus S'_2} > \dr{S_2,V\setminus S_2}$, then $S_1 \subseteq S_2$, or $S_2 \subseteq S_1$, or the two sets are disjoint.\label{d:minimal}
	\end{enumerate}
\end{lemma}
\begin{proof}
	Part~(\ref{d:sub}) follows from
	\begin{align*}
		\dr{S_1 \setminus S_2, V\setminus (S_1\setminus S_2)}
            &= \dr{S_1 \setminus S_2,S_2} + \dr{S_1 \setminus S_2,V\setminus (S_1 \cup S_2)}\\
            &\le \dr{V\setminus S_2, S_2} + \dr{S_1, V\setminus S_1}\,.
	\end{align*}
	
	Suppose $S_1, S_2 \subseteq V$ satisfy the condition of (ii) above. Then,
	\begin{align*}
		\dr{S_1,V\setminus S_1} + \dr{S_2, V\setminus S_2} 
		&= \dr{S_1, S_2 \setminus S_1} + \dr{S_1,V\setminus (S_1 \cup S_2)}\\
        &\qquad\qquad + \dr{S_2,S_1\setminus S_2} + \dr{S_2,V\setminus(S_1\cup S_2)}\\
		&\ge \dr{S_1, S_2 \setminus S_1} + \dr{S_1\setminus S_2,V\setminus(S_1 \cup S_2)}\\
        &\qquad\qquad + \dr{S_2,S_1\setminus S_2} + \dr{S_2\setminus S_1,V\setminus(S_1\cup S_2)}\\
		&= \dr{S_1 \setminus S_2, V\setminus (S_1 \setminus S_2)} + \dr{S_2 \setminus S_1, V\setminus (S_2 \setminus S_1)}\,.
	\end{align*}
	If $\emptyset \subset S_1 \setminus S_2 \subset S_1$ and $\emptyset \subset S_2 \setminus S_1 \subset S_2$, then, by the hypothesis,
    \[
        \dr{S_1 \setminus S_2, V\setminus (S_1 \setminus S_2)} + \dr{S_2 \setminus S_1, V\setminus (S_2 \setminus S_1)}
            > \dr{S_1,V\setminus S_1} + \dr{S_2, V\setminus S_2}\,,
    \]
    a contradiction.  This establishes (\ref{d:minimal}).
\end{proof}

\begin{lemma}\label{lem:deg-drift}
	Let $G = (V,E)$ be a connected graph with at least two vertices, let $S_1,S_2 \subseteq V$ be disjoint, and suppose that every vertex in $S_1$ has degree at most $d_1$ and every vertex in $S_2$ has degree at least $d_2>0$. 
	Then $\dr{S_1,S_2} \le d_1\phi(S_1)/d_2$.
\end{lemma}
\begin{proof}
	For all $u \in S_1$, we have 
	\[
		\dr{\{u\},S_2} = \frac{1}{d(u)}\sum_{w \in N(u) \cap S_2} \frac{1}{d(w)} \le \frac{1}{d(u)}d(u)\frac{1}{d_2} = \frac{1}{d_2}.
	\]
Since $|S_1| \leq \Delta(S_1) \phi(S_1)$ we have
	 $|S_1| \le d_1\phi(S_1)$. Therefore,
	\[
		\dr{S_1, S_2} = \sum_{u \in S_1}\dr{\{u\},S_2} \le \frac{|S_1|}{d_2} \le \frac{d_1\phi(S_1)}{d_2}.\qedhere
	\]
\end{proof}

\subsection{Potential lemmas}

In this section, we give a simplified proof of the following result of Giakkoupis~\cite{Gia2016}. We will make use of an intermediate result, Lemma~\ref{lem:multsubmart}, in the proof of Theorem~\ref{thm:abs-time-intro}. In Section~\ref{sec:discount-pot}, we will use the same techniques to prove a slightly weaker but more general result, which we will need to prove Theorems~\ref{thm:abs-time-intro} and~\ref{thm:new-phase-trans}. Note that the most difficult parts of these proofs are in Sections~\ref{sec:layers} and~\ref{sec:abs-time}; for the moment, we use standard martingale arguments.

\newcommand{\statepotential}{
	Let $G$ be a connected graph with at least two vertices.
	Let $\MM$ be a Moran process on~$G$, with fitness $r>1$ and deterministic initial state~$X$.
	\begin{equation*}
	\Pr(\MM\textnormal{ fixates})
	\ge \frac{1 - r^{-\phi(X)\delta(G)}}{1-r^{-\phi(V)\delta(G)}}\,.
	\end{equation*}
}
\begin{theorem}[{\cite[Lemma 6]{Gia2016}}]\label{thm:potential}
	\statepotential
\end{theorem}

First, we state and prove a corollary of Theorem~\ref{thm:potential}; then we prove two lemmas that, together, imply the theorem.

\begin{corollary}\label{cor:Dd}
	Let $G=(V,E)$ be a connected $n$-vertex graph with at least two vertices and average degree at most $\overline{d}$, and let $r>1$.
	Then 
	\[
		f_{G,r} \geq \frac{(r-1)\phi(V)}{2rn} \geq \frac{r-1}{2r\overline{d}}.
	\]
\end{corollary}
\begin{proof}
	For all $x \in V$, let $M^x$ be a Moran process on $G$ with fitness $r$ and initial state $\{x\}$. By Theorem~\ref{thm:potential}, we have
	\begin{align*}
		\Pr(M^x\mbox{ fixates}) &\ge \frac{1 - r^{-\phi(x)\delta(G)}}{1 - r^{-\phi(V)\delta(G)}} \ge 1 - r^{-\phi(x)\delta(G)} \ge 1-r^{-\phi(x)}\\
					&\ge 1 - e^{-\phi(x)\min\{1,\log r\}} \ge \frac{\phi(x)\min\{1,\log r\}}{2} \ge \frac{(r-1)\phi(x)}{2r}.
	\end{align*}
	(Here the penultimate inequality follows since $0 \le \phi(x)\min\{1,\log r\} \le 1$.) Thus
	\begin{equation}\label{eqn:Dd}
		f_{G,r} = \frac{1}{n}\sum_{x \in V}\Pr(M^x\mbox{ fixates}) \ge \frac{(r-1)\phi(V)}{2rn},
	\end{equation}
	as required. Moreover, $n/\phi(V)$ is the harmonic mean of the degrees of $G$'s vertices; it follows by the arithmetic mean--harmonic mean inequality that $n/\phi(V) \le \overline{d}$, and hence $\phi(V)/n \ge 1/\overline{d}$. Together with~\eqref{eqn:Dd}, this implies the result.
\end{proof}

To prove Theorem~\ref{thm:potential}, we require the following technical lemma.

\begin{lemma}\label{lem:tech}
    For all $r>1$, $\delta \ge 1$ and all $0 < \alpha,\beta \le 1/\delta$,
    \begin{equation*}
        \frac{r\alpha}{r\alpha+\beta}r^{-\beta\delta}
            + \frac{\beta}{r\alpha+\beta}r^{\alpha\delta} \le 1\,.
    \end{equation*}
\end{lemma}
\begin{proof}
    First note that, by rescaling $\alpha$ and $\beta$, it suffices to prove the result when $\delta=1$. So  fix $\alpha \in (0,1]$, and define $f\colon[0,1]\rightarrow\reals$ by $f(\beta) = \beta r^{\beta+\alpha} - (r\alpha + \beta)r^\beta + r\alpha$. For all $\beta \in (0,1]$, we have
    \begin{equation*}
        \frac{f(\beta)}{r^\beta(r\alpha+\beta)}
            = \frac{r\alpha}{r\alpha+\beta}r^{-\beta}
              + \frac{\beta}{r\alpha+\beta}r^{\alpha} - 1\,.
    \end{equation*}

    It therefore suffices to prove that $f(\beta) \le 0$ for all
    $\beta \in (0,1]$. We will show that $f(0),f(1) \le 0$, $f'(0) \le 0$,
    and that $f$~has at most one stationary point in $(0,1]$, which
    together imply the result.
		
    It is immediate that $f(0) = 0$. Define a function
    $g\colon[0,1]\rightarrow\reals$ by $g(x) = r^x-1+(1-r)x$, and note
    that $f(1) = rg(\alpha)$. Then we have $g''(x) = r^x(\log r)^2 >
    0$, so $g$~is convex. Moreover, $g(0) = g(1) = 0$, so $g(x) \le 0$
    for all $x \in [0,1]$. It follows that $f(1) \le 0$.
		
    Now, for all $\beta \in [0,1]$, we have
    \[
        f'(\beta) = r^\beta\big(\beta(r^\alpha\log r - \log r) - \alpha r \log r
                   + r^\alpha-1\big)\,.
    \]
    It follows that $f$ has at most one stationary point in
    $(0,1]$. Moreover, $f'(0) = -\alpha r\log r+r^\alpha-1$.
    Since $\alpha \in (0,1]$, $-0\cdot r\log r+r^0-1 = 0$, and 
    \begin{equation*}
        \frac{\partial}{\partial\alpha}(-\alpha r\log r+r^\alpha-1)
             = (r^\alpha-r)\log r < 0 \quad\text{for all }\alpha \in (0,1)\,,
    \end{equation*}
    it follows that $f'(0) \leq 0$, as required.
\end{proof}

\begin{lemma}\label{lem:multsubmart}
    Let $G=(V,E)$ be a connected graph on at least two vertices, and let $\MM$ be a Moran process on $G$ with fitness $r>1$.
    For all $t \ge 0$, $\E(r^{-\phi(\MM(t+1))\delta(G)} \mid \MM(t)) \le
    r^{-\phi(\MM(t))\delta(G)}$.
\end{lemma}
\begin{proof}
    Let $X\subseteq V$. Say that an edge $xy$ is ``chosen for reproduction'' at time~$t$ if, at that time, either $x$ spawns onto~$y$ or $y$ spawns onto~$x$.  For all $(x,y)\in E(X,V\setminus X)$,
    \begin{multline*}
	\E\big(r^{-\phi(\MM(t+1))\delta(G)} \mid \MM(t) = X,\ xy\mbox{ is chosen for reproduction at time }t\big)\\
		= r^{-\phi(X)\delta(G)}\left(\frac{r/d(x)}{r/d(x)+1/d(y)}r^{-\delta(G)/d(y)} + \frac{1/d(y)}{r/d(x)+1/d(y)}r^{\delta(G)/d(x)}\right)\,.
    \end{multline*}
    Let $\alpha = 1/d(x)$ and $\beta = 1/d(y)$.  Since
    $\alpha,\beta\leq 1/\delta(G)$, we may apply Lemma~\ref{lem:tech} and
    obtain
    \begin{equation*}
        \E\big(r^{-\phi(\MM(t+1))\delta(G)} \mid \MM(t) = X,\ xy\text{ is chosen for reproduction at time }t\big) \le r^{-\phi(X)\delta(G)}\,.
    \end{equation*}

    On the other hand, if $M(t)=X$ and an edge outside $E(X, V\setminus X)$ is chosen for reproduction, then $M(t+1)=X$.  It follows that
    \[
        \E(r^{-\phi(\MM(t+1))\delta(G)} \mid \MM(t) = X)
            \leq r^{-\phi(X)\delta(G)}\,.\qedhere
    \]
\end{proof}
	
\begin{reptheorem}{thm:potential}
	\statepotential
\end{reptheorem}
\begin{proof}
    Let $\tau = \min \{t \ge 0 \mid \phi(M(t)) \in \{0,\phi(V)\}\}$ be
    the absorption time, which is a stopping time. 
    By Lemma~\ref{lem:multsubmart} and the
    optional stopping theorem, $\E(r^{-\phi(M(\tau))\delta(G)}) \le
    r^{-\phi(X)\delta(G)}$.  Moreover, writing $p$ for the fixation
    probability of~$\MM$, we have $\E(r^{-\phi(\MM(\tau))\delta(G)}) =
    (1-p) + pr^{-\phi(V)\delta(G)}$.  Combining the two inequalities
    and rearranging gives $r^{-\phi(X)\delta(G)}-1 \ge
    p(r^{-\phi(V)\delta(G)}-1)$. Since $r^{-\phi(V)\delta(G)}-1 \le r^{-1}-1 < 0$,
    the result follows.
\end{proof}

\subsection{Discounted potential functions}\label{sec:discount-pot}

To prove our main results, Theorems~\ref{thm:abs-time-intro} and~\ref{thm:new-phase-trans}, we will require the following more general class of potential functions.

\begin{definition}\label{defn:phif}
    Let $G=(V,E)$ be a connected graph on at least two vertices and let $f\colon V\to\reals_{\geq 0}$.
    We denote by $\phif$ the function $V\to\reals_{\geq 0}$ given
    by $\phif(v) = f(v)/d(v)$, and for all $S \subseteq V$ we define
    $\phif(S) = \sum_{v\in S} f(v)/d(v)$. 
\end{definition}

Thus if $f(v) = 1$ for all $v \in V$, then $\phif = \phi$. Our goal in this section will be to prove an analogue of Theorem~\ref{thm:potential} for potential functions of the form $\phif$. Recall that in order to prove Theorem~\ref{thm:potential}, we first proved in Lemma~\ref{lem:multsubmart} that for any Moran process $\MM$ on $G$, $r^{-\phi(\MM)\delta(G)}$ is a supermartingale. 

\begin{definition}\label{defn:psif}
    For $r\geq 1$, let
    $\beta=\tfrac{r-1}{6r+2}$.  Let $G=(V,E)$ be a connected graph on at least two vertices and let
    $f\colon V\to\reals_{\geq 0}$ be not everywhere zero.  Let $m_f
    = \max_{v\in V}\phi_f(v)$, and define $\psif\colon 2^V\to\reals_{\geq 0}$
    by $\psif(S) = e^{-\phif(S)\beta/m_f}$.
\end{definition}

Thus if $f(v) = 1$ for all $v \in V$, then $m_f = 1/\delta(G)$ and $\psif(X) = e^{-\beta\phi(X)\delta(G)}$ for all $X \subseteq V$. In Lemma~\ref{lem:psif-supermart}, we will essentially show that if $\phif(\MM)$ is sufficiently strongly forward-biased, then $\psif(\MM)$ is backward-biased, allowing us to apply the optional stopping theorem as in Theorem~\ref{thm:potential}. However, $\phif(\MM)$ will not in general be forward-biased throughout the evolution of $\MM$, and so we will require the following sufficient condition.
\begin{definition}\label{def:valid}
	For $r > 1$, let $\lambda=(r+1)/2>1$.
	Let $G=(V,E)$ be a connected graph on at least two vertices and let $f\colon V\to\reals_{\geq 0}$. 
	For all $X \subseteq V$, $f$ is \emph{valid for $X$} if
	\begin{equation*}
	\sum_{\substack{(x,y)\in E(X,V\setminus X)\\ f(x) > \lambda f(y)}}
	\frac{f(x)}{d(x)\,d(y)}
	\leq \frac{r-1}{4r}
	\sum_{\substack{(x,y)\in E(X,V\setminus X)\\ f(x)\leq \lambda f(y)}}
	\frac{f(y)}{d(x)\,d(y)}.
	\end{equation*}    
	For all $0 \le x^- < x^+ \le \phif(V)$, say $f$ is \textit{$(x^-,x^+)$-valid} if $f$ is
	valid for all $X \subseteq V$ for which $x^- < \phif(X) < x^+$.
\end{definition}
\begin{lemma}\label{lem:add-submart}
	Let $G=(V,E)$ be a connected graph on at least two vertices, and 
	let $\MM$ be a Moran process on $G$ with
	fitness $r>1$.  Let $f\colon V\to \reals_{\geq 0}$ and let
	$X\subseteq V$.  If $f$ is valid for $X$ then, for all $t\geq 0$,
	$\E[\phif(M(t+1))\mid M(t)=X]\geq \phif(X)$.
\end{lemma}
\begin{proof}
	First note that 
	\[
		\E [\phif(\MM{t+1}) - \phif(\MM{t}) \mid \MM{t} = X] = \frac{1}{W(X)}\sum_{(x,y) \in E(X,V \setminus X)} \frac{rf(y) - f(x)}{d(x)d(y)}\,.
	\]
	By splitting the sum, we see that this expectation is at least
	\[
		\frac{1}{W(X)}\Bigg(
		\sum_{\substack{(x,y)\in E(X,V\setminus X)\\ f(x)\leq \lambda f(y)}}
		\frac{(r-\lambda)f(y)}{d(x)\,d(y)}
		- \sum_{\substack{(x,y)\in E(X,V\setminus X)\\ f(x) > \lambda f(y)}} \frac{f(x)}{d(x)\,d(y)} 
		\Bigg).
	\]
	Since $r-\lambda = (r-1)/2$, this expression is non-negative if and only if
	\[
		\sum_{\substack{(x,y)\in E(X,V\setminus X)\\ f(x) > \lambda f(y)}} \frac{f(x)}{d(x)\,d(y)}
		\le \frac{r-1}{2}\cdot\sum_{\substack{(x,y)\in E(X,V\setminus X)\\ f(x)\leq \lambda f(y)}}\frac{f(y)}{d(x)\,d(y)}.
	\]
	Since $(r-1)/4r < (r-1)/2$, the result follows.
\end{proof}

Note that while the above lemma is useful for motivating Definition~\ref{def:valid}, we will not actually use it in this section; however, we will need it later in Section~\ref{sec:layers}. We now prove an analogue of Lemma~\ref{lem:multsubmart}.

\begin{lemma}\label{lem:psif-supermart}
    Let $G=(V,E)$ be a connected graph on at least two vertices, and 
    let $\MM$ be a Moran process on $G$ with
    fitness $r>1$.  Let $f\colon V\to \reals_{\geq 0}$ be not everywhere zero and let
    $X\subseteq V$.  If $f$ is valid for $X$ then, for all $t\geq 0$,
    $\E[\psif(M(t+1))\mid M(t)=X]\leq \psif(X)$.    
\end{lemma}
\begin{proof}
    Write $x\sim y$ to
    abbreviate $(x,y)\in E(X,V \setminus X)$ and note that this relation is not
    symmetric.  The Moran process is time-invariant, so it suffices to
    prove the result for $t=0$. Let $W=W(X)$ for brevity. 

    Let $\mathcal{E}$ be the event that, at time~$1$, a mutant is
    spawned along an edge $xy$ with $x\sim y$ and
    $f(x) \le \lambda f(y)$. Let $\mathcal{E}'$ be the event that, at
    time~$1$, a non-mutant is spawned along an edge~$xy$ with
    $x\sim y$ (since $X$~is the set of mutants, the spawn must have
    been from $y$ to~$x$). Let
    \[
		p = r \!\!\!\!\sum_{\substack{x\sim y\\ f(x)\leq \lambda f(y)}}\!\!\!\!
		                \frac{1}{d(x)} = W\,\Pr(\calE \mid \MM(0) = X)
		\quad\text{and}\quad
		q = \sum_{x\sim y} \frac{1}{d(y)} 
          = W\,\Pr(\calE' \mid \MM(0) = X)\,.
    \]

    First, suppose that neither $\calE$ nor~$\calE'$ occurs. Then
    either $M(1)=M(0)$ or a mutant was spawned from some $x\in X$ to
    some $y\in V\setminus X$ with $f(x)>\lambda f(y)$.  So, in this
    case, $\psif(\MM(1)) = \psif(\MM(0))\,e^{-\phif(y)\beta/m_f} \leq \psif(\MM(0))$.  If $\calE$~occurs
    along~$xy$, then $\psif(\MM(1)) = \psif(\MM(0))\,e^{-\phif(y)\beta/m_f}$.
    If $\calE'$~occurs along~$xy$, then $\psif(\MM(1)) =
    \psif(\MM(0))\,e^{\phif(x)\beta/m_f}$. Thus,
    \begin{multline}
		\E\big(\psif(\MM(1))\mid \MM(0)=X\big) \leq {}\\ 
		    \psif(X) \Big(
		        1 - \frac{p}{W} - \frac{q}{W}
                + \frac{r}{W}
                  \!\!\!\!\sum_{\substack{x\sim y\\f(x)\leq\lambda f(y)}}\!\!\!\!
		              \frac{1}{d(x)}\,e^{-\phif(y)\beta/m_f}
                + \frac{1}{W}\sum_{x\sim y}\frac{1}{d(y)}\,e^{\phif(x)\beta/m_f}
		    \Big)\,.
	\label{eqn:EPsi}
	\end{multline}

    By Taylor expansion, $e^z\leq 1+z+z^2$ for all $z\in[-1,1]$.
    Since $0\leq \phif(x)\beta/m_f < 1$, we may apply this bound to the
    terms of the sums in~\eqref{eqn:EPsi}.

	For any $x\sim y$ with $f(x)\leq\lambda f(y)$,
	\begin{align*}
		\hspace{2em}&\hspace{-2em}
		\frac{r}{d(x)}\,e^{-\phif(y)\beta/m_f} + \frac{1}{d(y)}\,e^{\phif(x)\beta/m_f} \\
		&\leq \frac{r}{d(x)}\left(
                  1 - \frac{f(y)}{d(y)}\,\frac{\beta}{m_f}
		          + \frac{f(y)^2}{d(y)^2}\,\frac{\beta^2}{m_f^2}\right)
		+ \frac{1}{d(y)}\left(1 + \frac{f(x)}{d(x)}\,\frac{\beta}{m_f}
		+ \frac{f(x)^2}{d(x)^2}\,\frac{\beta^2}{m_f^2}\right)\\
		&= \frac{r}{d(x)} + \frac{1}{d(y)}
		+ \frac{\beta}{m_f\,d(x)\,d(y)}\left(
		f(x) - rf(y) + \frac{\beta}{m_f}\left(
		\frac{rf(y)^2}{d(y)} + \frac{f(x)^2}{d(x)}
		\right)\right)\\
		&\leq \frac{r}{d(x)} + \frac{1}{d(y)}
		+ \frac{\beta}{m_f\,d(x)\,d(y)}\big(
		f(x) - rf(y) + \beta(rf(y) + f(x))\big)\\
		&\leq \frac{r}{d(x)} + \frac{1}{d(y)}
		- \frac{\beta}{m_f\,d(x)\,d(y)}f(y)\,\big(
		r - \lambda - \beta r - \beta\lambda\big)\\
		&= \frac{r}{d(x)} + \frac{1}{d(y)} -  \frac{r-1}{4}\,\frac{\beta}{m_f\,d(x)\,d(y)}f(y)\,,
		\tagme\label{eqn:firstpart}
	\end{align*}
	where the second inequality is because $f(v)/m_fd(v)\leq 1$ for
	all~$v\in V$ and the third because $f(x)\leq \lambda f(y)$.

    For $x\sim y$ with $f(x) > \lambda f(y)$, again using
    $f(v)/m_fd(v)\leq 1$, we find
	\begin{align*}
		\frac{1}{d(y)}\,e^{\phif(x)\beta/m_f}
		&\leq \frac{1}{d(y)} \left(
		1 + \frac{f(x)}{d(x)}\,\frac{\beta}{m_f}
		+ \frac{f(x)^2}{d(x)^2}\,\frac{\beta^2}{m_f^2}\right)\\ 
		&= \frac{1}{d(y)} + \frac{\beta}{m_f\,d(x)\,d(y)}\left(
		f(x) + \frac{f(x)^2}{d(x)}\,\frac{\beta}{m_f}\right)\\
		&\leq \frac{1}{d(y)} + \frac{\beta}{m_f\,d(x)\,d(y)} f(x)\,(1+\beta)\\
		&< \frac{1}{d(y)} + \frac{r\beta}{m_f\,d(x)\,d(y)} f(x)\,,
		\tagme\label{eqn:secondpart}
	\end{align*}
    where the final inequality is because $1+\beta < r$ since $r>1$. 

	Now, summing \eqref{eqn:firstpart} and~\eqref{eqn:secondpart} gives
	\begin{multline*}
		\sum_{\substack{x\sim y\\f(x)\leq\lambda f(y)}}\!\!\!\!
		\frac{r}{d(x)}\,e^{-\phif(y)\beta/m_f}
		+ \sum_{x\sim y} \frac{1}{d(y)}\,e^{\phif(x)\beta/m_f} \\
		\leq p + q
		+ \frac{r\beta}{m_f}
		\!\!\!\!\sum_{\substack{x\sim y\\f(x)> \lambda f(y)}}\!\!\!\! \frac{f(x)}{d(x)\,d(y)}            - \frac{r-1}{4}\,\frac{\beta}{m_f}
		\!\!\!\!\sum_{\substack{x\sim y\\f(x)\leq \lambda f(y)}}\!\!\!\! \frac{f(y)}{d(x)\,d(y)}
		\,.
	\end{multline*}
	The right-hand side is at most $p+q$ by the hypothesis of the
	lemma, so the right-hand side of~\eqref{eqn:EPsi} is at most
	$\psif(X)$ and we are done.
\end{proof}

Lemma~\ref{lem:psif-supermart} will allow us to apply the optional stopping theorem to a simple variant of $\psif(M)$. We now exploit this to obtain an analogue of Theorem~\ref{thm:potential} (see Lemma~\ref{lem:gen-pot-b}).

\begin{lemma}\label{lem:gen-pot-a}
	Let $r>1$, let $G=(V,E)$ be a connected graph on at least two vertices, and let $f\colon V\rightarrow\mathbb{R}_{\ge 0}$ 
	be a function that is not
	everywhere zero. 
	Let $0\leq x^- <  x^+ \leq \phif(V)$, and suppose that $f$ is $(x^-, x^+)$-valid and $x^+ - x^- > m_f$. Let $\MM$ be a Moran process on $G$ with fitness~$r$, satisfying $\phi_f(\MM{0}) \ge x^+ - m_f$. Then with probability at least $1-2\exp(-\beta(x^+ - x^-)/m_f)$,
	\[
		\min\{t \ge 0 \mid \phi_f(\MM{t}) \ge x^+\} < \min\{t \ge 0 \mid \phi_f(\MM{t}) \le x^-\}.
	\]
	(Note that with probability 1, at most one of these times is infinite.)
\end{lemma}
\begin{proof}
	If $\phi_f(\MM{0}) \ge x^+$ then there is nothing to prove, so suppose $\phi_f(\MM{0}) < x^+$. Let $\tau$ be the minimum of the two stopping times in the lemma statement, and let $p = \Pr(\phi_f(\MM{\tau}) \le x^-)$. Thus $\E(\psi_f(\MM{\tau})) \ge pe^{-\beta x^-/m_f}$. For all $0 \le t \le \tau$, let $X_t = \psi_f(\MM{t})$; for all $t \ge \tau$, let $X_t = \psi_f(\MM{\tau})$. Then by Lemma~\ref{lem:psif-supermart} and the fact that $f$ is $(x^-,x^+)$-valid, $X$ is a supermartingale. By the optional stopping theorem, it follows that 
	\[
		\E(\psi_f(\MM{\tau})) = \E(X_\tau) \le \psi_f(X_0) = \exp\left(- \frac{\beta\phi_f(\MM{0})}{m_f} \right) \le \exp\left(- \frac{\beta x^+}{m_f} + \beta \right).
	\]
	Combining the two bounds on $\E(\psi_f(\MM{\tau}))$, we obtain $p \le e^{\beta}\exp(-\beta(x^+ - x^-)/m_f)$. The result then follows since $\beta \le 1/6$.
\end{proof}

\newcommand{\CC}{C_1}
\begin{lemma}\label{lem:gen-pot-b}
	Let $r > 1$, let $G=(V,E)$ be a connected graph on at least two vertices, and let $f\colon V\rightarrow\mathbb{R}_{\ge 0}$
	be a function that is not everywhere zero. Let $0 \le x^- <  x^+ \le \phif(V)$, and suppose that $f$ is $(x^-, x^+)$-valid and $x^+ - x^- > m_f$. Let $\MM$ be a Moran process on $G$ with fitness $r$ satisfying $\phi_f(\MM{0}) \ge x^+ - m_f$. Then 
	\[
		\Pr(\exists t \ge 0 \textnormal{ such that } \phi_f(\MM{t}) \le x^-) \le 
		{\left( {\tfrac{8r}{r-1}} \right)}^{1/2}
		|V|^2 e^{-\beta(x^+-x^-)/2m_f}.
	\]
	In particular, $\MM$ is at most this likely to go extinct.
\end{lemma}
\begin{proof}
	By Lemma~\ref{lem:gen-pot-a}, the probability that $\phif(\MM)$ reaches $[0,x^-]$ before reaching $[x^+,\phif(V)]$ is at most $2\exp(-\beta(x^+ - x^-)/m_f)$. Moreover, if $\phif(\MM{t}) < x^+$ and $\phif(\MM{t-1}) \ge x^+$, then $\phif(\MM{t}) \ge x^+ - m_f$ by the definition of $m_f$, and so we may again apply Lemma~\ref{lem:gen-pot-a} starting from~$t$. Thus by a union bound, for all $T \ge 0$, the probability that $\phif(\MM)$ reaches $[0,x^-]$ within $\floor{T}$ sojourns from the start or from $[x^+,\phif(V)]$ is at most $2\floor{T}e^{-\beta(x^+-x^-)/m_f}$. Therefore,
	\begin{equation}\label{eq:phif-small}
		\Pr\big(\exists t \in [0,\floor{T}] \text{ s.t.\@ } \phif(\MM{t}) \le x^- \big) \le 2Te^{-\beta(x^+-x^-)/m_f}\,.
	\end{equation}
	Moreover, by~\cite[Theorem~9]{DGMRSS2014}, the expected absorption time of $\MM$ from any state is at most $r|V|^4/(r-1)$. Thus by Markov's inequality, for all $T \ge 0$, we have
    \begin{equation}\label{eq:nofix}
        \Pr(M(\floor{T}) \notin \{0,V\}) \le \frac{r|V|^4}{(r-1)(\floor{T}+1)} < \frac{r|V|^4}{(r-1)T}\,.
    \end{equation}

	Now, take $T = {\left({\tfrac{r}{2(r-1)}} \right)}^{1/2}
	|V|^2e^{\beta(x^+-x^-)/2m_f}$. Then the result follows from~\eqref{eq:phif-small},~\eqref{eq:nofix} and a union bound.
\end{proof}

\section{The undirected phase transition}\label{sec:layers}

\subsection{Proof sketch}\label{sec:layers-sketch}

Let $G = (V,E)$ be a large connected graph, and let $r>1$. Our main objective in this section is to prove that, for any set $U \subseteq V$ with $\dr{U,V \setminus U}$ sufficiently small, a Moran process $\MM$ on $G$ with fitness $r$ and $|U \setminus \MM{0}| \le 1$ is overwhelmingly likely to fixate before $U$ empties of mutants. This is implied by Lemmas~\ref{lem:layered} and~\ref{lem:killers} (see Lemma~\ref{lem:magic}), which are crucial ingredients in our improved bound on absorption time (Theorem~\ref{thm:abs-time-intro}, proved in Section~\ref{sec:abs-time}). It will also turn out that the tools we develop along the way allow us to quickly prove the phase transition in fixation probability of undirected graphs (Theorem~\ref{thm:new-phase-trans}, proved in Section~\ref{sec:layered-app}). The remaining results of the section are purely ancillary.

We now sketch the proof of our results. In the sketch, we assume $U=V$; allowing $U$ and $V$ to be distinct does make the proof more complicated, but not by too much. Note that all the notation and definitions we introduce in this section are local in scope, so the reader may skip ahead to Section~\ref{sec:layered-proof} if they desire.

If $\phi(V)$ is large, then we immediately obtain a strong lower bound on fixation probability from Theorem~\ref{thm:potential}. (When $U\ne V$, our analogue of this bound will be Lemma~\ref{lem:killers}.) If $\phi(V)$ is small, however, then the bound of Theorem~\ref{thm:potential} is too weak to be useful. We instead construct a function $f\colon V\to\mathbb{R}_{\ge 0}$ to which we may apply Lemma~\ref{lem:gen-pot-b} to obtain a strong bound. We therefore require $f$ to satisfy the following properties:
\begin{enumerate}[(X1)]
	\item $f$ must be $(x^-,x^+)$-valid for some $0 \le x^- < x^+ \le \phif(V)$; in fact, we shall take $x^- = 0$; 
	\item $x^+$ must not be too small;
	\item $m_f = \max_{v \in V}\phif(v)$ must not be too large.
\end{enumerate}

We first consider a special case as a toy problem, in order to better motivate what follows. Let $n=|V|$, let $\lambda=(r+1)/2>1$ as in Definition~\ref{def:valid}, let $h = \ceil{\log_\lambda n}$, and suppose there is a partition $\calS = (S_0, \dots, S_h)$ of $V$ into disjoint sets satisfying the following properties:
\begin{enumerate}[(Y1)]
	\item for all $0 \le i,j \le h$ with $|i-j|\ge 2$, there are no edges between $S_i$ and $S_j$;
	\item $\phi(S_h) \ge 1$;
	\item for all $v \in V \setminus S_0$, $d(v) \ge \sqrt{n}$.
\end{enumerate}
Then we define $f\colon V\rightarrow \mathbb{R}_{\ge 0}$ by mapping each $v \in S_i$ to $\lambda^i/\lambda^h$. 

Recall from Definition~\ref{def:valid} that $f$ is valid for a set $X \subseteq V$ if
\begin{equation}\label{eqn:layers-sketch-valid}
	\sum_{\substack{(x,y)\in E(X,V\setminus X)\\ f(x) > \lambda f(y)}}
	\frac{f(x)}{d(x)\,d(y)}
	\leq \frac{r-1}{4r}
	\sum_{\substack{(x,y)\in E(X,V\setminus X)\\ f(x)\leq \lambda f(y)}}
	\frac{f(y)}{d(x)\,d(y)}.
\end{equation}  
If $f(x)>\lambda f(y)$ then, by our choice of~$f$, we must have $x\in S_i$ and $y\in S_j$ for some $i>j+1$.  Therefore, by property~(Y1), the left-hand term of~\eqref{eqn:layers-sketch-valid} is zero for all $X$, so $f$ is $(0,\phif(V))$-valid as in (X1). Moreover, we have $\phif(V) \ge \phif(S_h) = \phi(S_h) \ge 1$ by (Y2), which gives us the lower bound required by (X2). Finally, for all $v \in S_0$, we have $\phif(v) \le \lambda^{-h} \le 1/n$. For all $v \in V \setminus S_0$, by (Y3) we have $\phif(v) \le 1/d(v) \le 1/\sqrt{n}$. Thus $m_f \le 1/\sqrt{n}$, which gives us the upper bound required by (X3). Suppose $M$ is a Moran process on $G$ with fitness $r>1$ satisfying $|V \setminus M(0)| \le 1$. It is immediate that $\phif(M(0)) \ge \phif(V)-m_f$, so applying Lemma~\ref{lem:gen-pot-b} with $x^- = 0$ and $x^+ = \phif(V)$ yields 
\[
	\Pr(M\mbox{ fixates}) \ge 1-{\left( {\tfrac{8r}{r-1}} \right)}^{1/2}n^2 e^{-\beta(x^+-x^-)/2m_f} \ge 1-{\left( {\tfrac{8r}{r-1}} \right)}^{1/2}n^2e^{-\beta\sqrt{n}/2}.
\] 
We have therefore solved our toy problem.

Now let us return to our original situation: all we know about $G$ is that it is large and connected, and that $\phi(V)$ is small. We use our bound on $\phi(V)$ to find a large set $R \subseteq V$ (see Lemma~\ref{lem:makeR}), then use $R$ to construct a partition of $V$ satisfying similar properties to (Y1)--(Y3). For concreteness, suppose $\phi(V) \le 2$. Then $R$~satisfies the following ``clique-like'' properties:
\begin{enumerate}[(R1)]
	\item $\phi(R) \ge 1/2^6$;
	\item for all $X \subseteq R$ with $\phi(X) \le 1/2^7$, $\dr{X, R \setminus X} \ge \phi(X)/2^7$;
	\item for all $v \in R$, $d(v) \ge n/2^5$.
\end{enumerate}

We now set out the properties our partition will satisfy. In general, we cannot hope to satisfy (Y1), so we require an alternative condition that still implies (X1). Given a partition $\calS = (S_0, \dots, S_h)$, the exact condition that we need for $f$ to be $(0,x^+)$-valid (that is, for~\eqref{eqn:layers-sketch-valid} to hold) is that for all $X \subseteq V$ with $0 < \phif(X) < x^+$, we have
\begin{equation}\label{eqn:duno}
\sum_{j=0}^h \sum_{k=0}^{j-2} \lambda^{j-h}\dr{S_j \cap X, S_k \setminus X} \le \frac{r-1}{4r}\sum_{j=0}^h \sum_{k=j-1}^{h} \lambda^{k-h}\dr{S_j \cap X, S_k \setminus X}.
\end{equation}
Motivated by this, for all $X \subseteq V$ and all partitions $\calS = (S_0, \dots, S_h)$ of $V$, we define
\begin{align*}
B(X,\mathcal{S}) &= \sum_{j=0}^h \sum_{k=0}^{j-2} \dr{S_j \cap X, S_k \setminus X},\\
\Gamma(X,\mathcal{S}) &= \sum_{j=0}^h \sum_{k=j-1}^h \dr{S_j \cap X,S_k \setminus X}.
\end{align*}
We say that a partition $\calS$ is \emph{finished} if it satisfies the following properties:
\begin{enumerate}[(Z1)]
	\item Let $\Lambda = 4r\lambda^h/(r-1)$, and let $x^+$ be ``reasonably large''. For all $X \subseteq V$ with $0 < \phif(X) < x^+$, 
	$\Gamma(X,\calS) \ge \Lambda B(X,\calS)$;
	\item $\phi(R \setminus S_h) \le 1/2^9$;
	\item for all $v \in V \setminus S_0$, $d(v) \ge \sqrt{n}$.
\end{enumerate}
Note that: (Z1) implies (X1) by the above discussion; by (R1), (Z2) implies a weaker version of (Y2), and hence (X2); and (Z3) is the same as (Y3) and hence implies (X3). (See Definition~\ref{def:finished} for the true definition of finished, and Lemma~\ref{lem:true-potential} for the proof that a finished partition satisfies analogues of (X1)--(X3).)

Let $h$ be ``as large as possible'', and let $\delta_1 < \delta_2 < \dots < \delta_h$ be a rapidly increasing sequence with $\delta_1 \ge \sqrt{n}$ and $\delta_h \le n/2^5$; see Definition~\ref{def:consts} for the true values of $\delta_1, \dots, \delta_h$ and $h$. Let
\begin{align*}
S_0 &= \{v \in V \mid d(v) < \delta_1\},\\
S_i &= \{v \in V \mid \delta_i \le d(v) < \delta_{i+1}\}\mbox{ for all }i \in [h-1],\\
S_h &= \{v \in V \mid d(v) \ge \delta_h\}.
\end{align*}
Let $\calS = (S_0, \dots, S_h)$. Note that $\calS$ satisfies (Z3) by definition, and that by (R3) we have $R \subseteq S_h$ and so $\calS$ satisfies (Z2). In general, $\calS$ does not satisfy (Z1), but it does satisfy a weaker version of (Y1), which we shall use later; for all $i,j$ with $j \ge i+2$, by Lemma~\ref{lem:deg-drift} applied with $S_1 = S_i$, $d_1 = \delta_{i+1}$, $S_2 = S_j$ and $d_2 = \delta_j$, we have 
\begin{equation}\label{eqn:xyzzy}
\dr{S_i,S_j} \le \phi(S_i)\frac{\delta_{i+1}}{\delta_j} \le \phi(V)\frac{\delta_{i+1}}{\delta_{i+2}} \le 2\frac{\delta_{i+1}}{\delta_{i+2}},\mbox{ which is very small.}
\end{equation}
Thus instead of having no edges between $S_i$ and $S_j$, we have low drift.

Observe that for any set $X \subseteq V$, we can increase $\Gamma(X,\calS)$ and decrease $B(X,\calS)$ by moving $X$ downwards in the partition; for example, by moving $X \cap S_q$ from $S_q$ into $S_{q-1}$. Moreover, doing so does not decrease $\Gamma(Y,\calS)$ or increase $B(Y,\calS)$ for any set $Y \subseteq V$, and does not lower the minimum degree of any set $S_i$. It is therefore clear that we may use this to turn $\calS$ into a partition that satisfies (Z1) and (Z3). However, in the process we may empty $S_h$ entirely, thus violating (Z2). To avoid doing so, we make use of $R$ and~\eqref{eqn:xyzzy}.

We say that a set is \emph{safe} if $\phi(X \cap R) \le 1/2^8$. In constructing a finished partition from~$\calS$, we will only move safe sets downwards; thus we obtain a partition $\calS^\mathsf{fin} = (S_0^\mathsf{fin},\dots,S_h^\mathsf{fin})$ where all safe sets $X$ satisfy $\Gamma(X,\calS^\mathsf{fin}) \ge \Lambda B(X,\calS^\mathsf{fin})$. Suppose $\calS^\mathsf{fin}$ does not satisfy (Z2), so that $\phi(R \setminus S_h^\mathsf{fin}) > 1/2^9$. Then since we only moved safe sets, we must have passed through a partition $\calS' = (S_0',\dots,S_h')$ with $1/2^9 < \phi(R \setminus S_h') < 1/2^7$. Thus by (R2), we have $\dr{R \cap S_h', R \setminus S_h'} \ge 1/2^{16}$. But using~\eqref{eqn:xyzzy} and by being very careful about how we move sets downwards, we can ensure that this never happens, so that $\calS^\mathsf{fin}$ does satisfy (Z2). This is the most technically difficult part of the proof. We then take $x^+ = 1/2^9$ in (Z1) and define $f(v) = \lambda^i/\lambda^h$ for all $v \in S_i^\mathsf{fin}$ as before.  For any set $X \subseteq V$ with $\phif(X) \le 1/2^9$, by (Z2) we have 
\[
	\phi(X \cap R) \le \phi(X \cap R \cap S_h^\mathsf{fin}) + \phif(R \setminus S_h^\mathsf{fin}) \le \phif(X \cap R \cap S_h^\mathsf{fin}) + 1/2^9 \le \phif(X) + 1/2^9 \le 1/2^8,
\]
and so $X$ is safe and satisfies $\Gamma(X,\calS^\mathsf{fin}) \ge \Lambda B(X,\calS^\mathsf{fin})$ as required. Thus $\calS^\mathsf{fin}$ is finished, and so we may apply Lemma~\ref{lem:gen-pot-b} to $f$ to obtain a strong bound.

Note that while the sketch above is equivalent to our proof in spirit, we actually build up a finished partition step-by-step rather than starting with an $h$-level partition. This makes it much easier to give the details of the exact way in which we move safe sets downwards.

\subsection{The small-potential case}\label{sec:layered-proof}

We first embed a ``clique-like'' structure $R$ into $G$. The following lemma is folklore; recall that $G[R]$ is the subgraph induced by~$R$.

\begin{lemma}[{\cite[Proposition 1.2.2]{Diestel}}]\label{lem:deg-inc}
    For every graph $G=(V,E)$, there is a set $R\subseteq V$ such that
    $\delta(G[R])\geq \dbar(G)/2$.\pushQED{\qed}\qedhere
\end{lemma}

As in~\cite{Diestel}, Lemma~\ref{lem:deg-inc} may be proved by greedily deleting vertices of minimum degree. 
By assuming an ordering of~$V$ and always deleting the least of the
minimum-degree vertices according to that order, we may assume that
the procedure in the above proof uniquely defines the set~$R$.

\begin{definition}\label{def:R}
    Let $G=(V,E)$ be a graph and let $\emptyset\subset U\subseteq V$. Let $R(G,U)$~be the
    set given by applying Lemma~\ref{lem:deg-inc} to $G[U]$.
\end{definition}

\begin{lemma}\label{lem:makeR}
    Let $G=(V,E)$ be a connected graph on at least two vertices and let $\emptyset\subset U\subseteq V$.  If
    $\dr{U,V\setminus U}\leq 1/2|V|$, then $R = R(G,U)$ satisfies:
    \begin{enumerate}[(i)]
    \item $\phi(R) \geq 1/2^5\phi(U)$;
    \item for all $Y\subseteq R$ with $\phi(Y)\leq 1/2^6\phi(U)$, $\dr{Y,R\setminus Y}\geq \phi(Y)/2^6\phi(U)$; and
    \item for all $v\in R$, $d(v)\geq |U|/16\phi(U)$.
    \end{enumerate}
\end{lemma}
\begin{proof}
	Let $X = \{x \in U \mid d_U(x) \geq |U|/4\phi(U)\}$; thus 
	\begin{equation}\label{eqn:makeR-1}
		\overline{d}(G[U]) \ge \frac{1}{|U|}\sum_{x \in X}d_U(x) \ge \frac{1}{|U|} \cdot |X| \cdot \frac{|U|}{4\phi(U)}.
	\end{equation}
	 By Lemma~\ref{lem:drift-props}(\ref{d:ext-deg}) applied with $A=U$, each $u\in U$ sends at most $|V|\,\dr{U,V\setminus U}\,d(u)\leq d(u)/2$ edges outside~$U$, so $d(u) \le 2d_U(u)$. Therefore, for each $u\in U\setminus X$, $d(u) < |U|/2\phi(U)$, so $\phi(U) \ge \phi(U\setminus X) \ge |U\setminus X|\cdot 2\phi(U)/|U|$. This gives $|U\setminus X| \le |U|/2$ and, hence, $|X| \ge |U|/2$. It follows from~\eqref{eqn:makeR-1} that $\overline{d}(G[U]) \ge |U|/8\phi(U)$. Thus by Lemma~\ref{lem:deg-inc} and the definition of $R$,
	\begin{equation}\label{eqn:thingy}
		\delta(G[R]) \ge \frac{1}{2}\overline{d}(G[U]) \ge \frac{|U|}{16\phi(U)}.
	\end{equation}
	In particular, (iii) follows. Moreover, recall that every vertex $u \in U$ sends at most $d(u)/2$ edges outside $U$; hence for all $u \in U$, $d(u) \le 2|U|$. We therefore have $\phi(R) \ge |R|/2|U| \ge \delta(G[R])/2|U| \ge 1/2^5\phi(U)$, and so (i) holds.
	
	Towards proving (ii), let $Y \subseteq R$ with $\phi(Y) \le 1/2^6\phi(U)$.  Recall that every vertex in $U$ has degree at most $2|U|$, so $|Y| \le 2|U|\phi(Y) \le |U|/2^5\phi(U)$. Therefore, by~\eqref{eqn:thingy}, $|Y| \le \delta(G[R])/2$, which means that every $x \in Y$ must have at least $\delta(G[R])/2$ neighbours in $R \setminus Y$. Putting all of this together gives
	\begin{align*}
		\dr{Y, R \setminus Y} &= \sum_{x \in Y}\Big(\frac{1}{d(x)}\sum_{y \in N(x) \cap R \setminus Y} \frac{1}{d(y)} \Big)
		\ge \sum_{x \in Y}\Big(\frac{1}{d(x)} \cdot \frac{\delta(G[R])}{2} \cdot \frac{1}{2|U|} \Big)\\
		&= \frac{\phi(Y)\delta(G[R])}{4|U|}
		\ge \frac{\phi(Y)}{2^6\phi(U)},
	\end{align*}
	which establishes (ii).
\end{proof}

We aim to find a partition $\calS$ of $U$ with certain properties. Before we can state these properties, we need the following definitions.

\begin{definition}\label{def:consts}
    For the rest of this section, we will consider a connected graph $G(V,E)$ on at least two vertices
    and a non-empty set $U\subseteq V$ such that $\dr{U,V\setminus U}\leq 1/2|V|$ (so that Lemma~\ref{lem:makeR} applies). We will also assume that $|U| \ge 2$ and write $R = R(G,U)$.
    We will use the following real numbers, where $r>1$ is the fitness of
    the mutants in the Moran process under consideration. (Here $\lambda$ and $\beta$ are recalled from Section~\ref{sec:discount-pot}.)
    \begin{alignat*}{4}
        \alpha  &= 1/2^7\phi(U) &\qquad
        s		&= (\log_r |U|)^{1/3} &\qquad
        \nlevs  &= \ceil{s/2} &\qquad
        \Lambda &= \frac{8r^{\nlevs+1}}{r-1}\\
        D       &= \frac{\alpha}%
                        {2^7\phi(U)(\Lambda h)^{2\nlevs}}&
        \lambda &= \frac{r+1}{2}&
        \beta &= \frac{r-1}{6r+2}
    \end{alignat*}
    and
    \[
        \delta_i = |U|\left(\frac{D}{\phi(U)}\right)^{\nlevs+1-i} \mbox{ for all }i \in [\nlevs].
    \]
\end{definition}

\begin{remark}\label{rem:consts}\label{rem:constants}
    It is immediate from Lemma~\ref{lem:drift-props}(\ref{d:phi})
    (taking $A=U$)
    that $\phi(U)\geq 1/2$. Since we assume $|U| \ge 2$,
    we have $\nlevs\geq 1$. These bounds also give $0<\alpha\leq 1/2^6$,
    $\Lambda>8$, $D < \alpha/2^6 \le 2^{-12}$ 
    and $0 < \delta_1 < \delta_2 < \dots < \delta_h < |U|$. 
    Since $r>1$ we also have
 $\lambda > 1$.
\end{remark}

\begin{definition}\label{def:goodbad}
	Let $G=(V,E)$ be a connected graph on at least two vertices, let $q \ge 0$, and let $\calS = (S_0, \dots, S_q)$
	partition~$U$ for some $U\subseteq V$.  For all
	$X\subseteq U$, define
	\begin{align*}
	B(X,\mathcal{S}) &= \sum_{j=0}^q \sum_{k=0}^{j-2} \dr{S_j \cap X, S_k \setminus X},\\
	\Gamma(X,\mathcal{S}) &= \sum_{j=0}^q \sum_{k=j-1}^{q} \dr{S_j \cap X,S_k \setminus X}.
	\end{align*}
\end{definition}

(Recall from Section~\ref{sec:prelim} that $S_{-1} = \emptyset$.) Note that for any choice of $X$ and $\calS$, we have $B(X,\calS) + \Gamma(X,\calS) = \dr{X, U \setminus X}$. As we will see in the proof of Lemma~\ref{lem:true-potential}, if $U=V$ and $B(X,\calS) \le \Lambda\Gamma(X,\calS)$, then the function $f\colon V\rightarrow\mathbb{R}_{\ge 0}$ given by $f(v) = \lambda^i/\lambda^h$ for all $v \in S_i$ is valid for $X$.

\begin{definition}\label{def:ell-good-set}
    Let $\calS = (S_0, \dots, S_q)$ partition~$U$ for some $q \ge 0$. For all $i \ge 0$, a set $X \subseteq U$ is \emph{$i$-good in~$\calS$} if $\Gamma(X,\calS) \geq i\Lambda B(X,\calS)$, and \emph{safe} if $\phi(X \cap R) \le \alpha$.
\end{definition}

\begin{definition}\label{def:finished}
	A partition $\calS = (S_0, \dots, S_\nlevs)$
	of~$U$ is \emph{finished} if:
	\begin{enumerate}[(F1)]
		\item every safe set $X \subseteq U$ is $1$-good in $\calS$;
		\item every safe set $X \subseteq U$ satisfies $\dr{X,U \setminus X}\ge \phi(X \setminus S_0)D/\phi(U)$;
		\item $\phi(R \setminus S_h) \le \alpha/2$;
		\item for all $v \in U \setminus S_0$, $d(v) \ge \delta_1$.
	\end{enumerate}
\end{definition}

We will often abuse notation by saying that a safe set 
$X \subseteq U$ \emph{satisfies (F1)}
to mean that $X$ is $1$-good in~$\calS$
and by saying that $X$
  \emph{satisfies (F2)} to mean that
$\dr{X,U \setminus X}\ge \phi(X \setminus S_0)D/\phi(U)$. 

Given a finished partition $\mathcal{S}$ with $\nlevs$ levels, we will be able to construct a weighting $f\colon V\rightarrow\mathbb{R}_{\ge 0}$ to which we may usefully apply Lemma~\ref{lem:gen-pot-b} by taking $f(v) = \lambda^i/\lambda^{\nlevs}$ for all $v \in S_i$; see Lemma~\ref{lem:true-potential}. As in our sketch proof (see Section~\ref{sec:layers-sketch}), we will use (F1) to prove validity of~$f$, we will use (F3) to bound $\phif(U)$ below, and we will use (F4) to bound $m_f$ above. Note that we only need (F2) when $U \ne V$, in which case it will be used alongside~(F1) to prove validity of $f$. We next sketch the procedure we will use to construct a finished partition.

\begin{definition}\label{def:lift}
	Let $\calS = (S_0, \dots, S_q)$ partition~$U$ for some $q \ge 0$. If $q \le h$, the \emph{split} of~$\calS$ is the partition $(S_0, \dots, S_{q-1}, S_q', S_{q+1}')$, where $S_q' = \{v \in S_q \mid d(v) < \delta_{q+1}\}$ and $S_{q+1}' = S_{q} \setminus S_q'$. If $q \ge 1$, then for all $X \subseteq U$, the \emph{drop of~$X$ (in~$\calS$)} is the partition $(S_0, \dots, S_{q-2}, S_{q-1} \cup (S_q \cap X), S_q \setminus X)$ formed by moving $X\cap S_q$ into~$S_{q-1}$. 
\end{definition}

Let $\calS^1$ be the partition of $U$ formed by splitting the trivial partition $\calS^0 = (U)$ of $U$. Note that we could form the initial partition used in the sketch proof (Section~\ref{sec:layers-sketch}) by applying $h-1$ splitting operations to $\calS^1$. Instead, we greedily drop safe sets which fail to satisfy (F2) until no more exist. 
Let the resulting partition be called $\calS^{1+} = (S^{1+}_0,S^{1+}_1)$.
We then split $\calS^{1+}$   to get a new partition $\calS^{2} = (S^2_0,S^2_1,S^2_2)$.  
Next, we repeat the following  process: given a partition 
$\calS^q = (S^q_0, \dots, S^q_q)$ with $q \ge 2$, 
we greedily drop safe sets which fail to be $(h+1-q)$-good until no more exist.
If $q<h$, we split the resulting partition and  repeat the process again. If $q=h$,
we will show that the resulting partition is finished. To do so, we maintain the following invariant.

\begin{definition}\label{def:ell-good}
	For $q\in[h]$, a partition $\calS = (S_0, \dots, S_q)$
	of~$U$ is \emph{good} if, writing $T(\calS) = \{v \in S_{q-1} \mid d(v) \ge \delta_q\}$ and $\calS^- = (S_0, \dots, S_{q-2}, S_{q-1} \cup S_q)$:
	\begin{enumerate}[(G1)]
		\item every safe set $X \subseteq U$ is $(h+2-q)$-good in~$\calS^-$;
		\item if $q \ge 2$, every safe set $X \subseteq U$ satisfies $\dr{X,U \setminus X} \ge \phi(X \setminus S_0)D/\phi(U)$;
		\item $\phi(R \setminus S_q) \le \alpha/2$;
		\item for all $i \in [q]$ and all $v \in S_i$, $d(v) \ge \delta_i$;
		\item $\dr{S_{\le q-2}, S_q \cup T(\calS)} \le (\Lambda h)^{2(q-1)}D$;
		\item if $q=1$, then $\dr{S_q, T(\calS)} \le D\phi(T(\calS))/\phi(U)$; if instead $q\ge 2$, then $\dr{S_q, T(\calS)} \le h^2\Lambda \dr{S_{q-2}, T(\calS)}$.
	\end{enumerate}
\end{definition}

Note that   if $q\geq 2$
then (G2) is identical to (F2).
If $q=h$ then (G3) is identical to (F3).  If $q\geq 1$ then
(G4) trivially implies (F4).
Condition (G1) is related to condition (F1), although they are not the
same. 
By (G4), $T(\calS)$ tracks everything we have added to $S_{q-1}$ by dropping sets since our last splitting operation. As in the sketch proof of Section~\ref{sec:layers-sketch}, we will show that (G3) is maintained by upper-bounding the drift between $R \cap S_q$ and $R \setminus S_q$; this is the purpose of (G5) and (G6). We first prove three ancillary lemmas.

\begin{lemma}\label{lem:R-high-deg}
	Let $G=(V,E)$ be a connected graph. Let $U\subseteq V$ with $\dr{U, V \setminus U} \le 1/2|V|$ and $|U|\geq 2$, and let $R=R(G,U)$. Then for all $i \in [h]$, every vertex in $R$ has degree at least~$\delta_i$.
\end{lemma}
\begin{proof}
	First note that by Remark~\ref{rem:consts} it suffices to prove the result for $i=h$. We have $\delta_h = |U|D/\phi(U)$, and every vertex in $R$ has degree at least $|U|/16\phi(U)$ by Lemma~\ref{lem:makeR}(iii). Since $D < 1/2^5$ by Remark~\ref{rem:consts}, the result follows.
\end{proof}

\begin{lemma}\label{lem:split-drift}
	Let $G=(V,E)$ be a connected graph and let $U \subseteq V$ with $\dr{U,V \setminus U} \le 1/2|V|$ and $|U|\geq 2$. Let $\calS = (S_0, \dots, S_q)$ be a partition of $U$ with $q \ge 2$. Then for all $Y \subseteq U$, 
	\begin{align*}	
		\Gamma(Y,\calS^-) &= \Gamma(Y,\calS) + \dr{S_q \cap Y,S_{q-2} \setminus Y},\\
		B(Y,\calS^-) &= B(Y,\calS) - \dr{S_q \cap Y,S_{q-2} \setminus Y}.
	\end{align*}
\end{lemma}
\begin{proof}
	Observe that since $\Gamma(Y,\calS^-) + B(Y,\calS^-) = \dr{Y,U \setminus Y} = \Gamma(Y,\calS) + B(Y,\calS)$, the two parts of the lemma are equivalent and it suffices to prove the first. Write $\calS^- = (S_0^-, \dots, S_{q-1}^-)$, and recall that $S_i^- = S_i$ for all $i < q-1$ and $S_{q-1}^- = S_{q-1} \cup S_q$. Then by the definition of $\Gamma$, we have
	\begin{align*}
	\Gamma(Y,\calS^-)
	&= \sum_{j=0}^{q-1}\sum_{k=j-1}^{q-1}
	\dr{S^-_j\cap Y, S^-_k\setminus Y}\\
	&= \sum_{j=0}^{q-1}\sum_{k=j-1}^{q-1} \dr{S_j\cap Y, S_k\setminus Y}
	+ \sum_{j=0}^q\dr{S_j\cap Y, S_q\setminus Y}
	+ \sum_{k=q-2}^{q-1}\dr{S_q\cap Y, S_k\setminus Y}\\
	&= \Gamma(Y, \calS) + \dr{S_q\cap Y, S_{q-2}\setminus Y}\,.\qedhere
	\end{align*}
\end{proof}

\begin{lemma}\label{lem:goodness-bound}
	Let $G=(V,E)$ be a connected graph and let $U \subseteq V$ with $|U| \ge 2$. Let $\calS = (S_0, \dots, S_q)$ be a partition of $U$ with $q \ge 2$. Suppose that for some integer $i \ge 2$, some 
non-empty	
	set $Y\subseteq U$ is $i$-good in $\calS^-$ but not $(i-1)$-good in $\calS$. Then $\Gamma(Y,\calS) \le i^2\Lambda\dr{S_{q-2}, S_q \cap Y}$.
\end{lemma}
\begin{proof}
	Since $Y$ is not $(i-1)$-good in $\calS$, we have
	\begin{align*}
		\Gamma(Y,\calS) &\le (i-1)\Lambda B(Y,\calS)
		= (i-1)\Lambda\big(\dr{Y,U \setminus Y} - \Gamma(Y,\calS)\big)\\
		&= (i-1)\Lambda\big(\Gamma(Y,\calS^-) + B(Y,\calS^-) - \Gamma(Y,\calS)\big).
	\end{align*}
	Since $Y$ is $i$-good in $\calS^-$, it follows that
	\begin{align*}
		\frac{\Gamma(Y,\calS)}{(i-1)\Lambda} &\le \left(1 + \frac{1}{i\Lambda}\right)\Gamma(Y,\calS^-) - \Gamma(Y,\calS)\\
		&= \left(1+\frac{1}{i\Lambda} \right)\big(\Gamma(Y,\calS^-) - \Gamma(Y,\calS)\big)+\frac{\Gamma(Y,\calS)}{i\Lambda}.
	\end{align*}
	Multiplying both sides by $i(i-1)\Lambda$ and rearranging yields
	\begin{align*}
		\Gamma(Y,\calS) \le \big(i(i-1)\Lambda + i-1 \big)\big(\Gamma(Y,\calS^-) - \Gamma(Y,\calS)\big).
	\end{align*}
	Since $\Lambda > 1$ (see Remark~\ref{rem:constants}), this implies 
	\begin{equation*} 
		\Gamma(Y,\calS) < (i+1)(i-1)\Lambda\big(\Gamma(Y,\calS^-) - \Gamma(Y,\calS)\big) < i^2\Lambda\big(\Gamma(Y,\calS^-) - \Gamma(Y,\calS)\big).
	\end{equation*}
	The result therefore follows by Lemma~\ref{lem:split-drift}.
\end{proof}

We next prove the first of two lemmas that hold the argument together: dropping safe sets which violate (F2) (if $q=1$) or which fail to be $(h+1-q)$-good (if $q\ge 2$) preserves goodness. 

\begin{lemma}\label{lem:drop-safe}
	Let $G=(V,E)$ be a connected graph. Let $U\subseteq V$ with $\dr{U, V \setminus U} \le 1/2|V|$ and $|U| \ge 2$, and let $R=R(G,U)$. Let $q\in[h]$ and suppose $\calS = (S_0, \dots, S_q)$ is a good partition of $U$. Let $Y \subseteq U$ be safe. Suppose that either $q=1$ and $\dr{Y,U \setminus Y} < \phi(Y \setminus S_0)D/\phi(U)$ or that $q \ge 2$ and $Y$ is not $(h+1-q)$-good in $\calS$. Form $\calS' = (S_0', \dots, S_q')$ from $\calS$ by dropping~$Y$. Then $\calS'$ is good, and $T(\calS') \supset T(\calS)$.
\end{lemma}
\begin{proof}

We first show that $T(\calS') \supset T(\calS)$. 
Since $\calS$ is a good partition of $U$, it satisfies (G4), so every vertex of $S_q \cap Y$ has degree at least $\delta_q>\delta_{q-1}$
and therefore $T(\calS') = T(\calS) \cup (S_q \cap Y)$. 
We next consider two cases and show that, in either case, 
$S_q \cap Y$ is non-empty so
$T(\calS)$ is a strict subset of $T(\calS')$, as required.
\begin{itemize}
\item
If $q=1$, then $\dr{Y,U \setminus Y} < \phi(Y \setminus S_0)D/\phi(U)$, so 
$\phi(Y \setminus S_0) > 0$.
Since 
$S_1 \cap Y = Y \setminus S_0$, we conclude that $S_1 \cap Y$ is non-empty. 
\item Suppose that  $q \ge 2$. 
Suppose for contradiction that  $S_q \cap Y = \emptyset$.
By (G1), $Y$ is $(h+2-q)$-good in~$\calS^-$.
But   by Lemma~\ref{lem:split-drift}, $\Gamma(Y,\calS^-) = \Gamma(Y,\calS)$ and $B(Y,\calS^-) = B(Y,\calS)$,
so $Y$ is also $(h+2-q)$-good in~$\calS$, hence it is $(h+1-q)$-good in~$\calS$,
contradicting the definition of~$Y$ in the statement of the lemma. 
\end{itemize} 	

Having shown that $T(\calS') \supset T(\calS)$, it remains to show that $\calS'$ is good.
To do this, we show that  $\calS'$ has each of the six properties required by Definition~\ref{def:ell-good}. 
	
	\begin{enumerate}
		\item[(G1)] Since $\calS^- = (\calS')^-$ and $\calS$ satisfies (G1), $\calS'$ must also satisfy (G1).
		\item[(G2)] If $q=1$ then $\calS'$ vacuously satisfies (G2); if instead $q \ge 2$, then $S_0' = S_0$, so $\calS'$ satisfies (G2) since $\calS$ does.
		\item[(G4)] Since $\calS$ satisfies (G4) and $\delta_q > \delta_{q-1}$, $\calS'$ must also satisfy (G4) by the definition of dropping.
		\item[(G5)] We have $S_{\le q-2}' = S_{\le q-2}$ and $S_q' \cup T(\calS') = (S_q \setminus Y) \cup T(\calS) \cup (S_q \cap Y) = S_q \cup T(\calS)$, so $\calS'$ satisfies (G5) since $\calS$ does. 
		\item[(G6)] We have
		\begin{align*}
		\dr{S_q', T(\calS')}
		&= \dr{S_q \setminus Y, T(\calS) \cup (S_q \cap Y)}
		= \dr{S_q\setminus Y, T(\calS)} + \dr{S_q \setminus Y, S_q \cap Y}\\
		&\le \dr{S_q, T(\calS)} + \Gamma(Y,\calS)\,.\tagme\label{eq:DrST}
		\end{align*}
		If $q=1$, then since $\calS$ satisfies (G6) and $\Gamma(Y,\calS) \le \dr{Y,U \setminus Y} <  \phi(Y \setminus S_0)D/\phi(U)$ by hypothesis, it follows that 
		\begin{align*}
		\dr{S_1', T(\calS')} 
		&\le \frac{D\phi(T(\calS))}{\phi(U)} + \frac{D\phi(Y \setminus S_0)}{\phi(U)}
		= \frac{D\big(\phi(T(\calS)) + \phi(S_1\cap Y)\big)}{\phi(U)}
		= \frac{D\phi(T(\calS'))}{\phi(U)},
		\end{align*}
		and so $\calS'$ satisfies (G6) as required. Suppose instead $q\ge 2$. Then $Y$ is $(h+2-q)$-good in $\calS^-$ by (G1), and $Y$ is not $(h+1-q)$-good in $\calS$ by hypothesis. Thus by Lemma~\ref{lem:goodness-bound}, applied with~$i = h+2-q \le h$, $\Gamma(Y,\calS) \le h^2\Lambda\dr{S_{q-2}, S_q \cap Y}$. Since $\calS$ satisfies (G6), it follows from~\eqref{eq:DrST} that
		\[
		\dr{S_q', T(\calS')} \le h^2\Lambda\dr{S_{q-2}, T(\calS)} + h^2\Lambda\dr{S_{q-2}, S_q \cap Y} = h^2\Lambda\dr{S_{q-2}, T(\calS')}.
		\]
		Thus once again $\calS'$ satisfies (G6), as required.
		\item[(G3)] Observe that, since $Y$ is safe and $\calS$ satisfies (G3), we have $\phi(R \setminus S_q') \le \phi(R \setminus S_q) + \phi(R \cap Y) < 2\alpha = 1/2^6\phi(U)$. Thus by Lemma~\ref{lem:makeR}(ii), applied with $Y = R \setminus S_q'$, 
		\[
		\dr{R \setminus S_q', R \cap S_q'} \ge \phi(R \setminus S_q')/2^6\phi(U).
		\]
		Moreover, by Lemma~\ref{lem:R-high-deg}, every vertex in $R \setminus S_q'$ has degree at least $\delta_q$. Thus $R \setminus S_q' \subseteq S_{\le q-2}' \cup T(\calS')$, and
		\[
		\phi(R \setminus S_q') \le 2^6\phi(U)\dr{R \setminus S_q',R \cap S_q'} \le 2^6\phi(U)\big(\dr{S_{\le q-2}', S_q'} + \dr{T(\calS'), S_q'}\big).
		\]
		If $q=1$ then, since $\calS'$ satisfies (G6) and $S_{\le q-2}' = \emptyset$, it follows that
		\[
		\phi(R \setminus S_q') \le 2^6\phi(U)D < 2^6\phi(U)(h\Lambda)^{2h}D = \alpha/2.
		\]
		If $q \ge 2$, since $\calS'$ satisfies (G6) and (G5), it follows that
		\begin{align*}
		\phi(R \setminus S_q') &\le 2^6\phi(U)\big(\dr{S_{\le q-2}', S_q'} + h^2\Lambda\dr{S_{q-2}',T(\calS')}\big)\\
		&\le 2^6\phi(U)h^2\Lambda\dr{S_{\le q-2}', S_q' \cup T(\calS')}\\
		&\le 2^6\phi(U) (h\Lambda)^{2q}D \le 2^6\phi(U)(h\Lambda)^{2h}D = \alpha/2.
		\end{align*}
		Thus in both cases we have $\phi(R \setminus S_q') \le \alpha/2$, and so $\calS'$ satisfies (G3) as required.\qedhere
	\end{enumerate}
\end{proof}

We now show that the other major step of our algorithm, splitting a partition, also preserves goodness under the circumstances in which we do it.

\begin{lemma}\label{lem:split-safe}
	Let $G=(V,E)$ be a connected graph. Let $U\subseteq V$ with $\dr{U, V \setminus U} \le 1/2|V|$ and $|U| \ge 2$, and let $R=R(G,U)$. Let $q\in[h-1]$ and suppose $\calS = (S_0, \dots, S_q)$ is a good partition of $U$. Suppose moreover that either $q=1$ and every safe $Y \subseteq U$ satisfies $\dr{Y,U \setminus Y} \ge \phi(Y \setminus S_0)D/\phi(U)$, or that $q \ge 2$ and every safe $Y \subseteq U$ is $(h+1-q)$-good in~$\calS$. Form $\calS' = (S_0', \dots, S_{q+1}')$ by splitting $\calS$. Then $\calS'$ is good.
\end{lemma}
\begin{proof}
	We show that $\calS'$ has each of the six properties required by Definition~\ref{def:ell-good} for a partition with $q+1$ parts.
	\begin{enumerate}[(G1)]
		\item Observe that by the definition of splitting, $(\calS')^- = \calS$. If $q=1$ then every safe set $X \subseteq U$ is vacuously $h$-good in $\calS$, since $B(X,\calS) = 0$; if $q\ge 2$ then every set $X \subseteq U$ is $(h+1-q)$-good in $\calS$ by hypothesis. Thus in either case, $\calS'$ satisfies (G1).
		\item We have $S_0' = S_0$. Thus, if $q=1$ then $\calS'$ satisfies (G2) since (by hypothesis) every safe $Y\subseteq U$ has $\dr{Y,U \setminus Y} \ge \phi(Y \setminus S_0)D/\phi(U)$.  If $q \ge 2$ then $\calS'$ satisfies (G2) since $\calS$ does.
		\item By Lemma~\ref{lem:R-high-deg}, every vertex in $R$ has degree at least $\delta_{q+1}$, so $S_{q+1}' \cap R = S_q \cap R$. Thus $\calS'$ satisfies (G3) since $\calS$ does. 
		\item By the definition of splitting and the fact that $\calS$ satisfies (G4), $\calS'$ also satisfies (G4).
		\item Since $S_{\le q-1}' = S_{\le q-1}$, $S_{q+1}' \subseteq S_q$ and $T(\calS') = \emptyset$, we have
		\begin{align}\nonumber
		\dr{S_{\le q-1}', S_{q+1}' \cup T(\calS')} &= \dr{S_{\le q-2}, S_{q+1}'} + \dr{T(\calS), S_{q+1}'} + \dr{S_{q-1} \setminus T(\calS), S_{q+1}'}\\\label{eqn:G5-split}
		&\le \dr{S_{\le q-2}, S_q} + \dr{T(\calS), S_q} + \dr{S_{q-1} \setminus T(\calS), S_{q+1}'}.
		\end{align}
		Every vertex in $S_{q-1} \setminus T(\calS)$ has degree less than $\delta_q$, and every vertex in $S_{q+1}'$ has degree at least $\delta_{q+1}$. Thus by Lemma~\ref{lem:deg-drift},
		\[
		\dr{S_{q-1} \setminus T(\calS),S_{q+1}'} \le \frac{\delta_q\phi(U)}{\delta_{q+1}} = D.
		\]
		Moreover, since $\calS$ satisfies (G6), if $q=1$ we have $\dr{T(\calS), S_q} \le D$ and if $q \ge 2$ we have $\dr{T(\calS), S_q} \le h^2\Lambda\dr{S_{q-2}, T(\calS)}$. In all cases, it follows from~\eqref{eqn:G5-split} that
		\begin{align*}
		\dr{S_{\le q-1}', S_{q+1}' \cup T(\calS')} &\le \dr{S_{\le q-2}, S_q} + h^2\Lambda\dr{S_{q-2},T(\calS)} + 2D\\
		&\le h^2\Lambda\dr{S_{\le q-2}, S_q \cup T(\calS)} + 2D.
		\end{align*}
		Since $\calS$ satisfies (G5), it follows that
		\[
		\dr{S_{\le q-1}', S_{q+1}' \cup T(\calS')} \le h^2\Lambda\cdot h^{2(q-1)}\Lambda^{2(q-1)}D + 2D < h^{2q}\Lambda^{2q}D.
		\]
		For the final inequality we use the fact that $h \ge 1$ and $\Lambda \ge 2$ by Remark~\ref{rem:consts}. Thus $\calS'$ satisfies (G5) as required.
		\item We have $T(\cal{S}') = \emptyset$, so $\calS'$ satisfies~(G6).\qedhere
	\end{enumerate}
\end{proof}

We now repeatedly apply Lemmas~\ref{lem:drop-safe} and~\ref{lem:split-safe} to show that our construction yields a finished partition.

\begin{lemma}\label{lem:finished-partition}
	There exists $u_0\geq 2$ such that, for all connected graphs $G=(V,E)$, every $U\subseteq V$ with $|U| \ge u_0$ and $\dr{U, V \setminus U} \le 1/2|V|$ has a finished partition $\calS$.
\end{lemma}
\begin{proof}
	Let $R=R(G,U)$, and take $u_0$ large enough to ensure $h \ge 2$. 
	We take our initial partition to be $\calS^1=(S^1_0,S^1_1)$, where $S^1_0 = \{v \in U \mid d(v) < \delta_1\}$ and $S^1_1 = \{v \in U \mid d(v) \ge \delta_1\}$. We claim that $\calS^1$ is a good partition. 
First, $(\calS^1)^- = (U)$ and every set $X \subseteq U$ satisfies $B(X,(\calS^1)^-) = 0$, so (G1) is satisfied. 
Since $q=1$, (G2) is vacuous. 
By Lemma~\ref{lem:R-high-deg}, 
$R \subseteq S^1_1$, 
so (G3) is satisfied. (G4) is satisfied by construction. (G5) is vacuous since
$S^1_{\leq q-2} = S^1_{\le -1} = \emptyset$. Finally, (G6) is satisfied since $T(\calS^1) = \emptyset$.
	
We now form $\calS^{1+}$ from $\calS^1$ by the following greedy process. 
During each step of the greedy process we have a ``current partition'' $(S_0',S_1')$.
Originally, this is $\calS^1$.
The process from $(S_0',S_1')$ is as follows. If there exists a safe set $Y \subseteq U$ such that $\dr{Y, U \setminus Y} < \phi(Y \setminus S_0') D/\phi(U)$, we drop $Y$ and continue from the resulting partition $(S_0' \cup (S_1' \cap Y), S_1' \setminus   Y)$; otherwise we stop and set $\calS^{1+} = (S_0',S_1')$. By Lemma~\ref{lem:drop-safe}, goodness is preserved throughout this process, and $T$ is strictly increasing so the process must terminate. Thus $\calS^{1+} = (S_0^{1+},S_1^{1+})$ is a good partition with the property that, for all safe $Y \subseteq U$, $\dr{Y, U \setminus Y} \ge \phi(Y \setminus S_0^{1+}) D/\phi(U)$. We then form $\calS^2$ by splitting $\calS^{1+}$, and note that $\calS^2$ is good by Lemma~\ref{lem:split-safe}.
	
	We now proceed  iteratively. Given a good partition $\calS^q = (S_0^q, \dots, S_q^q)$ with $2 \le q \le h$, we form $\calS^{q+}$ from $\calS^q$ by the following greedy process. Given a partition $\calS'$, if there exists a safe set $Y \subseteq U$ such that $Y$ is not $(h+1-q)$-good in $\calS'$, we drop $Y$; otherwise we stop and set $\calS^{q+} = \calS'$. By Lemma~\ref{lem:drop-safe}, goodness is preserved throughout this process, and $T$ is strictly increasing so the process must terminate. Thus $\calS^{q+}$ is a good partition with the property that every safe set $Y \subseteq U$ is $(h+1-q)$-good in $\calS^{q+}$. If $q<h$, we then form $\calS^{q+1}$ by splitting $\calS^{q+}$, and note that $\calS^{q+1}$ is good by Lemma~\ref{lem:split-safe}.
	
	We claim that $\calS^{h+}$ is a finished partition of $U$. Indeed, (F1) holds by construction; (F2) holds by (G2), where $h\geq 2$ by our choice of $u_0$; (F3) holds by (G3); and (F4) holds by (G4).
\end{proof}

We now use the finished partition guaranteed by Lemma~\ref{lem:finished-partition} to construct a potential function to which we may productively apply Lemma~\ref{lem:gen-pot-b}.

\begin{lemma}\label{lem:true-potential}
	There exists $u_0 \ge 2$ such that, for all connected graphs $G = (V,E)$ and $U \subseteq V$ with $|U| \ge u_0$, $\dr{U,V \setminus U} \le 1/2|V|$, and $\phi(U) \le \lambda^{s/5}$, there is a function $f\colon V\rightarrow\mathbb{R}_{\ge 0}$ such that:
	\begin{enumerate}[(i)]
		\item $\phif(V \setminus U) = 0$;
		\item $\phif(U) \ge \alpha/2$;
		\item $m_f \le 1/\lambda^h$;
		\item $f$ is $(\alpha/4,\alpha/2)$-valid; moreover, if $U=V$, then $f$ is $(0,\alpha/2)$-valid (recall Definition~\ref{def:valid}).
	\end{enumerate}
\end{lemma}
\begin{proof}
	Let $\mathcal{S} = (S_0, \dots, S_\nlevs)$ be the finished partition of~$U$ guaranteed by Lemma~\ref{lem:finished-partition}. Let $R=R(G,U)$ as usual. For all $v \in V$, we define
	\[
		f(v) = \begin{cases}
			0 & \mbox{if }v \in V \setminus U,\\
			\lambda^i/\lambda^h & \mbox{if }v \in S_i\mbox{ for }0\le i \le h.
		\end{cases}
	\]
    Thus (i) is immediate.  By Remark~\ref{rem:constants} $\lambda>1$, so $\lambda^i/\lambda^h$ increases with increasing $i$.
    
    We next prove (ii). We have 
    \[
    	\phif(U) \ge \phif(R \cap S_h) = \phi(R \cap S_h) = \phi(R) - \phi(R \setminus S_h).
    \]
    By Lemma~\ref{lem:makeR}(i) and (F3), it follows that $\phif(U) > 1/2^5\phi(U) - \alpha/2 = 7\alpha/2$, and so (ii) holds.
    
    We next prove (iii). For all $v \in V \setminus U$, $\phif(v) = 0 < 1/\lambda^h$. For all $v \in S_0$, $\phif(v) \le f(v) = 1/\lambda^h$. Finally, by (F4), for all $v \in U \setminus S_0$, $\phif(v) \le 1/\delta_1$. Thus to prove $m_f \le 1/\lambda^h$, it suffices to prove $\delta_1/\lambda^h \ge 1$. We have
    \begin{align}\label{eqn:quendor}
    	\frac{\delta_1}{\lambda^h} = \frac{|U|D^h}{\lambda^h\phi(U)^h} = |U|\Big(\frac{1}{2^{14}\lambda \phi(U)^3(\Lambda h)^{2h}} \Big)^h = |U|\Bigg(\frac{1}{2^{14}\lambda}\Big(\frac{r-1}{8rh}\Big)^{2h}\cdot\frac{1}{\phi(U)^3r^{2h^2}} \Bigg)^h.
    \end{align}
    By hypothesis, $\phi(U)^3 \le \lambda^{3s/5} < r^{3s/5}$. Moreover, $r^{2h^2} \ge r^{s^2/2}$. Thus when $|U|$ is sufficiently large, the $r^{2h^2}$ term in the denominator of the right-hand-side of~\eqref{eqn:quendor} dominates and we have
    \begin{align}\label{eqn:frobozz}
    	\frac{\delta_1}{\lambda^h} \ge |U|r^{-3h^3} \ge |U|r^{-3(s/2+1)^3} > |U|r^{-s^3/2} = \sqrt{|U|} > 1.
    \end{align}
    Thus (iii) holds.
        
    Finally, we prove (iv). 
	Let $X \subseteq V$ with
    $0 < \phif(X) < \alpha/2$. As in the proof of
    Lemma~\ref{lem:psif-supermart}, we write $x\sim y$ to abbreviate
    that $(x,y)\in E(X,V\setminus X)$.  Suppose that $x\sim y$. Then $f(x) > \lambda f(y)$ 
    if and only if either $x \in U$ and $y \in V \setminus U$ or, for some $0 \le j \le h$,
    $x \in S_j \cap X$ and $y \in S_{\leq j-2} \setminus X$. Therefore,
	\begin{align*}
		\sum_{\substack{x\sim y\\f(x)> \lambda f(y)}}\!\!\!\! \frac{f(x)}{d(x)\,d(y)} 
		    &\le \sum_{\substack{x\sim y\\f(x)> \lambda f(y)}}\!\!\!\! \frac{1}{d(x)\,d(y)}\\
            &= \dr{X \cap U, V\setminus U} + \sum_{j=0}^q \sum_{k=0}^{j-2} \dr{S_j\cap X, S_k\setminus X}\\
            &\le \dr{U, V\setminus U} + B(X, \calS)\,.
    \end{align*}
    We have $\phi(X \cap R \cap S_h) = \phif(X \cap R \cap S_h) \le \phif(X) < \alpha/2$ by the definition of $X$. Moreover, by (F3), we have $\phi(X \cap R \setminus S_h) \le \phi(R \setminus S_h) \le \alpha/2$. Thus $\phi(X \cap R) \le \alpha$ overall, so $X$ is safe and therefore 1-good in $\mathcal{S}$ by (F1). It follows that
    \begin{equation}\label{eqn:something-2}
		\sum_{\substack{x\sim y\\f(x)> \lambda f(y)}}\!\!\!\! \frac{f(x)}{d(x)\,d(y)}
		\le \dr{U,V \setminus U} + \frac{1}{\Lambda}\Gamma(X,\calS).
    \end{equation}
    
   	We next bound $\Gamma(X,\mathcal{S})$ below in terms of $\dr{U, V \setminus U}$. If $U=V$, then it is vacuously true that $\Gamma(X,\mathcal{S}) \ge \Lambda\dr{U,V \setminus U}$. Suppose instead that $U \ne V$ and $\phif(X) > \alpha/4$. Since $X$ is 1-good and $\Gamma(X,\calS) + B(X,\calS) = \dr{X, U \setminus X}$, we have $\Gamma(X,\mathcal{S}) \ge \dr{X, U \setminus X}/2$. By (F2), it follows that 
   	\begin{align*}
   		\Gamma(X,\mathcal{S}) 
   		&\ge \frac{\phi(X \setminus S_0)D}{2\phi(U)} 
   		\ge \frac{\phif(X \setminus S_0)D}{2\phi(U)} 
   		\ge \frac{(\phif(X) - \phif(S_0))D}{2\phi(U)}\\
   		&\ge \frac{(\alpha/4 - \lambda^{-h}\phi(U))D}{2\phi(U)} 
   		= D\left(\frac{1}{2^{10}\phi(U)^2} - \frac{1}{2\lambda^{h}}\right).
   	\end{align*}
   	Since $\phi(U) \le \lambda^{s/5}$, when $|U|$ is sufficiently large we have
   	\[
   		\frac{1}{2^{10}\phi(U)^2} \ge \frac{1}{2^{10}\lambda^{2s/5}} > \frac{1}{\lambda^{\ceil{s/2}}} = \frac{2}{2\lambda^h}.
   	\]
   	It follows that
   	\begin{align*}
   		\Gamma(X,\calS) \ge \frac{D}{2^{11}\phi(U)^2} = \frac{1}{2^{25}\phi(U)^4(\Lambda h)^{2h}}.
   	\end{align*}
   	As in the proof of~\eqref{eqn:frobozz} from~\eqref{eqn:quendor}, it follows that when $|U|$ is sufficiently large the $\Lambda^{-2h}$ term dominates and we have $\Gamma(X,\calS) \ge r^{-3h^2} \ge r^{-s^2}$. 
We have $|U| = r^{s^3}$, so when $|U|$ is sufficiently large, $\Gamma(X,\calS) > \Lambda/2|U| \ge \Lambda/2|V|$. By hypothesis, it follows that $\Gamma(X,\calS) \ge \Lambda\dr{U, V \setminus U}$. Thus $\Gamma(X,\calS) \ge \Lambda\dr{U, V \setminus U}$ when either $U=V$ or $\phi(X) > \alpha/4$, and so by~\eqref{eqn:something-2}, in both cases we have
   	\[
   		\sum_{\substack{x\sim y\\f(x)> \lambda f(y)}}\!\!\!\! \frac{f(x)}{d(x)\,d(y)}
   		\le \frac{2}{\Lambda}\Gamma(X,\mathcal{S}).
   	\]
   	
   	Now, for all $x \sim y$, we have $f(x) \le \lambda f(y)$ whenever there exists $j$ such that $x \in S_j$ and $y \in S_{\ge j-1}$. These are precisely the edges summed over in the definition of $\Gamma(X,\mathcal{S})$. Thus by the definitions of $f$ and $\Lambda$,
   	\begin{align*}
   		\sum_{\substack{x\sim y\\f(x)> \lambda f(y)}}\!\!\!\! \frac{f(x)}{d(x)\,d(y)}
   		&\le \frac{2}{\Lambda}\Gamma(X,\calS)
   		= \frac{2}{\Lambda}\sum_{j=0}^q\sum_{k=j-1}^q \dr{X \cap S_j, X \setminus S_k}\\
   		&\le \frac{2\lambda^h}{\Lambda} \sum_{j=0}^q\sum_{k=j-1}^q \lambda^{k-h}\dr{X \cap S_j, X \setminus S_k}\\
   		&\le \frac{2\lambda^h}{\Lambda}\sum_{\substack{x\sim y\\f(x)\le \lambda f(y)}}\!\!\!\! \frac{f(y)}{d(x)\,d(y)}
   		< \frac{r-1}{4r}\sum_{\substack{x\sim y\\f(x)\le \lambda f(y)}}\!\!\!\! \frac{f(y)}{d(x)\,d(y)}.
   	\end{align*}
   	Thus (iv) holds.
\end{proof}

\begin{lemma}\label{lem:layered}
For all $r>1$, there exists $u_0 > 0$ such that the following holds. Let $G = (V,E)$ be a connected graph. Let $U \subseteq V$ satisfy $|U| \ge u_0$, 
$\dr{U,V \setminus U} \le 1/2|V|$,
and $\phi(U) \le \lambda^{s/5}$. Let $\MM$ be a Moran process on $G$ with fitness $r$ such that $|U \setminus \MM{0}| \le 1$. With probability at least $1 - |V|^2\exp(-\lambda^{s/5})$, $\MM$ fixates before $U$ empties of mutants.
\end{lemma}
\begin{proof}
Let $n=|V(G)|$.
Let $f\colon V\rightarrow\mathbb{R}_{\ge 0}$ be as in Lemma~\ref{lem:true-potential}. 
Since $U$ is non-empty, $f$ is not everywhere zero.
Our goal will be to apply 
Lemma~\ref{lem:gen-pot-b} with $x^- = \alpha/4$ and $x^+ = \alpha/2$.
We first establish the criteria in Lemma~\ref{lem:gen-pot-b}.
\begin{enumerate} 
\item[(C1)] $\alpha/2 \leq \phi_f(V)$.
\item[(C2)] $f$ is $(\alpha/4,\alpha/2)$-valid.
\item[(C3)] $\alpha/4 > m_f$.
\item[(C4)] $\phi_f(\MM(0)) \geq \alpha/2-m_f$.
\end{enumerate}

First, Lemma~\ref{lem:true-potential}(ii) guarantees that $\phi_f(U) \geq \alpha/2$, which gives (C1).
Since  $|U \setminus \MM{0}| \le 1$ by hypothesis and
 $m_f = \max_{v\in V}\phi_f(v)$,  by definition, we also obtain (C4).
 Lemma~\ref{lem:true-potential}(iv) guarantees (C2).
 Lemma~\ref{lem:true-potential}(iii) gives
 the following useful inequality:
  \begin{equation}
 \label{eq:makeiteasy}
 m_f \leq 1/\lambda^h.
 \end{equation}
 Recall Definition~\ref{def:consts}.
To get all of the tedious calculation out of the way at once, we 
note the following string of inequalities, which hold when $|U|$ is sufficiently large (so $s$ is sufficiently large).
The third inequality uses the hypothesis $\phi(U) \le \lambda^{s/5}$.
\begin{equation}
\label{eq:useitagain}
\frac{1}{\lambda^h} < 
\frac{\lambda^{s/4}}{\lambda^h} 
\leq \frac{1}{\lambda^{s/4}} 
\leq \frac{1}{\phi(U) \lambda^{s/4-s/5}} 
\leq \frac{\beta}{2^{10} \phi(U)} = \frac{\beta \alpha}{8}   < \frac{\alpha}{4}.
\end{equation}
Combining~\eqref{eq:makeiteasy} and $1/\lambda^h < \alpha/4$ from \eqref{eq:useitagain}, we get (C3).

Now let $p$ denote $\Pr(\exists t \ge 0 \text{ such that } \phi_f(\MM{t}) \le \alpha/4)$.
From Lemma~\ref{lem:gen-pot-b}, we obtain
$$
 p \le  {\left( {\tfrac{8r}{r-1}}\right)}^{1/2} n^2\cdot e^{-\beta\alpha/8m_f}. 
$$
Using~\eqref{eq:makeiteasy} and then~\eqref{eq:useitagain},
$\beta \alpha / 8 m_f \geq \beta \alpha \lambda^h / 8 \geq \lambda^{s/4}$, so
$$
p \leq   {\left({\tfrac{8r}{r-1}}\right)}^{1/2} n^2\exp(-\lambda^{s/4})
		\le n^2 \exp(-\lambda^{s/5}).$$
		
Finally, if $U$ empties of mutants at time $t$ then
$\MM{t} \cap U = \emptyset$ which, 		
by Lemma~\ref{lem:true-potential}(i), 
implies that   $\phi_f(\MM{t}) = 0$, so certainly $\phi_f(\MM{t}) \leq \alpha/4$.
It follows that
	\[
		\Pr(\MM\mbox{ fixates before $U$ empties of mutants}) \ge 1-p \geq 1 - n^2\exp(-\lambda^{s/5}),
	\]
	as required.
\end{proof}

\subsection{Applications of Lemmas~\ref{lem:true-potential} and~\ref{lem:layered}}\label{sec:layered-app}

We first require the following result, which we shall use in the large-potential case when~$U \ne V$.

\begin{lemma}\label{lem:killers}
	Let $x^+,r>1$, let $G=(V,E)$ be a sufficiently large connected graph, and let $U \subseteq V$ with $\phi(U) \ge x^+$. Suppose that for all 
	$S \subseteq U$ with $0 < \phi(S) < x^+$, $\dr{S,U \setminus S} \ge \frac{4r}{r-1}\dr{U, V \setminus U}$. Let $\MM$ be a Moran process on~$G$ with fitness $r$ and $|U\setminus \MM{0}|\leq 1$. Then, with probability at least $1-
	{\left( {\tfrac{8r}{r-1}}\right)}^{1/2}
	|V|^2e^{-\beta x^+/2}$, $\MM$ fixates before $U$ empties of mutants.
\end{lemma}
\begin{proof}
	Let $f\colon V\rightarrow\mathbb{R}_{\ge 0}$ be the indicator function of $U$ in $V$, and let $\phi_f$ be the function given by Definition~\ref{defn:phif}.
	For all $S \subseteq U$ with $0 < \phi(S) < x^+$, by hypothesis we have
	\[
	\dr{S, U \setminus S} \ge \frac{4r}{r-1}\,\dr{U, V \setminus U} \ge \frac{4r}{r-1}\,\dr{S, V \setminus U}\,.
	\]
	Since 
	$$ \sum_{\substack{(x,y)\in E(S,V\setminus S)\\ f(x)\leq \lambda f(y)}}
	\frac{f(y)}{d(x)\,d(y)} = \dr{S,U\setminus S}$$
	and
	$$\sum_{\substack{(x,y)\in E(S,V\setminus S)\\ f(x) > \lambda f(y)}}
	\frac{f(x)}{d(x)\,d(y)} = \dr{S,V\setminus U},$$ 
	it follows that $f$ is $(0,x^+)$-valid. Moreover, we have $m_f \le 1$. Thus by Lemma~\ref{lem:gen-pot-b} with $x^-=0$,
	\[
	\Pr(\MM\mbox{ fixates before $U$ empties of mutants}) \ge 1 - 
	{\left( {\tfrac{8r}{r-1}}\right)}^{1/2} |V|^2e^{-\beta x^+/2}\,.\qedhere
	\]
\end{proof}

We are now in a position to prove Theorem~\ref{thm:new-phase-trans}. 
To make the statement of the theorem more natural, the logarithm is to the base~$e$.
(This follows our convention, which is also used in the proof below that logarithms are always to
the base~$e$ unless we specify otherwise.) We first prove the upper bound, which applies when $r<1$.

\begin{theorem}\label{thm:disadvantageous}
Let $0<r'<1$. Then there exists $C>0$, depending   on $r'$, 
such that the following holds. Let $G = (V,E)$ be a connected $n$-vertex graph, and let $\MM$ be a Moran process on $G$ with fitness $r'$ such that $|\MM{0}| \leq 1$. Then if $n$ is sufficiently large relative to $r'$, $\MM$ goes extinct with probability at least $1 - e^{-\exp(C(\log n)^{1/3})}$.
\end{theorem}
\begin{proof}
Fix $r'<1$ and let $r=1/r'>1$. Recall the definitions of $\lambda$ and $\beta<1$ (depending on~$r$) from Definition~\ref{def:consts}.
Choose $C>0$ to be sufficiently small with respect to~$r$ 
so that, for all sufficiently large $x$, 
\begin{equation}
\label{eq:ugly}
{\left({\tfrac{8r}{r-1}}\right)}^{1/2} r^{2 x^3}
e^{\exp(C {(\log r)}^{1/3} x)}
  \leq 
   \exp(\beta\lambda^{x/5}/2).
\end{equation}

Now given the $n$-vertex graph $G$, let $U=V$  so that $s = (\log_r n)^{1/3}$
and the quantity $C(\log n)^{1/3}$, from the failure probability in the statement of the theorem, is equal to~$C {(\log r)}^{1/3} s$.
We will assume that $n$ is sufficiently large that \eqref{eq:ugly} holds for $x=s$.

Note that $\MM$ is dual to a Moran process $\MM'$ on $G$ with fitness $r=1/r' > 1$ and initial state $M'(0) = V\setminus \MM(0)$, such that $|V \setminus \MM'(0)| \le 1$ and $\MM$ goes extinct if and only if $\MM'$ fixates. 
 If  $\phi(V) > \lambda^{s/5}$, then by Lemma~\ref{lem:killers} (taking $U=V$ and $x^+ = \phi(V)$),
	\begin{equation*} 
		\Pr(\MM\mbox{ goes extinct}) \ge 1 - {\left({\tfrac{8r}{r-1}}\right)}^{1/2} n^2 \exp(-\beta\lambda^{s/5}/2).
	\end{equation*}
If instead  $\phi(V) \le \lambda^{s/5}$, then we apply Lemma~\ref{lem:layered} to $\MM'$ (again with $U=V$) to obtain
\begin{align*} 
	\Pr(\MM\mbox{ goes extinct}) \ge 1 - n^2\exp(-\lambda^{s/5}) 
		&\geq  1 - {\left({\tfrac{8r}{r-1}}\right)}^{1/2} n^2 \exp(-\beta\lambda^{s/5}/2)\\
		&=  1 - {\left({\tfrac{8r}{r-1}}\right)}^{1/2} r^{2 s^3} \exp(-\beta\lambda^{s/5}/2).
	\end{align*}
	
By~\eqref{eq:ugly}, the right-hand side is at least 
$1- e^{-\exp(C {(\log r)}^{1/3} s)}$, which is equal to
$1 - e^{-\exp(C(\log n)^{1/3})}$,
as desired.
	 \end{proof}

Finally, we prove the lower bound of Theorem~\ref{thm:new-phase-trans}, which applies when $r>1$.

\begin{theorem}\label{thm:advantageous}
	Let $r>1$. Then there exists $C>0$, depending   on $r$, such that the following holds. Let $G = (V,E)$ be a connected $n$-vertex graph, and let $\MM$ be a Moran process on $G$ with fitness $r$ such that $\MM{0}$ is a uniformly-random vertex. Then, if $n$ is sufficiently large relative to $r$, $\MM$ fixates with probability at least $e^{C(\log n)^{1/3}}/n$.
\end{theorem}
\begin{proof}
Let $U=V$, so that $s = (\log_r n)^{1/3}$
and the quantity $C(\log n)^{1/3}$, from the failure probability in the statement of the theorem, is equal to~$C {(\log r)}^{1/3} s$.
We will choose $C$ 
to be small enough (as a function of $r$) so that, when $n$ is sufficiently large, the final inequalities of~\eqref{eqn:advantageous-1} and~\eqref{eqn:advantageous-2} both hold. 
	
	Write $\MM{0} = \{m_0\}$, where $m_0 \in V$ is uniformly chosen. First suppose $\phi(V) \ge \lambda^{s/5}$. Then by Corollary~\ref{cor:Dd}, we have
	\begin{equation}\label{eqn:advantageous-1}
		\Pr(\MM\mbox{ fixates}) \ge \frac{(r-1)\phi(V)}{2rn} \ge \frac{(r-1)\lambda^{s/5}}{2rn} \ge \frac{e^{C(\log n)^{1/3}}}{n},
	\end{equation}
	and so the result follows when $n$ is sufficiently large.
	
Now suppose that $\phi(V) < \lambda^{s/5}$. 
If $n$ is sufficiently large, then 
$\phi(V) \leq \lambda^h / (\log n)^2 2^8$, so $\lambda^{-h} (\log n)^2 \leq \alpha/2$.
	Let $f$ be as in Lemma~\ref{lem:true-potential}. Let 
	\[
		\tau = \min\{t \ge 0 \mid \phif(\MM{t}) \in \{0\} \cup [\lambda^{-h}(\log n)^2,\phif(V)]\},
	\]
	and note that $\tau>0$ by Lemma~\ref{lem:true-potential}(iii), since $\phif(M(0)) \le m_f \le \lambda^{-h}$.
	Let $X_t = \phif(\MM{t})$ for $t \le \tau$, and let $X_t = \phif(\MM(\tau))$ for $t > \tau$. Lemma~\ref{lem:true-potential}(iv) implies that $f$ is $(0,\alpha/2)$-valid.
Thus, it is valid for all $Y\subseteq V$ such that $0 < \phif(Y) < \alpha/2$, hence for all $Y$
such that $0 < \phif(Y) < \lambda^{-h} (\log n)^2$. 
	Lemma~\ref{lem:add-submart} shows that 
	$X$ is a submartingale (that is, $\E[X_{t+1} \mid  X_t=c]\geq  c$).
	
	For all $x \in V$, let $p_x = \Pr(X_\tau \ge \lambda^{-h}(\log n)^2 \mid m_0 = x)$. We have $X_\tau \le \lambda^{-h}(\log n)^2+m_f$, and $m_f \le \lambda^{-h}$ by Lemma~\ref{lem:true-potential}(iii), so by the optional stopping theorem we have 
	\[
		X_0 \le \E(X_\tau) \le p_x(\lambda^{-h}(\log n)^2 + \lambda^{-h}) \le 2p_x\lambda^{-h}(\log n)^2.
	\]
	Since $X_0 = \phif(x)$, it follows that $p_x \ge \phif(x)\lambda^h/2(\log n)^2$. We therefore have
	\[
		\Pr\big(\phif(\MM{\tau}) \ge \lambda^{-h}(\log n)^2\big) = \frac{1}{n}\sum_{x \in V}p_x \ge \frac{\phif(V)\lambda^h}{2n(\log n)^2}.
	\]
	By Lemma~\ref{lem:true-potential}(ii), we have $\phif(V) \ge \alpha/2$. Hence
	\[
		\Pr\big(\phif(\MM{\tau}) \ge \lambda^{-h}(\log n)^2\big) \ge \frac{\alpha\lambda^h}{4n(\log n)^2} = \frac{\lambda^h}{2^9\phi(V)n(\log n)^2}. 
	\]
	Since $\phi(V) < \lambda^{s/5}$ and $h \ge {s/2}$, it follows that for sufficiently large $n$,
	\[
		\Pr\big(\phif(\MM{\tau}) \ge \lambda^{-h}(\log n)^2\big) \ge \frac{\lambda^{s/4}}{n}.
	\]
	By Lemma~\ref{lem:gen-pot-b}, applied with $x^- = 0$, $x^+ = \lambda^{-h}(\log n)^2$ and $m_f \le \lambda^{-h}$, it therefore follows that when $n$ is sufficiently large,
	\begin{align}\nonumber
		\Pr(\MM \mbox{ fixates}) &\ge \Pr(\MM \mbox{ fixates}\mid \phif(M(\tau)) 
		\ge \lambda^{-h}(\log n)^2) \cdot \lambda^{s/4}/n\\\label{eqn:advantageous-2}
		&\ge \left(1-{\left(\frac{8r}{r-1}\right)}^{1/2} n^2\cdot e^{-\beta(\log n)^2/2}\right) \frac{\lambda^{s/4}}{n} 
		\ge \frac{1}{2}\cdot\frac{\lambda^{s/4}}{n} 
		\ge \frac{e^{C(\log n)^{1/3}}}{n},
	\end{align}
	and so the result follows.
\end{proof}

Theorem~\ref{thm:new-phase-trans} now follows immediately from Theorems~\ref{thm:disadvantageous} and~\ref{thm:advantageous}, together with the fact~\cite[Lemma~1]{DGMRSS2014} that for all connected graphs $G$, $f_{G,1}=1/|V(G)|$.

\section{Double star}
\label{sec:dubstar}

\newcommand{\DS}{D_k}
For all integers $k \ge 1$, we define the \textit{double star} $\DS$ as follows. Let $L_1$, $L_2$, $\{x_1\}$ and $\{x_2\}$ be disjoint vertex sets with $|L_1| = |L_2| = k$. Add edges to join every vertex in $L_1$ to $x_1$, to join every vertex in $L_2$ to $x_2$, and to join $x_1$ to $x_2$. Thus $\DS$ consists of two $k$-leaf stars, induced by $L_1 \cup \{x_1\}$ and $L_2 \cup \{x_2\}$, connected by their centres.

\begin{theorem}
\label{thm:dubstar}
	Let $r > 1$, let $k$ be a positive integer, and let $m_0 \in V(\DS)$ be uniformly random. Let $\MM$ be a Moran process on $\DS$ with fitness $r$ and $\MM{0} = \{m_0\}$. Then the expected absorption time of $\MM$ is at least $\frac{(r-1)^2}{2^5r^4}|V(\DS)|^3$.
\end{theorem}
\begin{proof}
	For all $w \in V(\DS)$, let $\MM^w$ be a Moran process on $\DS$ with fitness $r$ and $\MM{0} = \{w\}$. Let $v \in L_1$; we first bound the expected absorption time of $\MM^v$, from which the result will follow easily. Let $\mathcal{E}_1$ be the event that $\MM^v$ goes extinct. Let $C = (r-1)/r^3$, and let $\mathcal{E}_2$ be the event that $\MM^v$ fixates at some time $t < C(k+1)^3$. Finally, let $\mathcal{E}_3$ be the event that $\MM^v$ fixates at some time $t \ge C(k+1)^3$. Then we have
	\begin{equation}\label{eqn:doub-star-1}
	\E(\mbox{Absorption time of }\MM^v) \ge C(k+1)^3\Pr(\mathcal{E}_3) = C(k+1)^3(1 - \Pr(\mathcal{E}_1) - \Pr(\mathcal{E}_2)).
	\end{equation}
	Since every vertex in $L_1$ has degree 1, by Theorem~\ref{thm:potential} we have
	\begin{equation}\label{eqn:doub-star-2}
	\Pr(\mathcal{E}_1) \le 1 - (1 - r^{-1}) = 1/r.
	\end{equation}
	
	We next bound $\Pr(\mathcal{E}_2)$ above. For all $t \ge 0$, let $\mathcal{E}_2^t$ be the intersection of the following three events:
	\begin{itemize}
		\item $\mathcal{A}_1^t$: $x_1 \in \MM{v}{t}$ and $(\{x_2\} \cup L_2) \cap \MM{v}{t} = \emptyset$;
		\item $\mathcal{A}_2^t$: in $\MM^v$, at time $t+1$, $x_1$ spawns a mutant onto $x_2$;
		\item $\mathcal{A}_3^t$: in $\MM^v$, in the time interval $[t+2,\infty)$, $x_2$ spawns before any vertex in $L_2$ spawns.
	\end{itemize}
	Note that $\MM^v$ cannot fixate until some vertex in $L_2$ becomes a mutant, so $\Pr(\mathcal{E}_2) \le \sum_{t=1}^{\floor{C(k+1)^3}}\Pr(\mathcal{E}_2^t)$. For all sets $S \subseteq V(\DS)$ with $(\{x_1,x_2\} \cup L_2) \cap S = \{x_1\}$, we have $\Pr(\mathcal{A}_2^t \mid \MM{v}{t} = S) = r/(k+1)W(S) \le r/2(k+1)^2$. Thus 
	\[
	\Pr(\mathcal{A}_2^t \mid \mathcal{A}_1^t) \le r/2(k+1)^2.
	\]
	Moreover, for all sets $S \subseteq V(\DS)$ with $L_2 \cap S = \emptyset$ and all $t' \ge 0$, conditioned on $\MM{v}{t'} = S$, $x_2$ spawns at time $t'+1$ with probability at most $r/W(S)$ and a vertex in $L_2$ spawns (a non-mutant) at time $t'+1$ with probability $k/W(S)$. It follows that
	\[
	\Pr(\mathcal{A}_3^t \mid \mathcal{A}_1^t \cap \mathcal{A}_2^t) \leq \frac{r}{k+r} < \frac{r}{k+1}.
	\]
	Thus $\Pr(\mathcal{E}_2^t) \le r^2/2(k+1)^3$, and so 
	\begin{equation}\label{eqn:doub-star-3}
	\Pr(\mathcal{E}_2) \le C(k+1)^3\cdot \frac{r^2}{2(k+1)^3} = \frac{r-1}{2r}.
	\end{equation}
	
	Combining~\eqref{eqn:doub-star-1},~\eqref{eqn:doub-star-2} and~\eqref{eqn:doub-star-3} yields
	\[
	\E(\mbox{Absorption time of }\MM^v) \ge C(k+1)^3\left(1 - \frac{1}{r} - \frac{r-1}{2r} \right) = \frac{(r-1)^2}{2r^4}(k+1)^3.
	\]
	By symmetry, the same bound holds for all $v \in L_2$. Since $m_0 \in L_1 \cup L_2$ with probability $k/(k+1) \ge 1/2$, it follows that the expected absorption time of $\MM$ is at least $\frac{(r-1)^2}{4r^4}(k+1)^3$. Since $|V(\DS)| = 2(k+1)$, the result follows.
\end{proof}

\section{Absorption time}\label{sec:abs-time}

\newcommand{\Tabs}{T_{\mathsf{abs}}}

In this section, we will prove our upper bound on absorption time (Theorem~\ref{thm:abs-time-intro}). In this section only, for a graph $G=(V,E)$ and $S\subseteq V$, we write $\overline{S} = V \setminus S$. We will focus on proving Theorem~\ref{thm:abs-time-intro} when $r>1$; as we will see in the proof of Theorem~\ref{thm:tabs-detail}, it is easy to show that the $r<1$ case is equivalent.

Throughout the section, we use the following lemma to bound the expected absorption time of a supermartingale. The techniques involved are well-known; see Hajek~\cite{Haj1982:Hitting-time-drift} or He and Yao~\cite{HY2001:Drift}.

\begin{lemma}\label{lem:he-yao}
	Let $Y$ be a Markov chain with finite state space $\Omega$. Let $k_1,k_2>0$, let $\Psi\colon \Omega\rightarrow\mathbb{R}_{\ge 0}$ be a function, and let $\tau \ge 0$ be a stopping time with $\tau \le \min\{i \mid \Psi(Y_i) = 0\mbox{ or }\Psi(Y_i) \ge k_1\}$. Suppose that:
	\begin{enumerate}[(i)]
		\item from every state $S_1 \in \Omega$ with $0 < \Psi(S_1) < k_1$, there exists a path in $Y$ from $S_1$ to some state $S_2$ with $\Psi(S_2) = 0$ or $\Psi(S_2) \ge k_1$;
		\item for all $i \ge 0$, if $\Psi(Y_i) < k_1$, then $\Psi(Y_{i+1}) \le k_1+1$; and
		\item for all $i \ge 0$ and all $Y \in \Omega$ such that the events $\tau>i$ and $Y_i=Y$ are consistent, $\E(\Psi(Y_{i+1}) - \Psi(Y_i) \mid Y_i=Y) \ge k_2$.
	\end{enumerate}
	Then we have $\E(\tau) \le (k_1-\Psi(Y_0)+1)/k_2$.
\end{lemma}
\begin{proof}
	Without loss of generality, suppose $Y_0$ is deterministic. For any fixed $i\geq 0$, by (iii) we have
	\begin{align*}
	\E\big(\Psi(Y_\tau)-\Psi(Y_i) \mid \tau > i\big) &= \E\big(\Psi(Y_{i+1}) -\Psi(Y_i) \mid \tau>i\big)\ +\\
	&\qquad\qquad\E\big(\Psi(Y_\tau)-\Psi(Y_{i+1})\mid \tau>i+1\big)\Pr(\tau>i+1\mid\tau>i)\\
	&\ge k_2 + \E\big(\Psi(Y_\tau)-\Psi(Y_{i+1})\mid \tau>i+1\big)\Pr(\tau>i+1\mid\tau>i).
	\end{align*}
	Multiplying through by $\Pr(\tau>i)$, 
	and 	letting $N_i$ denote $\E(\Psi(Y_\tau)-\Psi(Y_i) \mid \tau > i)\Pr(\tau>i) $,	
	we obtain
	$ N_i - N_{i+1} \ge k_2\Pr(\tau>i)$.
	
	Now fix any $I\geq 0$ and  sum both sides for $i \in \{0,\ldots,I\}$ to obtain
	$N_0 - N_{I+1} \geq k_2 \sum_{i=0}^I\Pr(\tau>i)$.
	Since $\E(\Psi(Y_\tau)-\Psi(Y_0)) = N_0$,
	we conclude that, for any fixed~$I$,
	\begin{equation}\label{eq:tolimit}
	\E(\Psi(Y_\tau)-\Psi(Y_0))	\geq 
	k_2 \sum_{i=0}^I\Pr(\tau>i)
	+ N_{I+1}.
	\end{equation}
	
	We will now take the limit of both sides of~\eqref{eq:tolimit} as $I \rightarrow \infty$. 
	For the left-hand-side, since it   does not depend on~$I$, we have
	$\lim_{I\rightarrow \infty}  \E(\Psi(Y_\tau)-\Psi(Y_0))	=   \E(\Psi(Y_\tau)-\Psi(Y_0))$.
	For the first term in the right-hand-side, 
	$\lim_{I\rightarrow\infty}  \sum_{i=0}^I\Pr(\tau>i) 
	=  \sum_{i=0}^\infty \Pr(\tau>i) =  \E(\tau)$.
	Finally, since
	$|\E(\Psi(Y_\tau)-\Psi(Y_{I+1}) \mid \tau > I+1)|$
	is bounded above by a constant function of $\Omega$
	and   $\lim_{I\rightarrow\infty}\Pr(\tau>I+1) = 0$ by~(i),
	we conclude that 
	$\lim_{I\rightarrow\infty} N_{I+1}=0$.
	Putting it all together, we have
	\[
	\E\big(\Psi(Y_\tau)-\Psi(Y_0)\big) \ge k_2\E(\tau).
	\]
	By (ii), we have $\Psi(Y_\tau) \le k_1+1$, so $k_1+1 - \Psi(Y_0) \ge k_2\E(\tau)$. The result follows.
\end{proof}

\newcommand{\bs}{\rho} 
\begin{definition}\label{def:barrier} 
	Given an $n$-vertex connected graph $G = (V,E)$ on at least two vertices and a Moran process $\MM$ on $G$ with fitness $r>1$, recall from Definition~\ref{def:valid} that $\lambda = (r+1)/2$. Let $\Tabs = \min\{t \ge 0 \mid \MM{t} \in \{\emptyset,V\}\}$ and let $\bs(n) = \frac{10r}{r-1}\exp((\frac{10}{\log \lambda}\log\log n)^3\log r)$. A non-empty set $S \subset V$ is a \textit{barrier} if $\dr{S,\overline{S}} < 1/2n\bs(n)$.
\end{definition}

The value $1-1/2n^7$ in the following lemma is what we need in order to prove our main result (Theorem~\ref{thm:tabs-detail}), though we actually establish a tighter bound.

\begin{lemma}\label{lem:magic}
	For all $r>1$, there exists $n_0 \ge 2$ such that the following holds. Let $G=(V,E)$ be a connected graph with $n\geq n_0$ vertices, let $U \subseteq V$, and suppose $\dr{U,\overline{U}} \le 1/n\bs(n)$. Let $\MM$ be a Moran process on~$G$ with fitness $r$ such that $|U\setminus \MM{0}|\leq 1$. Then, with probability at least $1 - 1/2n^7,$ $\MM$ fixates before $U$ empties of mutants.
\end{lemma}
\begin{proof}
	Let $x = (\log_r ((r-1)\bs(n)/10r))^{1/3}$. Observe that by the definition of $\bs(n)$ (Definition~\ref{def:barrier}), $\lambda^{x/5} = (\log n)^2$. We split into two cases.
	
	\noindent {\bf Case 1.}\quad Suppose $U$ contains a non-empty set $S$ with $\dr{S,\overline{S}} \le 5r\dr{U,\overline{U}}/(r-1)$ and $\phi(S) \le \lambda^{x/5}$. Thus $\dr{S,\overline{S}} \le 5r/(r-1)n\bs(n)$ by hypothesis, and in particular we can ensure $\dr{S,\overline{S}} \le 1/2n$ by increasing $n_0$. Thus by Lemma~\ref{lem:drift-props}(\ref{d:size}), it follows that $|S| \ge (r-1)\bs(n)/10r$; in particular, we can make $S$ as large as we like by increasing $n_0$. Moreover, writing $s = (\log_r |S|)^{1/3}$, we have $x \le s$ and hence $\phi(S) \le \lambda^{x/5} \le \lambda^{s/5}$. Thus by Lemma~\ref{lem:layered}, applied with $U=S$, we have
	\[
		\Pr(\mbox{$\MM$ fixates before $U$ empties of mutants}) \ge 1 - n^2 \exp(-\lambda^{s/5}) \ge 1 - n^2\exp(-\lambda^{x/5}).
	\]
	Since $\lambda^{x/5} = (\log n)^2$, when $n$ is sufficiently large, this is at least $1-1/2n^7$.
	
	\noindent {\bf Case 2.}\quad Suppose that for all $S \subseteq U$ with $0 < \phi(S) \le \lambda^{x/5}$, $\dr{S, \overline{S}} \ge 5r\dr{U,\overline{U}}/(r-1)$. Thus $\phi(U) \ge \lambda^{x/5}$ and, for all such $S$, 
	\[
		\dr{S, U \setminus S} = \dr{S,\overline{S}} - \dr{S, V \setminus U} \ge \frac{5r}{r-1}\dr{U,\overline{U}} - \dr{U, \overline{U}} > \frac{4r}{r-1}\dr{U,\overline{U}}.
	\]
	Then it follows from Lemma~\ref{lem:killers}, applied with $x^+ = \lambda^{x/5}$, that
	\begin{align*}
	\Pr(\mbox{$\MM$ fixates before $U$ empties of mutants}) &\ge 1 - {\left( {\tfrac{8r}{r-1}} \right)}^{1/2} n^2 \exp(-(\beta/2)\lambda^{x/5}).
	\end{align*}
	Since $\lambda^{x/5} = (\log n)^2$, when $n$ is sufficiently large, this is at least $1 - 1/2n^7$.
\end{proof}
	
In order to prove Theorem~\ref{thm:abs-time-intro}, we must bound $\E(\Tabs)$ above. Barriers are so named because they act as barriers to fast absorption; if $G$ contained no barriers, we would have $\dr{M(t),\overline{M(t)}} \ge 1/2n\bs(n)$ until absorption, and so Theorem~\ref{thm:abs-time-intro} would be immediate from Lemma~\ref{lem:he-yao} on taking $Y = \MM$, $\Psi = \phi$, $\tau = \Tabs$, $k_1 = \phi(V) \le n$ and $k_2 = (r-1)/2rn^2\bs(n)$. (See Lemma~\ref{lem:tabs-2}.)

We therefore need to bound the time spent at barriers. Here, our proof contains two crucial ideas. The first is that with very high probability, if $\MM$ encounters distinct barriers $S_1, \dots, S_K$ before absorption (in that order), then $M$ fixates and $S_1 \subset \dots \subset S_K$. The second is that if $S$ is a barrier, once $\MM$ has reached potential $\phi(S) + (\log n)^2$, $\MM$ is very unlikely to return to $S$. Thus if $\MM$ encounters a barrier $S$ at time $t_0$, we may define a stopping time $\tau$ to be the first time $t > t_0$ at which either $M(t)$ is a barrier distinct from $S$ or $\phi(M(t)) \ge \phi(S) + (\log n)^2$; then after time $\tau$, $\MM$ is very unlikely to return to $S$. Using Lemma~\ref{lem:he-yao}, we may show that $\E(\tau - t_0\mid M(t_0) = S)$ is at most roughly $n(\log n)^2/\dr{S,\overline{S}}$. As extinction is very unlikely after a barrier has been encountered, we think of this as the amount of time spent bypassing $S$ on the way to fixation. It turns out for any collection of barriers $S_1 \subset \dots \subset S_K$, we can bound the sum $\sum_i n(\log n)^2/\dr{S_i,\overline{S_i}}$ above (see Lemma~\ref{lem:tabs-1}). This yields our bound on expected absorption time. To make this argument rigorous, we partition the interval $[0,\Tabs]$ as follows.

\begin{definition}
	Given a connected $n$-vertex graph $G = (V,E)$ with $n \ge 2$ and a Moran process $\MM$ on $G$, we define stopping times $T_1,T_2, \dots$ and $T_0',T_1', \dots$ by $T_0' = 0$ and, for all $i \ge 1$,
	\begin{align*}
		T_i &= \min\{t \ge T_{i-1}' \mid \MM{t} \in \{\emptyset,V\} \mbox{ or $\MM{t}$ is a barrier}\},\\
		S_i &= \MM{T_i},\\
		T_i' &= \min\{t \ge T_i \mid \MM{t} \in \{\emptyset, V\} \mbox{ or }\phi(\MM{t}) \ge \phi(S_i) + (\log n)^2\mbox{ or }\\
		&\qquad\qquad\qquad\qquad\qquad\qquad\MM{t}\mbox{ is a barrier with }\dr{\MM{t},\overline{\MM{t}}} < \dr{S_i,\overline{S_i}}\}.
	\end{align*}
	Let $K = \min(\{i \ge 1 \mid S_i \in \{\emptyset,V\}\} \cup \{i\ge 2 \mid S_{i-1} \not\subset S_i\})$.
\end{definition}

Note that $K < \infty$ with probability 1, and that
\begin{equation}\label{eqn:tabs}
	\Tabs = \sum_{i=1}^{K-1} (T_i' - T_i) + \sum_{i=0}^{K-1} (T_{i+1}-T_i') + (\Tabs - T_K).
\end{equation}
We will proceed by applying linearity of expectation to~\eqref{eqn:tabs}, and bounding each term individually. We first show that with high probability, $S_1 \subset \dots \subset S_K$.

\begin{lemma}\label{lem:blockchain}
	Let $r > 1$, let $G = (V,E)$ be a connected $n$-vertex graph, and let $\MM$ be a Moran process on $G$ with fitness $r$. If $n$ is sufficiently large, then for all $i \ge 1$ and all barriers $\emptyset \subset S \subset V$, we have $\Pr(S_i \subset S_{i+1} \mid S_i = S) \ge 1-1/n^6$.
\end{lemma}
\begin{proof}
	Let $S$ be a barrier and let
	\[
		\mathcal{P} = \{\emptyset \subset A \subseteq S \mbox{ such that } \dr{A,\overline{A}} < 1/n\bs(n) \mbox{ and, }\forall\, \emptyset \subset A' \subset A,\,\dr{A',\overline{A'}} \ge 1/n\bs(n)\}.
	\]
	Note that the sets in $\mathcal{P}$ form an antichain under inclusion, so by Lemma~\ref{lem:more-drift-props}(\ref{d:minimal}) they are disjoint. Thus $|\mathcal{P}| \le n$. Moreover, every set $B \subseteq S$ with $\dr{B,\overline{B}} < 1/n\bs(n)$ contains some set in $\mathcal{P}$ as a subset.
	
	Suppose that $A$ is a barrier. Then by Lemma~\ref{lem:more-drift-props}(\ref{d:sub}), $\dr{S \setminus A,\overline{S \setminus A}} \le \dr{S,\overline{S}} + \dr{A,\overline{A}} < 1/n\bs(n)$. If $S \setminus A$ is non-empty, then it follows that there exists $S' \in \mathcal{P}$ with $S' \subseteq S \setminus A$ and hence $A \cap S' = \emptyset$. It follows that if $S \not\subseteq S_{i+1}$, then some set in $\mathcal{P}$ must be disjoint from $S_{i+1}$. However, by Lemma~\ref{lem:magic} (taking $U \in \mathcal{P}$) combined with a union bound, we have
	\[
		\Pr(\exists t \ge T_i\mbox{ and }S' \in \mathcal{P} \mbox{ with }\MM{t} \cap S' = \emptyset \mid S_i = S) \le n\cdot 1/2n^7.
	\]
	It follows that $\Pr(S_i \subseteq S_{i+1}\mid S_i=S) \ge 1-1/2n^6$.
	
	Now, if $S_i = S_{i+1} = S$, then by the definitions of $T_i'$ and $T_{i+1}$, we must have $\phi(V) > \phi(\MM{T_i'}) \ge \phi(S)+(\log n)^2$ and $\phi(\MM{T_{i+1}}) = \phi(S)$. However, by Lemma~\ref{lem:multsubmart}, writing $\delta = \delta(G)$, $r^{-\phi(\MM)\delta}$ is a supermartingale. Let 
	\[
		\tau_i = \min\{t \ge T_i' \mid \MM{t} = V\mbox{ or }\phi(\MM{t}) \le \phi(S)\}.
	\]
	For all $X \subseteq V$, let $p_X = \Pr(\phi(\MM{\tau_i}) \le \phi(S) \mid \MM{T_i'} = X)$. Then by the optional stopping theorem, for all $X \subseteq V$ with $\phi(X) \ge \phi(S) + (\log n)^2$, we have
	\[
		r^{-\phi(X)\delta} \ge \E(r^{-\phi(\MM{\tau_i})\delta} \mid \MM{T_i'} = X) \ge p_Xr^{-\phi(S)\delta} \ge p_Xr^{-(\phi(X) - (\log n)^2)\delta}.
	\]
	When $n$ is sufficiently large, rearranging yields $p_X \le r^{-\delta(\log n)^2} \le 1/2n^6$. Thus $\Pr(S_i \ne S_{i+1}\mid S_i=S) \ge 1-1/2n^6$, and so the result follows by a union bound. 
\end{proof}
\begin{corollary}\label{cor:blockchain}
	Let $r > 1$, let $G = (V,E)$ be a connected $n$-vertex graph, and let $\MM$ be a Moran process on $G$ with fitness $r$. If $n$ is sufficiently large, then with probability at least $1-1/n^5$, $S_1 \subset S_2 \subset \dots \subset S_K \in \{\emptyset,V\}$.
\end{corollary}
\begin{proof}
	If $S_1 \in \{\emptyset, V\}$, then $K=1$ and so the result holds. If instead $S_1 \notin \{\emptyset,V\}$, then $S_1$ is a barrier. The definition of $K$ then implies that $\emptyset \subset S_1 \subset \dots \subset S_{K-1} \subset V$, so $K \le n$. It follows by repeated application of Lemma~\ref{lem:blockchain} that
	\[
		\Pr(S_1 \subset \dots \subset S_K) \ge 1 - (n-1)/n^6 > 1 - 1/n^5.
	\]
	We have $S_{K-1} \subset S_K$ precisely when $S_K = V$, so the result follows. 
\end{proof}

We now bound the expected values of the three terms of~\eqref{eqn:tabs}, one by one.
	
\begin{lemma}\label{lem:tabs-1}
	Let $r > 1$, let $G = (V,E)$ be a connected $n$-vertex graph, and let $\MM$ be a Moran process on $G$ with fitness $r$. If $n$ is sufficiently large, then 
	\[
	\E\left(\sum_{i=1}^{K-1} (T_i'-T_i) \right) \le \frac{24ern^3(\log n)^3}{r-1}.
	\]
\end{lemma}
\begin{proof}
	First note that for any sequence $X_1, X_2, \dots$ of random variables, and non-negative integer random variable $Y$, $\E(\sum_{i=1}^{Y-1} X_i) = \sum_{i=1}^\infty \E(X_i \mid Y>i) \Pr(Y>i)$. 
	Thus
	\begin{align*}
	\E\left(\sum_{i=1}^{K-1} (T_i'-T_i) \right) &=
	\sum_{i=1}^{\infty}\E(T_i'-T_i \mid K>i)\Pr(K>i)\\
	&= \sum_{i=1}^\infty \E\Big(\E\big(T_i'-T_i \mid K>i,S_i\big)\,\Big|\, K>i\Big)\Pr(K>i).
	\end{align*}
	For all $S \subseteq V$, we have $\E\big(T_i'-T_i \mid K>i,S_i=S\big) = \E\big(T_i'-T_i \mid S_i=S\big)$. Thus
	\begin{align}\nonumber
	\E\left(\sum_{i=1}^{K-1} (T_i'-T_i) \right) &= \sum_{i=1}^\infty \E\Big(\E\big(T_i'-T_i \mid S_i\big)\,\Big|\, K>i\Big)\Pr(K>i)\\
	\label{eqn:tabs-2-1}
	&= \E\left(\sum_{i=1}^{K-1} \E(T_i' - T_i \mid S_i)\right)\,.
	\end{align}
	Let $i \ge 0$ and consider a possible value $S$ of $S_i$, subject to $K>i$; in particular, this implies $\emptyset \subset S \subset V$. Note that conditioned on $S_i = S$, 
	\[
	T_i' \le \min\big\{t \ge T_i \mid \phi(\MM{t}) = 0 \mbox{ or }\phi(\MM{t}) \ge \min\{\phi(S) + (\log n)^2,\phi(V)\}\big\}\,.
	\]
	Moreover, conditioned on $S_i = S$, for all $T_i \le t < T_i'$ we have $\dr{\MM{t},\overline{\MM{t}}} \ge \dr{S,\overline{S}}$. Thus for all $T_i \le t < T_i'$ and all possible values $X$ of $\MM{t}$, by Lemma~\ref{lem:drift-works},
	\begin{align*}
	\E\big(\phi(\MM{t+1})-\phi(\MM{t}) \mid \MM{t} = X\big) \ge \frac{r-1}{rn}\dr{X,\overline{X}} \ge \frac{r-1}{rn}\dr{S,\overline{S}}.
	\end{align*}
	By applying Lemma~\ref{lem:he-yao} with $Y_t = \MM{t-T_i}$, $\Psi = \phi$, $k_1 = \min\{\phi(S) + (\log n)^2,\phi(V)\}$, $k_2 = \frac{r-1}{rn}\dr{S,\overline{S}}$ and $\tau = T_i' - T_i$, we obtain 
	\[
	\E(T_i'-T_i\mid S_i=S) \le \frac{rn((\log n)^2+1)}{(r-1)\dr{S,\overline{S}}}.
	\]
	Now summing over all $i \in [K-1]$, by~\eqref{eqn:tabs-2-1} it follows that
	\begin{equation}\label{eqn:tabs-2-2}
	\E\left(\sum_{i=1}^{K-1} (T_i'-T_i) \right) \le \frac{2rn(\log n)^2}{r-1}\E\left(\sum_{i=1}^{K-1} \frac{1}{\dr{S_i,\overline{S_i}}} \right).
	\end{equation}
	
	We now bound $\sum_{i=1}^{K-1} \frac{1}{\dr{S_i,\overline{S_i}}}$ above. For all $i \ge 0$, let $I_i = (2n\bs(n)e^i, 2n\bs(n)e^{i+1}]$, and let $\mathcal{P}_i = \{S_j \mid 0\le j \le K-1,\,1/\dr{S_j,\overline{S_j}} \in I_i\}$. For all $i \in [K-1]$, since $S_i$ is a barrier we have $\dr{S_i,\overline{S_i}} < 1/2n\bs(n)$; and since $G$ is connected, there must be some edge from $S_i$ to $V \setminus S_i$, and so $\dr{S_i,\overline{S_i}} \ge 1/n^2$. It follows that writing $\gamma = \max\{x \mid 2\bs(n)e^x \le n\}$, 
	\begin{equation}\label{eqn:bdd-drift}
	\sum_{i=1}^{K-1} \frac{1}{\dr{S_i,\overline{S_i}}} \le \sum_{i=0}^{\gamma}\big(|\mathcal{P}_i|\cdot 2n\bs(n)e^{i+1}\big).
	\end{equation}
	
	Let $0 \le k \le \gamma$, and write $\mathcal{P}_k$ = $\{S_{i_1}, \dots, S_{i_{|\mathcal{P}_k|}}\}$ where $i_1 < \dots < i_{|\mathcal{P}_k|}$. Since $S_1 \subset \dots \subset S_{K-1}$ (by the definition of $K$), we have $|S_{i_{|\mathcal{P}_k|}}| \ge \sum_{j=1}^{|\mathcal{P}_k|-1} |S_{i_{j+1}} \setminus S_{i_j}|$. By Lemma~\ref{lem:more-drift-props}(\ref{d:sub}), for all $j \in [|\mathcal{P}_k|-1]$ we have $\dr{S_{i_{j+1}} \setminus S_{i_j},\overline{S_{i_{j+1}} \setminus S_{i_j}}} \le \dr{S_{i_{j+1}},\overline{S_{i_{j+1}}}} +\dr{S_{i_j},\overline{S_{i_j}}} \le 1/n\bs(n)e^k$. Note that $i_j \le K-1$ for all $j$, so $S_{i_j+1} \setminus S_{i_j}$ is non-empty; hence by Lemma~\ref{lem:drift-props}(\ref{d:size}) applied to each set $S_{i_{j+1}} \setminus S_{i_j}$, it follows that 
	\[
	n \ge |S_{i_{|\mathcal{P}_k|}}| \ge \sum_{j=1}^{|\mathcal{P}_k|-1} |S_{i_{j+1}} \setminus S_{i_j}| \ge (|\mathcal{P}_{k}|-1)\cdot \frac{\bs(n)e^k}{2}.
	\]
	Rearranging, we obtain
	\begin{equation*}
	|\mathcal{P}_k| \le \frac{2n}{\bs(n)e^k}+1 \le \frac{3n}{\bs(n)e^k} \mbox{ for all }0 \le k \le \gamma.
	\end{equation*}
	By~\eqref{eqn:bdd-drift}, it follows that
	\[
	\sum_{i=1}^{K-1} \frac{1}{\dr{S_i,\overline{S_i}}} \le (\gamma+1)\cdot 6en^2 \le 12en^2\log n.
	\]
	The result now follows by~\eqref{eqn:tabs-2-2}.
\end{proof}

\begin{lemma}\label{lem:tabs-2}
	Let $r > 1$, let $G = (V,E)$ be a connected $n$-vertex graph, and let $\MM$ be a Moran process on $G$ with fitness $r$. If $n$ is sufficiently large, then $\E\left(\sum_{i=0}^{K-1}(T_{i+1}-T_i') \right) \le \frac{4r}{r-1}n^3\bs(n)$.
\end{lemma}
\begin{proof}
	We will prove the result by applying Lemma~\ref{lem:he-yao} to a rescaled version of $\MM$ which only runs on the intervals $[T_i', T_{i+1}]$ and $[\Tabs,\infty)$. To this end, we define a ``time compression'' map $\tau\colon\mathbb{N}\rightarrow\mathbb{N}$ as follows. Let
	\[
		\mathcal{T} = \bigcup_{i=0}^\infty [T_i',T_{i+1}] \cup [\Tabs,\infty).
	\]
	Then write $\mathcal{T} = \{\tau(0), \tau(1), \tau(2), \dots\}$ with $\tau(0) < \tau(1) < \dots$, let $X_t = \MM{\tau(t)}$, and let $\tau^{-1}\colon\mathcal{T}\rightarrow\mathbb{N}_{\ge 0}$ be the inverse of $\tau$ on its image. Then we have
	\begin{equation}\label{eqn:elboreth}
		\E\left(\sum_{i=1}^{K-1}(T_{i+1}-T_i') \right) \le \E(\tau^{-1}(T_K)).
	\end{equation}
	In Lemma~\ref{lem:he-yao}, we will take $Y = X$, $\Psi = \phi$, $k_1 = \phi(V)$, $k_2 = (r-1)/2rn^2\bs(n)$ and $\tau = \tau^{-1}(T_K)$. We now show that the conditions of Lemma~\ref{lem:he-yao} are satisfied. Note that for all $i$, $\tau(i)$ is a stopping time of $M$, and so $X$ is a Markov chain. Observe that if $M(T_i) \in \{\emptyset,V\}$, then $\tau^{-1}(T_i') = \tau^{-1}(T_i)$, and otherwise we have $T_i' \ne T_i$ and hence $\tau^{-1}(T_i') = \tau^{-1}(T_i)+1$. Moreover,
	\[
		\tau^{-1}(T_{i+1}) = \min\{t \ge \tau^{-1}(T_i') \mid X_t \in \{\emptyset,V\}\mbox{ or $X_t$ is a barrier}\}.
	\]
	Thus each of the times $\tau^{-1}(T_i)$ and $\tau^{-1}(T_i')$ is a stopping time of $X$, and (by the definition of $K$) so is $\tau^{-1}(T_K)$. Condition (ii) of Lemma~\ref{lem:he-yao} is clearly satisfied, and condition (i) is satisfied since $\Tabs \in \mathcal{T}$; we next show that condition (iii) is satisfied.
	
	Suppose that $i \ge 0$ and $A \subseteq V$ are such that the events $\tau(i) < T_K$ and $X_i = A$ are consistent; in particular, this implies $\emptyset \subset A \subset V$. If $A$ is not a barrier, then we must have $\tau(i) \in [T_j', T_{j+1}-1]$ for some $j$ and $\tau(i+1) = \tau(i)+1$. Moreover, $\dr{A,\overline{A}} \ge 1/2n\bs(n)$. Thus by Lemma~\ref{lem:drift-works}, it follows that
	\[
		\E\big(\phi(X_{i+1}) - \phi(X_{i}) \mid X_{i} = A\big) \ge \frac{r-1}{rn}\dr{A,\overline{A}} \ge \frac{r-1}{2rn^2\bs(n)} = k_2.
	\]
	If instead $A$ is a barrier, then for some $j$ we must have $\tau(i) = T_j$, $X_i = S_j$ and $\tau(i+1) = T_j'$. By the definitions of $T_j'$ and $T_{j+1}$, one of the following must hold.
	\begin{itemize}
		\item Either $M(T_j') = \emptyset$, or $M(T_j')$ is a barrier with $M(T_j')\not\supset S_j$. In this case, $T_{j+1} = T_j'$ and hence $S_{j+1} = M(T_j') \not\supset S_j$.
		\item Either $M(T_j') = V$ or $M(T_j')$ is a barrier with $M(T_j') \supset S_j$. In this case, $X_{i+1} = M(T_j') \supset S_j = X_i$, so $\phi(X_{i+1}) \ge \phi(X_i) + 1/n$.
		\item $\phi(M(T_j')) \ge \phi(M(T_j)) + (\log n)^2$. In this case, we also have $\phi(X_{i+1}) \ge \phi(X_i) + 1/n$.
	\end{itemize}
	Thus in all cases, we have $S_{j+1} \not\supset S_j$ or $\phi(M(T_j')) \ge \phi(M(T_j)) + 1/n$.
	The former case occurs with probability at most $1/n^6$ conditioned on $X_i=A$ by Lemma~\ref{lem:blockchain}, so it follows that if $n$ is sufficiently large,
	\[
		\E\big(\phi(X_{i+1}) - \phi(X_{i}) \mid X_i = A\big) \ge - \frac{\phi(V)}{n^6} + \left(1 - \frac{1}{n^6}\right)\frac{1}{n} \ge \frac{1}{2n} > k_2.
	\]
	Thus in either case, condition (iii) of Lemma~\ref{lem:he-yao} is satisfied.
	
	It now follows by Lemma~\ref{lem:he-yao} that
	\[
		\E(\tau^{-1}(T_K)) \le \frac{\phi(V) - \phi(\MM{0}) + 1}{k_2} \le \frac{4rn^2\phi(V)\bs(n)}{r-1}.
	\]
	Since $\phi(V) \le n$, the result therefore follows by~\eqref{eqn:elboreth}.
\end{proof}

\begin{lemma}\label{lem:tabs-3}
	Let $r > 1$, let $G = (V,E)$ be a connected $n$-vertex graph, and let $\MM$ be a Moran process on $G$ with fitness $r$. If $n$ is sufficiently large, then $\E(\Tabs - T_K) < 1$.
\end{lemma}
\begin{proof}
	We have
	\[
	\E(\Tabs-T_K) = \E\big(\Tabs-T_K \mid \MM{T_K} \notin \{\emptyset,V\}\big)\Pr\big(\MM{T_K} \notin \{\emptyset,V\}\big).
	\]
	By Corollary~\ref{cor:blockchain}, if $n$ is sufficiently large then we have $\MM{T_K} \in \{\emptyset,V\}$ with probability at least $1-1/n^5$. Moreover, by~\cite[Theorem~9]{DGMRSS2014}, the expected absorption time of $\MM$ from any state is at most $rn^4/(r-1)$. It follows that
	\[
	\E(\Tabs-T_K) \le \frac{1}{n^5}\cdot \frac{rn^4}{r-1},
	\]
	and so the result follows.
\end{proof}

\begin{theorem}\label{thm:tabs-detail}
	Let $r \ne 1$ be positive. There exists $C$ (depending on $r$) such that the following holds. Let $G = (V,E)$ be a connected $n$-vertex graph, and let $\MM$ be a Moran process on $G$ with fitness $r$. If $n$ is sufficiently large, then the expected absorption time of $\MM$ is at most $n^3e^{C(\log \log n)^3}$.
\end{theorem}
\begin{proof}
	As in the proof of Theorem~\ref{thm:disadvantageous}, $\MM$ is dual to a Moran process $\MM'$ on $G$ with fitness $1/r$ and initial state $V\setminus \MM(0)$, obtained by switching the roles of mutants and non-mutants. Since $\MM'$ absorbs precisely when $\MM$ does, we may assume without loss of generality that $r>1$.

	By~\eqref{eqn:tabs} and Lemmas~\ref{lem:tabs-1},~\ref{lem:tabs-2} and~\ref{lem:tabs-3}, we have
	\[
		\Tabs \le \frac{24er}{r-1}n^3(\log n)^3 + \frac{4r}{r-1}n^3\bs(n) + 1.
	\]
	When $n$ is sufficiently large, the middle term dominates and the result follows by the definition of $\bs(n)$.
\end{proof}

Theorem~\ref{thm:abs-time-intro} now follows immediately from Theorem~\ref{thm:tabs-detail}.

\section{Directed suppressors}
\label{sec:dir-supp}
\newcommand{\dirG}{G_{k,a}}

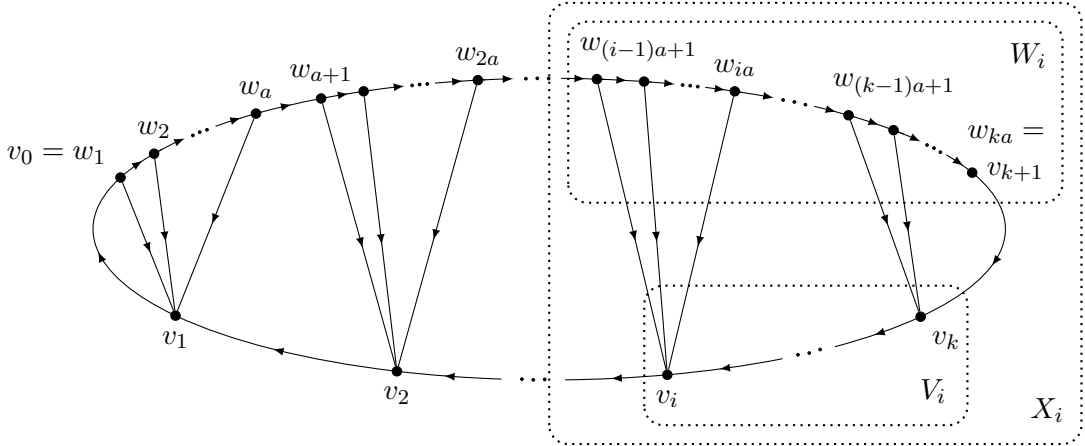
\begin{figure}
\begin{center}

\begin{tikzpicture}[scale=1,node distance = 1.5cm]
\tikzstyle{dot}   =[fill=black, draw=black, circle, inner sep=0.1mm]
\tikzstyle{vertex}=[fill=black, draw=black, circle, inner sep=1.25pt]

\tikzset{
    partial ellipse/.style args={#1:#2:#3}{
        insert path={+ (#1:#3) arc (#1:#2:#3)}
    }
}
\tikzset{->-/.style={decoration={
  markings,
  mark=at position .55 with {\arrow{latex}}},postaction={decorate}}}
\newcommand{\arcfromm}[2]{\draw[->-] (0,0) [partial ellipse=#1:#2:6 and 2];}
\newcommand{\arcfrome}[2]{\draw[-latex] (0,0) [partial ellipse=#1:#2:6 and 2];}

    \node[vertex] (w1) at (160:6 and 2) [label=135:{$v_0=w_1$}] {};
    \node[vertex] (w2) at (150:6 and 2) [label=90:$w_2$] {};
    \node[vertex] (w3) at (130:6 and 2) [label=90:$w_a$] {};

    \node[dot] (d1) at (140:6 and 2) {};
    \node[dot]      at ($(d1)!1mm!(w2)$) {};
    \node[dot]      at ($(d1)!1mm!(w3)$) {};

    \arcfromm{160}{148};
    \arcfrome{150}{143};
    \arcfromm{137}{129};

    \node[vertex] (w4) at (120:6 and 2) [label=90:$w_{a+1}$] {};
    \node[vertex] (w5) at (114:6 and 2) {};
    \node[vertex] (w6) at ( 99:6 and 2) [label=90:$w_{2a}$] {};

    \node[dot] (d2) at (106.5:6 and 2) {};
    \node[dot]      at ($(d2)!1mm!(w5)$) {};
    \node[dot]      at ($(d2)!1mm!(w6)$) {};

    \arcfromm{130}{120};
    \arcfromm{120}{113};
    \arcfrome{114}{108.5};
    \arcfromm{104.5}{98};

    \node[vertex] (w7) at (84:6 and 2) {};
    \node[vertex] (w8) at (78:6 and 2) {};
    \node[vertex] (w9) at (66:6 and 2) [label=90:$w_{ia}$] {};

    \node at ($(w7)+(0.55,0.4)$) {$w_{(i-1)a+1}$};

    \node[dot] (d3) at (91.5:6 and 2) {};
    \node[dot]      at ($(d3)!1.5mm!(w6)$) {};
    \node[dot]      at ($(d3)!1.5mm!(w7)$) {};

    \node[dot] (d4) at (72:6 and 2) {};
    \node[dot]      at ($(d4)!1mm!(w8)$) {};
    \node[dot]      at ($(d4)!1mm!(w9)$) {};

    \arcfrome{99}{94.5};
    \arcfromm{88.5}{83};
    \arcfromm{84}{77};
    \arcfrome{78}{74};
    \arcfromm{70}{65};

    \node[vertex] (w10) at (49:6 and 2) {};
    \node[vertex] (w11) at (41:6 and 2) {};
    \node[vertex] (w12) at (22:6 and 2) {};
    \node[align=left] at ($(w12)+(0.5,0.25)$) {$w_{ka} ={}$\\$\ \ v_{k+1}$};

    \node at ($(w10)+(0.6,0.4)$) {$w_{(k-1)a+1}$};

    \node[dot] (d5) at (57.5:6 and 2) {};
    \node[dot]      at ($(d5)!1.5mm!(w9)$) {};
    \node[dot]      at ($(d5)!1.5mm!(w10)$) {};

    \node[dot] (d6) at (32.5:6 and 2) {};
    \node[dot]      at ($(d6)!1mm!(w11)$) {};
    \node[dot]      at ($(d6)!1mm!(w12)$) {};

    \arcfrome{66}{60.5};
    \arcfromm{54.5}{48};
    \arcfromm{49}{39};
    \arcfrome{41}{35.5};
    \arcfromm{29.5}{20.5};

    \node[vertex] (v1) at (215:6 and 2) [label=-90:$v_1$] {};
    \arcfromm{215}{160};
    \foreach \x in {w1, w2, w3} { \draw[->-] (\x)--(v1); }

    \node[vertex] (v2) at (250.5:6 and 2) [label=-90:$v_2$] {};
    \arcfromm{250.5}{215};
    \foreach \x in {w4, w5, w6} { \draw[->-] (\x)--(v2); }
    
    \node[vertex] (v3) at (285:6 and 2) [label=-90:$v_i$] {};

    \node[dot] (d7) at (268:6 and 2) {};
    \node[dot]      at ($(d7)!1.5mm!(v3)$) {};
    \node[dot]      at ($(d7)!1.5mm!(v2)$) {};

    \arcfromm{285}{272};
    \arcfromm{264}{250.5};
    \foreach \x in {w7, w8, w9} { \draw[->-] (\x)--(v3); }
    
    \node[vertex] (v4) at (324.5:6 and 2) [label=-75:$v_k$] {};

    \node[dot] (d8) at (304.75:6 and 2) {};
    \node[dot]      at ($(d8)!1.5mm!(v4)$) {};
    \node[dot]      at ($(d8)!1.5mm!(v3)$) {};

    \arcfromm{324.5}{308.75};
    \arcfromm{300.75}{285};
    \foreach \x in {w10, w11} { \draw[->-] (\x)--(v4); }

    \arcfromm{22}{-35.5};

    \draw[rounded corners=2.5mm,dotted,thick]
        (0.25,0.35) rectangle (6.75,2.75);
    \node at (6.3,2.3) {$W_i$};

    \draw[rounded corners=2.5mm,dotted,thick]
        (1.25,-2.6) rectangle (5.5,-0.75);
    \node at (5.05,-2.15) {$V_i$};

    \draw[rounded corners=2.5mm,dotted,thick]
        (0,-2.85) rectangle (7,3);
    \node at (6.55,-2.4) {$X_i$};

\end{tikzpicture}
\caption{The directed suppressor $G_{k,a}$.}
\label{fig:dir-supp}
\end{center}
\end{figure}

In this section, we exhibit a family of directed graphs which suppress the effects of fitness on fixation probability, proving Theorem~\ref{thm:dir-nophase}. We first define our graphs. Let $k,a \ge 1$ be integers. Let $w_1, \dots, w_{ka}, v_1, \dots, v_k$ be distinct vertices. For all $i \in [k]$, let $I_i = \{w_{(i-1)a+1}, \dots, w_{ia}\}$. Form the graph $\dirG$ from the directed cycle $w_1w_2\dots w_{ka} v_k v_{k-1}\dots v_1w_1$ by adding $\bigcup_{i\in [k]} (I_i \times \{v_i\})$ to the edge set.  See Figure~\ref{fig:dir-supp}.
For convenience, we use $v_0$ as another name for $w_1$
and we use $v_{k+1}$ as another name for $w_{k a}$.

For all $j \in [k]$, let 
$W_j =  I_j \cup \dots \cup I_k$,
$V_j = \{v_j, \dots, v_k \}$, and
$X_j = W_j \cup  V_j$. 
Let $W_{k+1} = V_{k+1} = X_{k+1}= \emptyset$.

Fix $r>1$, 
$a = \ceil{4r}$,  and $k \ge 2$.
Intuitively, a Moran process on~$\dirG$ with fitness~$r$ behaves as follows.  If the initial mutant is on some~$v_i$ with $i>0$, then it is roughly twice as likely to be replaced by a non-mutant spawned from an adjacent~$w_j$ than it is to spawn a mutant to~$v_{i-1}$.  Therefore, the probability that the mutants will reach~$v_0$ before going extinct is exponentially small in~$i$.  Alternatively, suppose that the initial mutant is at some~$w_j$, adjacent to~$v_i$.  Even if the mutation spreads to vertices $w_{j+1}, \dots, w_{ka}$ and $v_k, v_{k-1}, \dots, v_i$, once $v_{i-1}$ becomes a mutant, it is again roughly twice as likely to be replaced by a non-mutant spawned from a neighbouring~$w$ than to spawn a mutant to~$v_{i-2}$ so, each time a mutant is spawned to~$v_{i-1}$ it is, again, exponentially unlikely to reach~$v_0$.  Thus, with high probability, exponentially many attempts would be required before one succeeded in getting a mutant to~$v_0$.  However, in that time, the non-mutant on~$w_{j-1}$ is almost certain to have replaced all the $w$'s and all the~$v$'s with non-mutants, leading to extinction.

Formally, we will  consider a Moran process $M$ on $\dirG$ with fitness~$r$.
As $M$ evolves, certain events will be relevant for us.
These depend on a fixed $i\in \{2,\ldots,k\}$ (which will be clear from context whenever they are used) and are identified as follows.
\begin{itemize}
	\item $M$ ``loses'' at $t$ if $w_1 \notin M(t-1)$ and $w_1\in M(t)$.
	\item $M$ ``wins'' at $t$ if $M(t-1)\setminus X_{i+1}$ is non-empty, but $M(t) \subseteq X_{i+1}$.
	\item $M$ ``spawns'' at $t$ if  
	$M(t-1) \cap \{v_1,\ldots,v_i\} = \{v_i\}$
	and $M(t) \cap \{v_1,\ldots,v_i\} = \{v_{i-1},v_i\}$.
	\item $M$ ``progresses'' (towards winning) at $t$ if $M(t) \cap I_i$ is a strict subset of $M(t-1) \cap I_i$ (which implies that $M(t-1)\cap I_i\neq\emptyset$).
\end{itemize}
``Losing'', ``winning'' and ``progressing'' are named from the perspective of the non-mutants.
All of these events are disjoint except that $M$ might progress at~$t$ and also win at~$t$.  
We say that a state $M(t)$ is ``consistent''
if, for all~$j$, $w_{j} \in M(t)$ implies that $w_{j'}\in M(t)$ for all $j'>j$.

\begin{observation}\label{obs:dir-duh}
	If $M(0)$ is consistent and $w_1 \notin M(0), \ldots, M(T)$
	then, for all $t\in[T]$, $M(t)$ is consistent
	and  
	$M(t) \cap W_1 \subseteq M(t-1) \cap W_1$. \end{observation}

\begin{observation}\label{obs:dir-tspawn}
	Fix $i\in \{2,\ldots,k\}$.
	If $M(0) \subseteq X_i$
	and
	$M(0) \setminus X_{i+1}$ is non-empty
	then $M$ must win or spawn before it can lose.
\end{observation}

In the following lemmas, we take $\min(\emptyset) = \infty$.

\begin{lemma}\label{lem:dir-tspawn} 
	Let $r>1$, let $a=\ceil{4r}$, and let $k\ge 2$. 
	Let $M$ be a Moran process on $G_{k,a}$ with fitness $r$.	
	Fix $i\in \{2,\ldots,k\}$, and
	suppose that
	$M(0)$ is a consistent state  such that
	$M(0)   \subseteq W_i \cup V_{i-1}$  
	and  
	$v_{i-1} \in M(0)$. 
	Then,
	with probability at least $1-2^{-i+2}$,
    \[
        \min\{t\geq 0\mid M\textnormal{ wins or spawns at }t\}
          < \min\{t\geq 0\mid M\textnormal{ loses at }t\}\,.
    \]
\end{lemma}

Recall that the definitions of ``win'' and ``spawn'' are specific to the chosen value of~$i$.  For example, the process~$M$ can only spawn at a time~$t$ if $v_{i-1}\notin M(t-1)$.  In contrast to the deterministic statement of Observation~\ref{obs:dir-tspawn}, the hypothesis of the lemma has $v_{i-1}\in M(0)$ so it is possible (albeit, as we show, unlikely) for $M$ to lose without either winning or spawning, e.g., by reproducing from $v_{i-1}$ to $v_{i-2}$, then to $v_{i-3}$ and so on, until it reaches~$v_0=w_1$.

\begin{proof}
    Note that, with probability~$1$, at most one of the quantities in
    the inequality is infinite.

    Let $h(t) = \min(\{j\in\{0, \dots, i\}\mid v_j\in M(t)\} \cup
    \{i+1\})$.
	The constraints on $M(0)$ in the statement of the lemma  guarantee
	that $h(0)=i-1$.  For any~$t$, if $M$ loses at time $t$ then
    $h(t)=0$.  If $M$~does not lose at any time in $\{0, \dots, t\}$, then:
    \begin{itemize}
    \item for all $t' \in \{0, \dots, t\}$, $h(t')>0$;
    \item $h(t)-h(t-1)\geq -1$;
	\item if $M$ wins  at time $t$ then $h(t) = i+1$; 
	\item if $M$ spawns at time $t$ then $h(t-1) = i$. 
    \end{itemize}

	Consider any positive integer~$T$ such that $M$~does not win or
    lose at any time $t\leq T$ and $h(T)\geq i$.
	Then as $M$ evolves from $M(T)$, $M$ cannot lose before it has won or spawned, as in Observation~\ref{obs:dir-tspawn}.
	Thus, to prove the lemma, it suffices to consider how $h(t)$
	changes and to prove that, with probability at least $1-2^{-i+2}$, 
	$h(t)$  becomes at least~$i$ before it hits~$0$.
	
	So consider any  state $M(t)$ with 
	$h(t) \in \{1,\ldots,i-1\}$, and suppose $M$ does not win or lose in $\{0, \dots, t-1\}$.
	By hypothesis and Observation~\ref{obs:dir-duh}, $M(t)$ is consistent 
	and $M(t) \cap W_1 \subseteq W_i$. Thus $M(t) \cap I_j = \emptyset$ for all $j \le h(t)$. 
	The event $h(t+1)=h(t)-1$ occurs precisely when $v_{h(t)}$ fires a mutant onto $v_{h(t)-1}$.
	This happens with probability $r/F$, where $F$ is the total fitness of $M(t)$.
	When some vertex in $I_{h(t)}$ fires a non-mutant onto $v_{h(t)}$,
    the event $h(t+1)>h(t)$ occurs.
	This event happens with probability at least $a/(2 F) \geq 2r/F$.
	
	Thus, the progress of $h(t)$ is dominated from below by a gambler's ruin which 
	starts at state $i-1$, absorbs at $0$ and $i$,
	goes up by one with probability $2/3 $
	and goes down by one with probability $ 1/3$.
	Thus (see e.g.\ \cite[Chapter XIV]{Fel1968:Probability}),
	the probability that it hits~$0$ is at most 
	$ {(\tfrac12)^{i}}/{(1-(\tfrac12)^i)}\leq   2^{-i+2}$. 
\end{proof}

The following lemma shows that, with sufficiently high probability, after starting at state~$X_i$, 
$M$~wins before it loses.

\begin{lemma}\label{lem:dir-fixate}
	Let $r>1$, let $a = \ceil{4r}$, let $k \ge 2$, let $2 \le i \le k$. 
	Let $M$ be a Moran process on $\dirG$ with fitness $r$ and initial state $X_i$. 
	Then with probability at least $1-2^{-i+4}ar$, 
	$$\min\{t \mid M(t) \subseteq X_{i+1}\} < \min\{t \mid w_1 \in M(t)\}.$$ 
\end{lemma}

\begin{proof}
	Let $\tau_0=0$. For $j>0$,
	let $\tau_j = \min \{t > \tau_{j-1} \mid \mbox{$M$ loses, wins, or spawns at~$t$}\}$.
	Let $\calE_j$ be the event that $M$ spawns at $\tau_1,\ldots,\tau_j$ and loses at $\tau_{j+1}$.
	We must show that $\sum_{j=0}^\infty \Pr(\calE_j) \leq 2^{-i+4} ar$.
	
	If $\calE_j$ occurs, then $M(\tau_j)= Y$
	for some state $Y$ that  
	is consistent 
	and satisfies $ Y\subseteq W_i \cup V_{i-1}$ (by Observation~\ref{obs:dir-duh}), and
	such that
	$v_{i-1} \in Y$ (since $M$ spawns at $\tau_j$). By Lemma~\ref{lem:dir-tspawn},
	$\Pr(\calE_j \mid M(\tau_j)= Y) \leq 2^{-i+2}$.
	
	Now let $S = \{ j>0 \mid \text{$M$ spawns at all of $\tau_1,\ldots,\tau_j$} \}$.
	Then
	\begin{align*}
	\sum_{j=0}^\infty \Pr(\calE_j) 
	&\leq 2^{-i+2} \sum_{j=0}^\infty \sum_Y
	\Pr( \text{$M$ spawns at all of $\tau_1,\ldots,\tau_j$ and $M(\tau_j)=Y$}) \\
	&= 2^{-i+2} \sum_{j=0}^\infty 
	\Pr( \text{$M$ spawns at all of $\tau_1,\ldots,\tau_j$})  \\
	&= 2^{-i+2} \mathbb{E}[|S|].
	\end{align*}
	
	So to finish, we need only show that $\mathbb{E}[|S|] \leq 4 ar $.
	To do this, let 
	$\tau'_0=0$. For $j>0$,
	let 
	$\tau'_j = \min \{t > \tau'_{j-1} \mid \mbox{$M$ loses, wins, progresses or spawns at~$t$}\}$.
	We say that $\tau'_j$ is a ``spawning event'' if $M$ spawns at~$\tau'_j$ and
	a ``non-spawning event'' otherwise.

	Let $j^* = \min\{j > 0 \mid \mbox{$M$ wins or loses at $\tau'_j$}\}$.
	From the definition of $S$,  $|S|$ is the number of spawns before the first win or loss,  so
	$|S| = |\{j \in  [j^*-1] \mid \mbox{$\tau'_j$ is a spawning event}\}|$.
	We will 
	show that
	the expected number of spawning events before $\tau'_{j^*}$ is at most $4 ar$,
	but in order to do this, it helps to give an alternative definition of $j^*$.
	$$j^* = \min\left\{j > 0 \ \ \middle| \begin{array}{l}\mbox{$M$ wins or loses at $\tau_j'$ or there are at least}\\\qquad\mbox{ $a+1$ non-spawning events in $\tau'_1,\ldots, \tau'_j$}\end{array} \right\}.$$
	To see that the new definition of~$j^*$ is equivalent to the old one,
	note that $M(0)=X_i$, so $M$ can progress at most $a$ times (after that, its intersection with~$I_i$ is empty)
	so the $(a+1)$st non-spawning event has to be a win or lose.
	
	Using the new definition of $j^*$, we will finish the proof by 
	showing that the expected number of spawning events before $\tau'_{j^*}$ is at most $4 a r$.
	To do this, we will consider any $j< j^*$ and
	show (in the bulleted items below)
	that  for any $t \in [\tau'_j,\tau'_{j+1})$, 
	conditioned on $M(t)$,
	the probability that $M$ has a spawning event at~$t+1$ is at most $2r$ times the probability
	that $M$ has a non-spawning event at~$t+1$.
	Thus the probability that $\tau_{j+1}'$ is a non-spawning event is at least $1/(2r+1)$, and the number of spawning events before~$j^*$
	is dominated above by a negative binomial variable with failure probability $1/(2r+1)$, where we 
	count the successes before $a+1$ failures have occurred.
	The expectation of this variable is at most $(2r+1)(a+1)\leq 4ar$,
	so we can conclude that $\mathbb{E}(|S|) \leq 4ar$.
	
	We conclude by looking at the relevant probabilities.
	Fix $0 < j < j^*$ and $t\in [\tau'_j,\tau'_{j+1})$.
	By Observation~\ref{obs:dir-duh},  $M(t)$ is  consistent and satisfies $M(t) \cap W_1 \subseteq W_i$.
	Also, since $j< j^*$,
	$M(t) \setminus X_{i+1}$ is non-empty.
	Our goal is to show that
	the probability that $M$ has a spawning event at~$t+1$ is at most $2r$ times the probability
	that $M$ has a non-spawning event at~$t+1$. 
	\begin{itemize}
		\item
		If $M(t) \cap \{v_1,\ldots,v_{i-1}\}$ is non-empty then the probability that $M$ has a spawning event at $t+1$ is zero, which is trivially
		at most $2r$ times the probability of  a non-spawning event at $t+1$.
		\item If $M(t) \cap \{v_1,\ldots,v_{i-1}\}$ is empty
		and $F$ is the total fitness of $M(t)$, then the probability that $M$ spawns at $t+1$ is at most $r/F$. 
		We will show that the probability that it has a non-spawning event is at least $1/(2F)$, as required.
		\begin{itemize}
			\item If $M(t) \cap I_i$ is non-empty then the probability that $M$ progresses at $t+1$ is at least $1/(2 F)$.
			\item If $M(t) \cap I_i$ is empty 
			then, since $M(t) \setminus X_{i+1}$ is non-empty, $v_i\in M(t)$.
			So the probability that $M$ wins at $t+1$ is at least $a/(2F) > 1/(2F)$.\qedhere
		\end{itemize}
	\end{itemize}
\end{proof}	

We now convert Lemma~\ref{lem:dir-fixate} into a bound on fixation probability.

\begin{lemma}\label{cor:dir-fixate}
	Let $r>1$, and let $a = \ceil{4r}$. Then for all $i \ge 1$, the probability of a Moran process on $\dirG$ with fitness $r$ fixating from any initial state contained in $X_i$ is at most $2^{-i+5}ar$.
\end{lemma}
\begin{proof}
	Note that if $i\le 7$ then the bound is vacuous since $ar>5$; so suppose $i \ge 8$. For all $Y \subseteq V(\dirG)$, let $\MM^Y$ be a Moran process on $\dirG$ with fitness $r$ and initial state $Y$.  For all $j \in [k]$, let
	\begin{equation*}
	\tau_j = \min\{t\ge 0 \mid \MM{X_j}{t} \subseteq X_{j+1}\mbox{ or }\MM{X_j}{t} = V(\dirG)\}.
	\end{equation*}
	It is well-known that adding mutants to a Moran process does not increase its extinction probability; indeed, this is immediate from~\cite[Theorem 6]{MoranAbsorb}. In particular, for all $Y \subseteq X_j$, the extinction probability of $\MM^Y$ is at least that of $\MM^{X_j}$. It follows that for all $j \in [k]$,
	\begin{align*}
	\Pr(\MM^{X_j}\mbox{ goes extinct}) 
	&\ge \Pr(\MM{X_j}{\tau_j} \subseteq X_{j+1})\cdot \Pr(\MM^{X_{j+1}}\mbox{ goes extinct}).
	\end{align*}
	By Lemma~\ref{lem:dir-fixate}, it follows that
	\[
	\Pr(\MM^{X_j}\mbox{ goes extinct}) \ge (1-2^{-j+4}ar)\cdot \Pr(\MM^{X_{j+1}}\mbox{ goes extinct}).
	\]
	Solving the recurrence relation, since $X_{k+1} = \emptyset$ we obtain
	\[
	\Pr(\MM^{X_i}\mbox{ goes extinct}) \ge \prod_{j=i}^k (1-2^{-j+4}ar) \ge 1 - \sum_{j=i}^\infty 2^{-j+4}ar = 1-2^{-i+5}ar.
	\]
	Finally, we note that for all $Y \subseteq X_i$, we have $\Pr(\MM^Y\mbox{ goes extinct}) \ge \Pr(\MM^{X_i}\mbox{ goes extinct})$. The result follows.
\end{proof}

Our main results now follow easily.

\begin{theorem}\label{thm:dir-fixate}
	Let $r>1$, and let $a=\ceil{4r}$. Then for all $k \ge 2$, the fixation probability of a Moran process on $\dirG$ with fitness $r$ and uniformly random initial mutant is at most $72r\log_2(r+1)/|V(\dirG)|$.
\end{theorem}
\begin{proof}
	Let $m_0 \in V(\dirG)$ be uniformly random, and let $\MM$ be the Moran process on $\dirG$ with fitness $r$ and initial state $\{m_0\}$. We have
	\[
	\Pr(\MM\mbox{ fixates}) = \frac{1+a}{|V(\dirG)|}\sum_{i=1}^k \Pr(\MM \mbox{ fixates} \mid m_0 \in \{v_i\} \cup I_i).
	\]
	It follows by Lemma~\ref{cor:dir-fixate} that
	\[
	\Pr(\MM\mbox{ fixates}) \le \frac{1+a}{|V(\dirG)|} \sum_{i=1}^\infty \min\{1,2^{-i+5}ar\}.
	\]
	Let $c = 6+\ceil{\log_2(5r^2)}$, so that $\sum_{i=c}^\infty 2^{-i+5}ar \le 1$. Then we have
	\[
	\Pr(\MM\mbox{ fixates}) \le \frac{(1+a)}{|V(\dirG)|}\Big(c-1 + \sum_{i=c}^\infty 2^{-i+5}ar\Big) \le \frac{6r(10+2\log_2 r)}{|V(\dirG)|} < \frac{72r\log_2(r+1)}{|V(\dirG)|},
	\]
	as required.		
\end{proof}

\begin{theorem}\label{thm:dir-extinct}
	Let $r>1$, and let $a=\ceil{4r}$. Then for all $k \ge 7+\ceil{\log_2(5r^2)}$,
	the fixation probability of a Moran process on $\dirG$ with fitness $1/r$ and uniformly random initial mutant is at least 
	$1/\big(|V(\dirG)|\,(9r)^{60r^2}\big)$. 
\end{theorem}

\begin{proof}
	Let $n = |V(\dirG)|$.
	Let $m_0 \in V(\dirG)$ be uniformly random. By switching the roles of non-mutants and mutants, it suffices to prove that the extinction probability of a Moran process $\MM$ on $\dirG$ with fitness $r$ and initial state $V(\dirG) \setminus \{m_0\}$ is at least 
	$1/n(9r)^{60r^2}$. 
	
	Let $c = 6+\ceil{\log_2(5r^2)}$.
	We will first crudely lower-bound the probability that $\MM$ reaches state $X_{c+1}$ 
	before absorption, and then apply Lemma~\ref{cor:dir-fixate}. 
	For all $i \in \{1,\ldots,c-1\}$, define the
	$F_i$ to be the following sequence of $a+1$
	edges in $\dirG$.
	\[
	F_i = ((w_{(i-1)a+1}, w_{(i-1)a+2}), (w_{(i-1)a+2}, w_{(i-1)a+3}), \dots, (w_{ia}, w_{ia+1}), (w_{(i-1)a+1}, v_i)).
	\]
	Also, define $F_c$ to be the following sequence of $a$ edges.
	\[
	F_c = ((w_{(c-1)a+1}, w_{(c-1)a+2}), (w_{(c-1)a+2}, w_{(c-1)a+3}), \dots, (w_{c a-1}, w_{c a}), (w_{(c-1)a+1}, v_c)).
	\] 
	
	Let $\mathcal{E}$ be the event that $m_0 = w_1$
	(so the initial state is 
	$V(\dirG) \setminus \{w_1\}$)
	and the first $(a+1)c-1$ active steps of $M$ are   all spawns of non-mutants starting from~$w_1$ and following
	(exactly) the order of edges in   $F_1, \dots, F_c$.
	The resulting state is $X_{c+1}$.
	Given any intermediate state along this
	sequence of states,
	the probability that the next
	active step is the specified one
	is at least $(\tfrac12)/(\tfrac12 + \tfrac{a}{2}+ r+\tfrac12)\geq 1/9r$.  (Here, the first three terms in the first denominator correspond, respectively, to the desired step, up to $a$ of the vertices $w_{(i-1)a+1}, \dots, w_{ia}$ spawning a non-mutant onto~$v_i$, and $v_i$~spawning a mutant onto~$v_{i-1}$. The final term corresponds to $w_{ia+1}$~spawning a non-mutant onto~$w_{ia+2}$, if $i<c$, or $w_{ia}$ spawning a non-mutant onto $w_{ia+1}$, if $i=c$.)  Therefore,
	the probability that 
	$M$ reaches state $X_{c+1}$ before absorption is at least the probability 
	that $\mathcal{E}$ occurs, which is 
	at least 
	$1/n(9r)^{(a+1)c-1}$.
	
	By Lemma~\ref{cor:dir-fixate}, the probability that $\MM$ goes extinct from state $X_{c+1}$ is at least 
	$1 - 2^{-c+5}ar\geq 1/2\geq 1/9r$, so the probability that $M$ goes
	extinct is at least 
	$1/n(9r)^{(a+1)c}$.
	The result follows from the fact that 
	$a+1 \leq 6 r$
	and $c \leq 10r$.  (Obviously, this is not tight, but the precise function of $r$ will not be important for us.)
\end{proof}

Theorem~\ref{thm:dir-nophase} now follows easily from Theorems~\ref{thm:dir-fixate} and~\ref{thm:dir-extinct}, taking $\calG = \{G_{7+\ceil{\log_2(5r^2)},\ceil{4r}} \mid r >1\}$. (Note that fixation probability is increasing in $r$~\cite[Corollary~7]{MoranAbsorb}.)

\section{Undirected suppressors}
\label{sec:undir-supp}

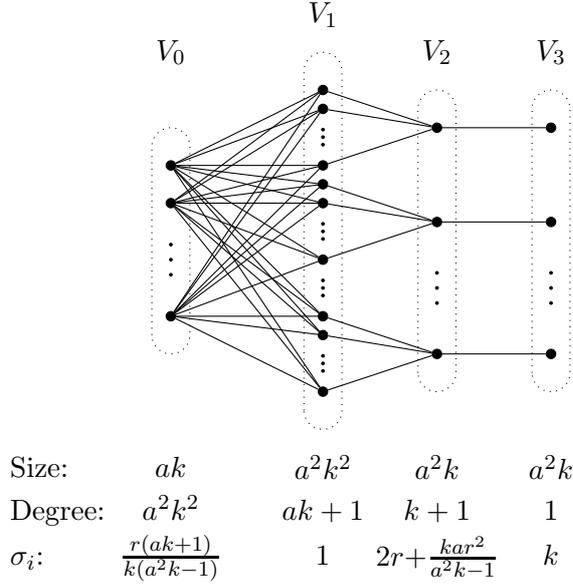
\begin{figure}
\begin{center}

\begin{tikzpicture}[scale=1,node distance = 1.5cm]
\tikzstyle{dot}   =[fill=black, draw=black, circle, inner sep=0.1mm]
\tikzstyle{vertex}=[fill=black, draw=black, circle, inner sep=1.25pt]

    \node[vertex] (v01) at (0,1.50) {};
    \node[vertex] (v02) at (0,1.00) {};
    \node[vertex] (v03) at (0,-0.5) {};

    \node[dot] (d0) at ($(v02)!0.5!(v03)$) {};
    \node[dot]      at ($(d0)!2mm!(v02)$) {};
    \node[dot]      at ($(d0)!2mm!(v03)$) {};

    \node[vertex] (v11) at (2,2.50) {};
    \node[vertex] (v12) at (2,2.25) {};
    \node[vertex] (v13) at (2,1.50) {};

    \node[vertex] (v14) at (2,1.25) {};
    \node[vertex] (v15) at (2,1.00) {};
    \node[vertex] (v16) at (2,0.25) {};

    \node[vertex] (v17) at (2,-0.50) {};
    \node[vertex] (v18) at (2,-0.75) {};
    \node[vertex] (v19) at (2,-1.50) {};

    \node[dot] (d1) at ($(v12)!0.5!(v13)$) {};
    \node[dot]      at ($(d1)!1mm!(v12)$) {};
    \node[dot]      at ($(d1)!1mm!(v13)$) {};

    \node[dot] (d2) at ($(v15)!0.5!(v16)$) {};
    \node[dot]      at ($(d2)!1mm!(v15)$) {};
    \node[dot]      at ($(d2)!1mm!(v16)$) {};

    \node[dot] (d3) at ($(v16)!0.5!(v17)$) {};
    \node[dot]      at ($(d3)!1mm!(v16)$) {};
    \node[dot]      at ($(d3)!1mm!(v17)$) {};

    \node[dot] (d4) at ($(v18)!0.5!(v19)$) {};
    \node[dot]      at ($(d4)!1mm!(v18)$) {};
    \node[dot]      at ($(d4)!1mm!(v19)$) {};

    \node[vertex] (v21) at (3.5,2.00) {};
    \node[vertex] (v22) at (3.5,0.75) {};
    \node[vertex] (v23) at (3.5,-1.0) {};

    \node[dot] (d5) at ($(v22)!0.5!(v23)$) {};
    \node[dot]      at ($(d5)!2mm!(v22)$) {};
    \node[dot]      at ($(d5)!2mm!(v23)$) {};

    \node[vertex] (v31) at (5,2.00) {};
    \node[vertex] (v32) at (5,0.75) {};
    \node[vertex] (v33) at (5,-1.0) {};

    \node[dot] (d6) at ($(v32)!0.5!(v33)$) {};
    \node[dot]      at ($(d6)!2mm!(v32)$) {};
    \node[dot]      at ($(d6)!2mm!(v33)$) {};

    \foreach \x in {1,2,3} {
        \foreach \y in {1, ..., 9} {
            \draw (v0\x) -- (v1\y);
        }
    }

    \foreach \x in {1,2,3} { \draw (v1\x) -- (v21); }
    \foreach \x in {4,5,6} { \draw (v1\x) -- (v22); }
    \foreach \x in {7,8,9} { \draw (v1\x) -- (v23); }

    \foreach \x in {1,2,3} { \draw (v2\x) -- (v3\x); }

    \draw[rounded corners=2.5mm,dotted]
        (-0.25,-1) rectangle (0.25,2);
    \draw[rounded corners=2.5mm,dotted]
        (1.75,-2) rectangle (2.25,3);
    \draw[rounded corners=2.5mm,dotted]
        (3.25,-1.5) rectangle (3.75,2.5);
    \draw[rounded corners=2.5mm,dotted]
        (4.75,-1.5) rectangle (5.25,2.5);

    \node at (0.0,3.0) {$V_0$};
    \node at (2.0,3.5) {$V_1$};
    \node at (3.5,3.0) {$V_2$};
    \node at (5.0,3.0) {$V_3$};

    \node[anchor=west] at (-2.25,-2.5)  {Size:};
    \node at (0.0,-2.5) {$ak$};
    \node at (2.0,-2.5) {$a^2k^2$};
    \node at (3.5,-2.5) {$a^2k$};
    \node at (5.0,-2.5) {$a^2k$};

    \node[anchor=west] at (-2.25,-3.1)  {Degree:};
    \node at (0.0,-3.1) {$\strut a^2k^2$};
    \node at (2.0,-3.1) {$\strut ak+1$};
    \node at (3.5,-3.1) {$\strut k+1$};
    \node at (5.0,-3.1) {$\strut 1$};

    \node[anchor=west] at (-2.25,-3.7) {$\strut \sigma_i$:};
    \node at (0.0,-3.7) {$\strut \frac{r(ak+1)}{k(a^2k-1)}$};
    \node at (2.0,-3.7) {$\strut 1$};
    \node at (3.5,-3.7) {$\strut 2r{+}\frac{kar^2}{a^2k-1}$};
    \node at (5.0,-3.7) {$\strut k$};
\end{tikzpicture}
\caption{The undirected suppressor $H_{a,k}$.}
\label{fig:undir-supp}
\end{center}
\end{figure}

\newcommand{\undG}{H_{a,k}}
In this section, we exhibit a family of strong undirected suppressors, proving Theorem~\ref{thm:undir-suppress}. We first define our graphs. Let $r>1$, and let $a,k\ge 1$ be integers. Let $V_0, \dots, V_3$ be disjoint vertex sets with accompanying weights $\sigma_0, \dots, \sigma_3$, where
\begin{alignat*}{5}
|V_0| &= ak,\qquad &\sigma_0 &= \frac{r(ak+1)}{k(a^2k-1)},\\
|V_1| &= a^2k^2, &\sigma_1 &= 1,\\
|V_2| &= a^2k, &\sigma_2 &= 2r + \frac{kar^2}{a^2k-1},\\
|V_3| &= a^2k, &\sigma_3 &= k.
\end{alignat*}
Define $\undG$ to be a graph with vertex set $V_0 \cup V_1 \cup V_2 \cup V_3$ which is the union of a complete bipartite graph between $V_0$ and $V_1$, $a^2k$ vertex-disjoint $k$-leaf stars between $V_1$ and $V_2$, and a perfect matching between $V_2$ and $V_3$. See Figure~\ref{fig:undir-supp}. For all $0 \le i \le 3$ and all $v \in V_i$, define $\sigma(v) = \sigma_i$. Define a potential function $\sigma\colon2^{V(\undG)} \rightarrow \mathbb{R}_{\ge 0}$ by $\sigma(S) = \sum_{v \in S}\sigma(v)$.

We will first show that if $\MM$ is a Moran process on $\undG$ then, unless $\sigma(\MM{t})$ is fairly large (for example if $\MM{t}$ contains a vertex in $V_3$), it decreases in expectation. This will allow us to use a standard optional stopping theorem argument to upper-bound the fixation probability of the Moran process from any given initial state. The rest of the proof will consist of trivial calculations.

\begin{lemma}\label{lem:undir-martin}
	Let $r>1$, let $a = \ceil{7r^2/2}$, and let $k \ge 28$ be an integer. Let $\MM$ be a Moran process on $\undG$ with fitness $r$. For all $S \subseteq V$ with $0 < \sigma(S) < k$ and all $t \ge 0$, 
	\[
	\E\big(\sigma(\MM{t+1}) - \sigma(\MM{t}) \mid \MM{t} = S\big) \le 0.
	\]
\end{lemma}
\begin{proof}
	For all $t \ge 0$ and all $S \subseteq V$, writing $W(S)$ for the total fitness of state~$S$, we have
	\[
	\E\big(\sigma(\MM{t+1}) - \sigma(\MM{t}) \,\big|\, \MM{t} = S\big) = \frac{1}{W(S)} \sum_{x \in S}\sum_{y \in N(x) \setminus S} \Bigg(\frac{r\sigma(y)}{d(x)} - \frac{\sigma(x)}{d(y)} \Bigg).
	\]
	For all $x \in S$, define
	\[
	f_S(x) = \sum_{y \in N(x) \setminus S} \Bigg(\frac{r\sigma(y)}{d(x)} - \frac{\sigma(x)}{d(y)} \Bigg).
	\]
	We must therefore prove that $\sum_{x \in S}f_S(x) \le 0$ when $0 < \sigma(S) < k$. We split this sum into three parts and bound each part separately.
	
	First, note that since $\sigma(S) < k$, we have $|V_1 \setminus S| > a^2k^2-k$. It follows that
	\begin{equation}\label{eqn:undir-martin-1}
	\mbox{for all $x \in S \cap V_0$, }f_S(x) = \sum_{y \in V_1 \setminus S} \Bigg(\frac{r\sigma_1}{|V_1|} - \frac{\sigma_0}{ak+1}\Bigg) \le r\sigma_1 - (a^2k^2-k)\cdot\frac{\sigma_0}{ak+1} = 0.
	\end{equation}
	Next, let $S'$ be the set of all $x \in S \cap V_1$ such that the unique neighbour of $x$ in $V_2$ also lies in~$S$. We have
	\begin{equation*}
	\mbox{for all $x \in S'$, }f_S(x) \le \sum_{y \in V_0} \frac{r\sigma_0}{ak+1} = \frac{akr\sigma_0}{ak+1} = \frac{ar^2}{a^2k-1}.
	\end{equation*}
	Moreover, 
	\begin{equation*}
	\mbox{for all $x \in S \cap V_2$, }f_S(x) \le \frac{rk\sigma_1}{k+1} + \frac{r\sigma_3}{k+1} - \sigma_2 \le 2r - \sigma_2.
	\end{equation*}
	By the definition of $S'$, we have $|S'| \le k|S \cap V_2|$. It follows that
	\begin{equation}\label{eqn:undir-martin-2}
	\sum_{x \in S' \cup (S \cap V_2)}f_S(x) \le |S\cap V_2|\Bigg(\frac{kar^2}{a^2k-1} + 2r - \sigma_2 \Bigg) = 0\,.
	\end{equation}
	Finally, for all $x \in (S \cap V_1) \setminus S'$,
	\begin{align*}
	f_S(x) 
	&\le \frac{akr\sigma_0}{ak+1} + \frac{r\sigma_2}{ak+1} - \frac{\sigma_1}{k+1} 
	< \frac{ar^2}{a^2k-1} + \frac{r\sigma_2}{ak} - \frac{1}{k+1}
	= \frac{ar^2+r^3}{a^2k-1} + \frac{2r^2}{ak} - \frac{1}{k+1}\\
	&\le \frac{r^3+3ar^2}{a^2k-1}-\frac{1}{k+1}
	= \frac{(k+1)(r^3+3ar^2-a^2)+1+a^2}{(k+1)(a^2k-1)}.
	\end{align*}
	Let $f(x) = r^3+3xr^2-x^2$. Then by considering the derivative of $f$, it is clear that $f$ is decreasing on $[3r^2/2,\infty]$. Thus $f(a) \le f(7r^2/2) \le r^4 + 21r^4/2 - 49r^4/4 = -3r^4/4$. Since $k \ge 28$ and $a \le 9r^2/2$, it follows that
	\[
	(k+1)(r^3+3ar^2-a^2) + 1+a^2 \le -\frac{3(k+1)r^4}{4} + r^4 +\frac{81r^4}{4} \le 0.
	\]
	Thus
	\begin{equation}\label{eqn:undir-martin-3}
	\mbox{for all $x \in S \cap V_1\setminus S'$, } f_S(x) \le 0.
	\end{equation}
	Finally, note that $S \cap V_3 = \emptyset$ since $\sigma(S) < k$. 
	From~\eqref{eqn:undir-martin-1},~\eqref{eqn:undir-martin-2} and~\eqref{eqn:undir-martin-3}, it now follows that 
	\[
	\sum_{x \in S}f_S(x) = \sum_{x \in S \cap V_0} f_S(x) + \sum_{x \in S' \cup (S \cap V_2)} f_S(x) + \sum_{x \in S \cap V_1 \setminus S'} f_S(x) \le 0,
	\]
	as required.
\end{proof}

\begin{theorem}\label{thm:undir-suppress-main}
	Let $r>1$, let $a = \ceil{7r^2/2}$, and let $k \ge 36r$ be an integer. Let $\MM$ be a Moran process on $\undG$ with fitness $r$ whose initial state is the set containing a uniformly random vertex. Then $\Pr(\MM\textnormal{ fixates}) \le 10r^2/\sqrt{|V(\undG)|}$.
\end{theorem}
\begin{proof}
	For all $x \in V(\undG)$, let $\MM^{x}$ be a Moran process on $\undG$ with fitness $r$ and initial state $\{x\}$. Let 
	\[
	\tau^{x} = \min\{t \ge 0 \mid \sigma(\MM{x}{t}) = 0 \mbox{ or }\sigma(\MM{x}{t}) \ge k\}.
	\]
	Then since $\sigma(\MM{x}{t}) = 0$ if and only if $\MM{x}{t} = \emptyset$, we have
	\[
	\Pr(\MM^{x}\mbox{ fixates}) \le \Pr(\sigma(\MM{x}{\tau^{x}}) \ge k) \le \frac{1}{k}\E(\sigma(\MM{x}{\tau^x})),
	\]
    by Markov's inequality.

	For all $t \le \tau^x$, define $\widehat{\MM}^x(t) = \MM{x}{t}$; otherwise, define $\widehat{\MM}^x(t) = \MM{x}{\tau^x}$. Thus $\E(\sigma(\MM{x}{\tau^x})) = \E(\sigma(\widehat{M}^x(\tau^x)))$ and, by Lemma~\ref{lem:undir-martin}, $\sigma(\widehat{M}^x(t))$ is a supermartingale. Thus by the optional stopping theorem, we have $\E(\sigma(\MM{x}{\tau^{x}})) \le \sigma(x)$ and hence
	\begin{equation}\label{eqn:undir-supp}
	\mbox{for all $x \in V(\undG)$, }\Pr(\MM^{x}\mbox{ fixates}) \le \sigma(x)/k.
	\end{equation}
	
	The remainder of the proof is pure calculation. From~\eqref{eqn:undir-supp}, we have
	\[
	\Pr(\MM\mbox{ fixates}) \le \frac{1}{|V(\undG)|} \sum_{i=0}^3 \frac{\sigma_i|V_i|}{k}.
	\]
	Taking the terms individually, since $k \ge 36r$ and $a = \ceil{7r^2/2}$, we have
	\begin{align*}
	\frac{\sigma_0|V_0|}{k} &= \frac{ar(ak+1)}{k(a^2k-1)} = \frac{r}{k}\left(1+\frac{a+1}{a^2k-1}\right) \le \frac{2r}{k},\\
	\frac{\sigma_1|V_1|}{k} &= a^2k,\\
	\frac{\sigma_2|V_2|}{k} &= 2ra^2 + \frac{ka^3r^2}{a^2k-1} \le 2ra^2 + 2ar^2 \le 3ra^2,\\
	\frac{\sigma_3|V_3|}{k} &= a^2k.
	\end{align*}
	Since $k \ge 36r$, it follows that
	\[
	\Pr(\MM^{x}\mbox{ fixates}) \le \frac{1}{|V(\undG)|}\left(\frac{2r}{k} + 3ra^2 + 2a^2k \right) \le \frac{1}{|V(\undG)|}\left(4ra^2 + 2a^2k \right) \le \frac{19a^2k}{9|V(\undG)|}.
	\]
	We have $|V(\undG)| > |V_1| = a^2k^2$, so $ak \le \sqrt{|V(\undG)|}$. Thus
	\[
	\Pr(\MM^{x}\mbox{ fixates}) \le \frac{19a}{9\sqrt{|V(\undG)|}} < \frac{10r^2}{\sqrt{|V(\undG)|}}.\qedhere
	\]
\end{proof}

Theorem~\ref{thm:undir-suppress} now follows immediately, taking $\mathcal{H} = \{H_{\ceil{7r^2/2},\ceil{36r}} \mid r>1\}$. (Note that fixation probability is increasing in $r$~\cite[Corollary~7]{MoranAbsorb}.)

\section{Improved FPRAS for fixation probability}\label{sec:FPRAS}

In this section, we give an improved FPRAS for the fixation
probability of the Moran process on undirected graphs with constant fitness $r>1$ and
a uniformly-chosen initial mutant, thus proving Theorem~\ref{thm:new-fpras}.

We first give a brief overview of prior work. Given a graph $G$ on $n$ vertices and an error tolerance $\eps \in (0,1)$, D\'iaz \emph{et al.}~\cite{DGMRSS2014} proceed using a simple Monte Carlo method. They choose a suitably large integer $N$ depending on $n$ and $\eps$ and simulate $N$ independent iterations (``\emph{runs}'') of the Moran process to absorption. Writing $X^i$ for the indicator variable of the event that the $i$th run fixates, they then output the sample mean of the $X^i$'s. They did not optimise the algorithm, and they simulated $O(n^8\eps^{-4})$ steps of Moran processes in total. Chatterjee \emph{et al.}~\cite{CIN2017} used the same basic approach, but cleaned up the analysis and only simulated the steps of the process in which the state actually changes, as encapsulated by the following definition.

\newcommand{\actM}{M_{\mathsf{act}}}
\begin{definition}\label{def:actm}
	Let $G$ be a graph, and let $M$ be a Moran process on $G$. Let $\tau_0 = 0$, and let $0<\tau_1 < \tau_2 < \dots < \tau_s$ be the times $t$ at which $M(t) \ne M(t-1)$; thus with probability 1, $s < \infty$ and $\tau_s$ is the time at which $M$ absorbs. We define the \textit{active Moran process} $\actM$ as follows. For $t \le s$, let $\actM(t) = M(\tau_t)$, and for $t > s$, let $\actM(t) = \actM(\tau_s)$.
\end{definition}

Writing $\Delta$ for the maximum degree of the input graph $G$, by simulating active Moran processes, Chatterjee \emph{et al.}\ were able to reduce the total number of steps simulated to $\Theta(n^2\Delta\eps^{-2}\log(n\eps^{-1}))$. The most important part of our algorithm is encapsulated in Theorem~\ref{thm:new-fpras}, which we prove in Section~\ref{sec:new-fpras-main}. This theorem says that writing $\dbar$ for the average degree of $G$, we improve over~\cite{CIN2017} by a factor of roughly $n^2\,/\,\dbar$, simulating $O(\Delta\dbar\eps^{-2}\log(\dbar\eps^{-1}))$ steps of active Moran processes in total.

We have two main sources of improvement. The first is an improved lower
bound on fixation probability: 
the bound used in~\cite{CIN2017} is $f_{G,r} \ge 1/n$, which requires they take $N = \Omega(n\epsilon^{-2})$ to ensure concentration of the output. (Note that with fewer than $1/f_{G,r}$ runs, it is reasonably likely that not even a single run will fixate.) Using Corollary~\ref{cor:Dd} it is easy to improve this to $N = \Theta(\overline{d}\epsilon^{-2})$.

The second main source of improvement is that Theorem~\ref{thm:potential} allows
us to stop a run
once the mutant set has high enough potential that extinction is
overwhelmingly unlikely, rather than waiting until absorption as in~\cite{CIN2017} and~\cite{DGMRSS2014}. A mutant set~$X$ goes extinct with
probability at most $r^{-\phi(X)}$ and we may assume that fixation
will occur when this probability is close enough to zero. 
The algorithm in~\cite{CIN2017} requires $\Theta(n\Delta\log (n\epsilon^{-1}))$ simulation steps per run; early termination based on potential allows us to use only $\Theta(\Delta\log(\dbar\eps^{-1}))$ steps per run (on average).

In Section~\ref{sec:new-fpras-cor}, we discuss turning Theorem~\ref{thm:new-fpras} into an FPRAS. Chatterjee \emph{et al.}~\cite{CIN2017} have already shown that a $T$-step run of an active Moran process can be simulated in $O(\Delta T)$ time with $O(n\dbar)$ preprocessing time. While the cost of preprocessing is irrelevant to their time bounds, it would dominate our running time since it is incurred with each run and our runs are shorter on average. We therefore present an alternative sampling algorithm which removes it. While the $O(\Delta\dbar\eps^{-2}\log(\dbar\eps^{-1}))$-step algorithm of Theorem~\ref{thm:new-fpras} is based on the simple Monte Carlo approach, this does not immediately yield an FPRAS for $f_{G,r}$ with running time $O(\Delta^2\dbar\eps^{-2}\log(\dbar\eps^{-1}))$ because if $G$ is sparse this may be substantially smaller than the time required to read $G$. We therefore proceed more carefully and formally.

We work in the standard word RAM model. Thus, given an input of total size $s$, we assume that all standard arithmetic and randomising operations on $O(\log s)$-sized words can be carried out in $O(1)$ time. We assume that $\eps$ is given in the form $p/q$ for some positive integers $p<q$ given in unary. (This ensures that arithmetic operations on $\eps$ can be carried out in $O(1)$ time, which is natural since an FPRAS has running time polynomial in $1/\eps$ rather than $\log(1/\eps)$.) We require that $G=(V,E)$ be presented in what we call ``augmented adjacency-list form'', in which we are given $|V|$, $|E|$, $\Delta(G)$, and, for each $v \in V$, $d(v)$ together with a list of $v$'s neighbours. (See Definition~\ref{def:aug-adj}.) Note that if $G$ is already in adjacency-list form, then it is trivial to put $G$ into augmented adjacency-list form in $O(n\dbar)$ time. We also assume that $G$ is connected and $|V| \ge 2$ (which can easily be verified in $O(n\dbar)$ time), since otherwise the problem is trivial. Under these assumptions, we will prove the following.

\begin{corollary}\label{cor:new-fpras}
	Let $r>1$. Then there is an FPRAS for $f_{G,r}$ whose running time is $O(\Delta^2\dbar\eps^{-2}\log(\dbar\eps^{-1}))$.
\end{corollary}

\subsection{Proof of Theorem~\ref{thm:new-fpras}}\label{sec:new-fpras-main}

We require the following standard Chernoff bound.

\begin{lemma}[{\cite[Corollary~4.6]{MU2005}}]\label{lem:Chernoff}
	Let $Z_1, \dots, Z_N$ be independent Bernoulli trials and let $Z =
	Z_1 + \dots + Z_N$.  For any $\eps\in(0,1)$,
	\[
	\Pr\big(|Z-\E(Z)|\geq \eps \E(Z)\big)
	\leq 2\exp(-\eps^2\E(Z)/3)\,.\pushQED{\qed}\qedhere
	\]
\end{lemma}

We also require the following bound on the expected change in potential at an active
step due to Chatterjee \emph{et al}. Since their terminology is different from ours, we give a short proof for
completeness.

\begin{lemma}[\cite{CIN2017}]\label{lem:CIN-active}
	Let $\MM$ be a Moran process with fitness~$r>1$ on a connected
	graph~$G=(V,E)$ with at least two vertices and maximum degree at
	most~$\Delta$.  For any $\emptyset\subset X\subset V$, and any
	$t\geq 0$, 
	\[
	\E\big(\phi(\actM(t+1)) - \phi(\actM(t))\mid \actM(t) = X\big)
	\geq \frac{r-1}{(r+1)\Delta}\,.
	\]
\end{lemma}
\begin{proof}
	It suffices to prove that for all $t \ge 0$,
	\begin{equation}\label{eqn:CIN-active}
	\E\big(\phi(M(t+1)) - \phi(M(t))\mid M(t+1)\ne M(t),\ M(t) = X\big)
	\geq \frac{r-1}{(r+1)\Delta}\,.
	\end{equation}
	Our argument is very similar to that of~\cite[Lemma~5]{DGMRSS2014}. Fix $t \ge 0$. For convenience, let $\calE_1$ be the event that $M(t)=X$. For all $(x,y) \in E(X,V\setminus X)$, let $\calE_{xy}$ be the event that at time $t+1$, either $x$ spawns onto $y$ or $y$ spawns onto $x$. Let $\calE_2$ be the (disjoint) union of all events $\calE_{xy}$, and note that $M(t+1) \ne M(t)=X$ if and only if $\calE_1 \cap \calE_2$ occurs. We have
	\begin{align}\nonumber
	&\E\big(\phi(M(t+1)) - \phi(M(t)) \mid \calE_1 \cap \calE_2\big)\\\label{eqn:CIN-active-2}
	&\qquad\qquad= \sum_{\substack{(x,y)\in\\ E(X,V\setminus X)}}\E\big(\phi(M(t+1)) - \phi(M(t)) \mid \calE_1 \cap \calE_{xy}\big)\Pr(\calE_{xy} \mid \calE_1 \cap \calE_2).
	\end{align}
	For all $(x,y) \in E(X,V\setminus X)$, we have 
	\begin{align*}
	\E\big(\phi(M(t+1)) - \phi(M(t)) \mid \calE_1 \cap \calE_{xy}\big) &= \frac{r/d(x)}{r/d(x)+1/d(y)}\cdot\frac{1}{d(y)} - \frac{1/d(y)}{r/d(x)+1/d(y)}\cdot\frac{1}{d(x)}\\
	&= \frac{r-1}{rd(y)+d(x)} \ge \frac{r-1}{(r+1)\Delta}.
	\end{align*}
	Thus by~\eqref{eqn:CIN-active-2} we obtain~\eqref{eqn:CIN-active}, and hence the result.
\end{proof}

\begin{lemma}\label{lem:fpras-run}
	Let $G=(V,E)$ be an $n$-vertex connected graph with maximum degree at most $\Delta$, where $n \ge 2$, and let $M$ be a Moran process on $G$ with fitness $r>1$. Let $0 < P \le \phi(V)$, and let $\tau = \min\{t \ge 0 \mid \actM(t) = \emptyset \mbox{ or }\phi(\actM(t)) \ge P\}$. Then:
	\begin{enumerate}[(i)]
		\item $\Pr(M\mbox{ goes extinct} \mid \phi(\actM(\tau)) \ge P) \le r^{-P}$;
		\item $\E(\tau) \le 2r(P+1)\Delta/(r-1)$.
	\end{enumerate}
\end{lemma}
\begin{proof}
	We first prove (i). Suppose $\phi(X) \ge P$. If $P=\phi(V)$, then $M$ cannot go extinct from state $X$; otherwise, by Theorem~\ref{thm:potential}, the probability that $M$ goes extinct from state $X$ is at most $r^{-P}$. 
	
	We now prove (ii). By Lemma~\ref{lem:CIN-active}, for all $t \ge 0$ and all $X \subseteq V$ such that $0<\phi(X)<P$,
	\[
	\E\big(\phi(\actM(t+1)) - \phi(\actM(t)) \mid \actM(t) = X\big) \ge \frac{r-1}{(r+1)\Delta}.
	\]
	It follows by Lemma~\ref{lem:he-yao} applied with $Y_t=\actM(t)$, $\Psi = \phi$, $k_1 = P$, $k_2 = (r-1)/(r+1)\Delta$ and~$\tau$, that
	\[
	\E(\tau) \le (P+1)\frac{(r+1)\Delta}{r-1} \le \frac{2r(P+1)\Delta}{r-1}.\qedhere
	\]
\end{proof}

We are now in a position to prove Theorem~\ref{thm:new-fpras}.

\begin{reptheorem}{thm:new-fpras}
	\statenewfpras
\end{reptheorem}
\begin{proof}
	First note that by a standard argument, to prove the result it suffices to find a RAS which $\eps$-approximates $f_{G,r}$ while simulating $T = O(\Delta\dbar\eps^{-2}\log(\dbar\eps^{-1}))$ steps of active Moran processes \emph{in expectation}. Indeed, given such a RAS, we modify it to output $-1$ if it attempts to simulate more than $27T$ steps, then run it three times and output the median result. By Markov's inequality, each run of the algorithm has failure probability at most $1/3+1/27 = 10/27$. The median only fails to be a valid $\eps$-approximation if at least two of the three runs fail, so the output is correct with probability at least $1 - (\tfrac{10}{27})^3 - 3\cdot(\tfrac{10}{27})^2\cdot\tfrac{17}{27} > \tfrac{2}{3}$. 
	
	We now describe the algorithm. Suppose $G=(V,E)$ has at least two vertices (since otherwise $f_{G,r} = 1$) and that $G$ is connected (since otherwise $f_{G,r} = 0$). Let
	\[
		N = 18\left\lceil\frac{r\overline{d}}{\eps^2(r-1)}\right\rceil, \qquad P=\min\{\ceil{\log_r(6N)},\phi(V)\}.
	\]
	Let $M^1, \dots, M^N$ be independent Moran processes on $G$ with fitness $r$ and uniformly-chosen initial mutants. For each $i \in [N]$, let 
	\[
		\tau^i = \min\{t \ge 0 \mid \actM^i(t) = \emptyset \textnormal{ or }\phi(\actM^i(t)) \ge P\},
	\]
	and let $X^i$ be the indicator variable of the event that $\phi(\actM^i(\tau^i)) \ge P$. Let $Y = \frac{1}{N}\sum_{i=1}^N X^i$. Our algorithm samples each $\actM^i(\tau^i)$ by sampling each state $\actM^i(t)$ for $0 \le t \le \tau^i$, then outputs $Y$.
	
	We now prove correctness. Let $\widehat{X}^i$ be the indicator variable of the event that $M^i$ fixates, and let $\widehat{Y} = \frac{1}{N}\sum_{i=1}^N\widehat{X}^i$; we view each $X^i$ as an estimator for $\widehat{X}^i$, and $Y$ as an estimator for $\widehat{Y}$. By Lemma~\ref{lem:Chernoff}, we have 
	\[
	\Pr\big(|\widehat{Y} - f_{G,r}| \ge \eps f_{G,r}\big) = \Pr\big(|N\widehat{Y} - Nf_{G,r}| \ge \eps Nf_{G,r}\big) \le 2\exp\big({-}\eps^2Nf_{G,r}/3 \big).
	\]
	Corollary~\ref{cor:Dd} shows that $f_{G,r} \ge (r-1)/2r\overline{d}$. Since $N \ge 18r\overline{d}/\eps^2(r-1)$, it follows that $\Pr(|\widehat{Y} - f_{G,r}| \ge \eps f_{G,r}) \le 2e^{-3} < 1/6$. Let $\calE$ be the event that $\widehat{X}^i \ne X^i$ for some $i \in [N]$. Then by a union bound, it follows that
	\begin{equation*}
		\Pr\big(|Y - f_{G,r}| \ge \eps f_{G,r} \big) \le \Pr\big(Y \ne \widehat{Y}\big) + \Pr\big(|\widehat{Y}-f_{G,r}| \ge \eps f_{G,r}\big) \le \Pr(\calE) + 1/6.
	\end{equation*}
	If $P = \phi(V)$ then $\Pr(\calE) = 0$, since $\actM^i(\tau^i) = \phi(V)$ if and only if $\actM^i$ fixates. Otherwise, $P = \ceil{\log_r(6N)}$, so by Lemma~\ref{lem:fpras-run}(i) and a union bound we have $\Pr(\calE) \le Nr^{-P} \le 1/6$. Thus $Y$ is a valid $\eps$-approximation of $f_{G,r}$ with probability at least 2/3, as required.
	
	Finally, we bound the expected number of steps simulated. By Lemma~\ref{lem:fpras-run}(ii), for all $i \in [N]$ we have $\E(\tau^i) \le 2r(P+1)\Delta/(r-1)$. Thus the expected total number of steps simulated is at most
	\[
		\frac{2r(P+1)\Delta}{r-1}\cdot N = O(\Delta N\log N) = O\left(\dbar\Delta\eps^{-2}\log(\dbar\eps^{-1})\right),
	\]
	as required.
\end{proof}

\subsection{Proof of Corollary~\ref{cor:new-fpras}}\label{sec:new-fpras-cor}

We first formally define the form of our input graph $G$. We assume that it has at least two vertices (since otherwise $f_{G,r} = 1$ for all $r$) and is connected (since otherwise $f_{G,r} = 0$ for all $r$).

\begin{definition}\label{def:aug-adj}
	A connected graph $G$ in \emph{augmented adjacency-list form} is presented as a tuple $(n,m,\Delta, \prod_{v \in [n]}(d_v,L_v))$. Here:
	\begin{itemize}
		\item $n \ge 2$ is the number of vertices of $G$, and $[n]$ is the vertex set of $G$.
		\item $m \ge 1$ is the number of edges of $G$.
		\item $\Delta$ is the maximum degree of $G$.
		\item For each $v \in [n]$, $d_v \ge 1$ is the degree of $v$ in $G$.
		\item For each $v \in [n]$, $L_v$ is a non-empty list of $v$'s neighbours in $G$, in arbitrary order.
	\end{itemize}
\end{definition}

\newcommand{\Xx}{B}
\newcommand{\Ax}{A_B}
\newcommand{\Hx}{H_B}
\newcommand{\Nm}{N_\mathsf{mut}}
\newcommand{\Hm}{H_{\mathsf{mut}}}
\newcommand{\dbd}{d_\mathsf{bdry}}
\newcommand{\Pot}{\Phi}
\newcommand{\lcm}{D}

Throughout the rest of the section, we will assume that our input graph $G$ is given in this form, and write $\dbar = 2m/n$, $N = 18\ceil{r\overline{d}/\eps^2(r-1)}$, and $P = \min\{\ceil{\log_r(6N)},\phi(V)\}$ (as in the proof of Theorem~\ref{thm:new-fpras}). 

While simulating an active Moran process $\actM$, we will need to be somewhat careful about bounding the number of arithmetic operations used to keep track of $\phi(\actM(t))$, a sum of fractions with denominators in $[\Delta]$. For this reason, we first precompute the least common multiple $\lcm$ of $1, 2, \dots, \Delta$. It is not hard to show that $\lcm$ has $O(\Delta)$ bits and can be computed in $O(\Delta^2)$ time, but we give a proof below for completeness. We also precompute $N$ and $P' = \ceil{\log_r(6N)}$, which takes $O(\log N) = O(\log(\Delta/\eps))$ time.

Now, let $r>1$, let $M$ be a Moran process on $G$ with fitness $r$, and let $\tau = \min\{t \ge 0 \mid \actM(t) = \emptyset \textnormal{ or }\phi(\actM(t)) \ge P\}$. The algorithm of Theorem~\ref{thm:new-fpras} requires us to sample $\actM(\tau)$, and to determine whether or not $\actM(\tau) = \emptyset$. We will give a randomised algorithm to do this in expected time $O(\Delta)$ per step simulated. Given this, it is immediate that the algorithm of Theorem~\ref{thm:new-fpras} yields a RAS with expected running time $O(\Delta^2\dbar\eps^{-2}\log(\dbar\eps^{-1}))$. By derandomising the running time using Markov's inequality, as in the proof of Theorem~\ref{thm:new-fpras}, we obtain the FPRAS of Corollary~\ref{cor:new-fpras}.

We say an edge $xy \in E(G)$ is \emph{on the boundary} at time $t \ge 0$ if $x \in \actM(t)$ and $y \notin \actM(t)$ or vice versa. For all $v \in [n]$ and all $t \ge 0$, $\dbd(v,t)$ denotes the number of boundary edges incident to $v$ at time $t$ and $\Xx(t)$ denotes the set of all $v \in [n]$ with $\dbd(v,t) > 0$.

\textbf{Precomputing $\boldsymbol{\lcm}$:} Let $p_1, \dots, p_k$ be the primes in $[\Delta]$, and let $T_i = p_i^{\floor{\log_{p_i}(\Delta)}}$ for all $i \in [k]$; thus $\lcm = \prod_{i=1}^k T_i$. We have $T_i \le \Delta$ for all $i$, and by the prime number theorem, $k = O(\Delta/\log \Delta)$. Thus $\lcm = \Delta^{O(\Delta/\log\Delta)}$, and so $\log \lcm = O(\Delta)$. Use the sieve of Eratosthenes to compute $p_1, \dots, p_k$ in $O(\Delta^2)$ time. Then compute each term $T_i$; since $T_i \le \Delta$ for all $i$, this requires $O(\Delta\log \Delta)$ operations on $O(\log\Delta)$-bit numbers. Finally, multiply the $T_i$'s to form $\lcm$; since $\log \lcm = O(\Delta)$, this requires $O(\Delta)$ multiplications of $O(\Delta)$-bit numbers. Recall that in the word RAM model we can perform arithmetic operations on $O(\log \Delta)$-sized words in $O(1)$ time, and hence arithmetic operations on $O(\Delta)$-bit numbers in $O(\Delta)$ time (see e.g.\ Knuth~\cite[Chapter 4.3]{TAOCP}). We have therefore used $O(\Delta^2)$ time in total, as claimed.

\textbf{Defining the data structure:} In order to simulate $\actM$, we will build a data structure to represent its state and potential. Let $\Nm$ and $\Pot$ be integer variables with $0 \le \Nm \le n$ and $0 \le \Pot \le n\lcm$. Let $\Hm$ and $\Hx$ be dynamically-sized hash tables whose keys are in $[n]$; the keys in $\Hm$ have no associated data, and the keys in $\Hx$ map to data entries in $[3n]$. Let $\Ax$ be a dynamically-sized array whose entries are in $\{\mathsf{unoccupied}\} \cup ([n]\times[\Delta])$; we say $\Ax[i]$ is \emph{unoccupied} if $\Ax[i] = \mathsf{unoccupied}$ and \emph{occupied} otherwise. We resize $\Ax$ to ensure that if it is non-empty, then it is between one-third and two-thirds occupied. (Thus when it exceeds two-thirds occupation, we multiply its size by~$4/3$ and, when it falls below one-third full, we multiply its size by~$2/3$ so, in both cases, the new array is very close to half-occupied.)

Over uniform choices of suitable hash functions, any $N$ insertion, deletion and update operations on $\Hm$ and $\Hx$ require $O(N)$ expected time. Moreover, if $\Ax$ is non-empty, it supports sampling a uniform occupied or unoccupied slot in $O(1)$ expected time. Thus accounting for resizing, any $N$ operations taken from insertion, deletion, update and uniform sampling (of occupied or unoccupied elements of $\Ax$) require $O(N)$ expected time. 

As we simulate $\actM$, we will maintain the following invariants at all times $t \ge 0$.
\begin{enumerate}[({I}1)]
	\item $\Nm = |\actM(t)|$.
	\item $\Pot = \phi(\actM(t))\cdot \lcm$.
	\item Each $v\in [n]$ is a key in $\Hm$ if and only if $v \in \actM(t)$.
	\item Each $v\in [n]$ is a key in $\Hx$ if and only if $v \in \Xx(t)$.
	\item If $\Hx$ maps $v\in [n]$ to $i$, then $\Ax[i]$ contains the pair $(v,\dbd(v,t))$. Conversely, if $\Ax[i]$ is occupied with some value $(v,d)$, then $\Hx$ maps $v$ to $i$ and $d = \dbd(v,t)$. 
\end{enumerate}

\textbf{Initialisation: } We first choose $v_0 \in [n]$ uniformly at random, taking $\actM(0) = \{v_0\}$, and initialise our data structure accordingly. Set $\Nm$ to $1$, so (I1) holds. Set $\Pot$ to $\lcm/d_{v_0}$ (which is an integer), so (I2) holds. Add $v_0$ to $\Hm$, so (I3) holds. To guarantee (I4) and (I5), first insert $(v_0,d_{v_0})$ into $\Ax$ and add a corresponding entry to $\Hx$, then for each $w \in L_{v_0}$, insert $(w,1)$ into $\Ax$ and add a corresponding entry to $\Hx$. Note that since $\log \lcm = O(\Delta)$, we can compute $\lcm/d_{v_0}$ in $O(\Delta)$ arithmetic operations, and so we have used $O(\Delta)$ operations in total. 

\textbf{Simulating a step: } Recall that $\tau = \min\{t \ge 0 \mid \actM(t) = \emptyset \textnormal{ or }\phi(\actM(t)) \ge P\}$, $P' = \ceil{\log_r(6N)}$, and $P = \min\{P',\phi(V)\}$. Suppose that $\Nm$, $\Pot$, $\Hm$, $\Hx$ and $\Ax$ satisfy (I1)--(I5) at some time $0 \le t \le \tau$. We have $t = \tau$ if and only if $\Nm \in \{0,n\}$ or $\Pot \ge P'\lcm$, which we can test in $O(\Delta)$ time. If $t = \tau$ then we are done, so suppose $t < \tau$. Then we must simulate a step of $\actM$ to sample the random variable $\actM(t+1)$ and update our data structures accordingly. Proceed as follows.

\begin{itemize}
	\item {\bf Decide from which vertex to spawn:}
	\begin{enumerate}[(i)]
		\item Choose an occupied entry $(v,d)$ of $\Ax$ uniformly at random, so that $v$ is a uniformly random vertex in $\Xx(t)$ and $d = \dbd(v,t)$.
		\item If $v$ is a key in $\Hm$, so that $v \in \actM(t)$, then with probability $1 - d/d_v$ discard $v$ and return to (i).
		\item If $v$ is not a key in $\Hm$, so that $v \notin \actM(t)$, then with probability $1 - d/rd_v$ discard $v$ and return to (i).
		\item Choose to spawn from $v$.
	\end{enumerate}
	Note that for all $w \in \Xx(t)\cap \actM(t)$, we have
	$$
	\Pr(v = w) = \frac{\frac{\dbd(w)}{d_w}}{\sum_{x \in \Xx(t)\cap
	\actM(t)}\frac{\dbd(x)}{d_x} + \sum_{x \in \Xx[t] \setminus
	\actM(t)}\frac{\dbd(x)}{rd_x}}.$$
	Multiplying by~$r$ and using the fact that
	  $ \dbd(x)=0$
	if $x\notin\Xx(t)$, this is equivalent to
	$$\Pr(v = w) = \frac{\frac{r\dbd(w)}{d_w}}{\sum_{x \in
	\actM(t)}\frac{r\dbd(x)}{d_x} + \sum_{x \in [n] \setminus
	\actM(t)}\frac{\dbd(x)}{d_x}},$$
	which is what we require.
	Similarly, the probability of choosing any $w \in \Xx(t) \setminus \actM(t)$ is correct, and the probability of choosing any $w\in [n] \setminus B(t)$ is zero as required. Each iteration of (i)--(iv) takes $O(1)$ operations, and $d/d_v > d/rd_v \ge 1/r\Delta$ for all $v \in [n]$, so at most $O(\Delta)$ iterations are required in expectation.
	
	\item {\bf Decide onto which vertex to spawn:}
	If $v$ is a key in $\Hm$, then construct $\{w \in L_v \mid w \notin \Hm\}$, the list of non-mutant neighbours of $v$, and choose one uniformly at random to spawn onto. Otherwise, construct $\{w \in L_v \mid w \in \Hm\}$ and choose an element uniformly at random to spawn onto. In either case, $O(\Delta)$ operations are required.
	
	\item {\bf Update the data structures:}
	Suppose that we have decided to spawn from a mutant $v$ onto a non-mutant $w$ (spawning from a non-mutant is similar).
	We first increment $\Nm$ and add $\lcm/d_w$ to $P$; this maintains (I1) and (I2). We next add $v$ to $\Hm$, maintaining (I3). It remains to maintain (I4) and (I5). 
	
	Construct the set $Y_1 = \{x \in L_w \mid w \in \Hm\}$ of mutant neighbours of $w$ and the set $Y_2 = \{x \in L_w \mid w \notin \Hm\}$ of non-mutant neighbours of $w$. For each $x \in Y_1$, we have $x \in \Xx(t)$; look up $x$ in $\Hx$ to obtain an index $a_x$, and let $(x,d)$ be the entry in $\Ax[a_x]$. If $d > 1$, then update $\Ax[a_x]$ to $(x,d-1)$, and otherwise delete $\Ax[a_x]$ from $\Ax$ and $x$ from $\Hx$. For each $x \in Y_2$, check whether $x$ is a key in $\Hx$. If it is, with value $a_x$, update the value of $\Ax[a_x]$ from $(x,d)$ to $(x,d+1)$. Otherwise, add the entry $(x,1)$ to a uniformly random unoccupied slot $a_x$ of $\Ax$, and add $x$ to $\Hx$ as a key with entry $a_x$. Finally, look up $w$ in $\Hx$ to obtain an index $a_w$, and let $(w,d)$ be the entry in $\Ax[a_w]$. If $d=d_w$, then delete $\Ax[a_w]$ from $\Ax$ and $w$ from $\Hx$. Otherwise, update $\Ax[a_w]$ to $(w,d_w-d)$. It is clear that in total we have used $O(\Delta)$ operations.
\end{itemize}

Thus we have presented an algorithm to sample from $\actM(\tau)$ in $O(\Delta)$ time per step simulated, and $\actM(\tau) = \emptyset$ if and only if $\Nm = 0$. Corollary~\ref{cor:new-fpras} therefore follows.

\bibliographystyle{plain}
\bibliography{\jobname}

\end{document}